\documentclass[reqno]{amsart}

\pdfoutput = 1

\usepackage{amsmath,amssymb}
\usepackage{amsfonts}
\usepackage{array,dcolumn}
\usepackage{graphicx}

\usepackage[T1]{fontenc}
\usepackage[latin1]{inputenc}
\usepackage[activate={true,nocompatibility},spacing,kerning]{microtype}
\usepackage[pdftex,
            colorlinks,
            hyperfootnotes=false,
            pdffitwindow=true,
            plainpages=false,
            pdfpagelabels=true,
            pdfpagemode=UseOutlines,
            pdfpagelayout=SinglePage,
            pdftitle={Numerical Evaluation of Distributions in RMT},
            pdfauthor={Folkmar Bornemann, Technische Universität München},
            pagebackref,
            hyperindex,
            filecolor=purple,
            urlcolor=purple,
            citecolor=blue,
            linkcolor=blue]{hyperref}
\usepackage{myharvard}
\usepackage{bm}
\usepackage[sc,osf]{mathpazo}

\graphicspath{%
 {./Images/}%
}

\newtheorem{theorem}{Theorem}[section]

\theoremstyle{definition}

\theoremstyle{remark}
\newtheorem{remark}[theorem]{Remark}

\numberwithin{equation}{section}

\makeatletter
\renewcommand{\leq}{\leqslant}
\renewcommand{\geq}{\geqslant}
\renewcommand{\Re}{{\operator@font Re}}
\renewcommand{\Im}{{\operator@font Im}}
\newcommand{\tr}{{\operator@font tr}}
\newcommand{\range}{{\operator@font ran}}
\newcommand{\spn}{{\operator@font span}}
\newcommand{\sinc}{{\operator@font sinc}}
\newcommand{\erf}{{\operator@font erf}}
\newcommand{\Ai}{{\operator@font Ai}}
\newcommand{\cov}{{\operator@font cov}}
\newcommand{\var}{{\operator@font var}}
\newcommand{\sgn}{{\operator@font sign}}
\makeatother

\newcommand{\C}{{\mathbb  C}}

\newcommand{\R}{{\mathbb  R}}

\newcommand{\N}{{\mathbb  N}}
\newcommand{\prob}{{\mathbb  P}}

\newcommand{\projected}[1]{\negthickspace\upharpoonright_{#1}}

\begin{document}

\title[On the Numerical Evaluation of Distributions in Random Matrix Theory]{On the Numerical Evaluation of Distributions\\ in Random Matrix Theory: A Review}

\author{Folkmar Bornemann}
\address{Zentrum Mathematik -- M3, Technische Universität München,
         80290~München, Germany}
\email{bornemann@ma.tum.de}


\begin{abstract} In this paper we review and compare the numerical evaluation of those probability distributions
in random matrix theory that are analytically represented in terms of Painlevé transcendents or Fredholm
determinants. Concrete examples for the Gaussian and Laguerre (Wishart) $\beta$-ensembles and their various
scaling limits are discussed. We argue that the numerical approximation of Fredholm determinants is the
conceptually more simple and efficient of the two approaches, easily generalized to the computation of joint probabilities
and correlations. Having the means for extensive numerical explorations at hand, we discovered new and
surprising determinantal formulae for the $k$-th largest (or smallest) level in the edge scaling limits of the Orthogonal and Symplectic Ensembles; formulae that
in turn led to improved numerical evaluations. The paper comes with a toolbox of Matlab functions that
facilitates further mathematical experiments by the reader.
\end{abstract}

\maketitle

\section{Introduction}\label{sect:intro}

Random Matrix Theory (RMT) has found many applications, most notably in physics, multivariate statistics, electrical engineering, and finance.
As soon as there is the need for specific numbers, such as moments, quantiles, or correlations, the actual numerical evaluation of the
underlying probability distributions becomes of interest. Without additional structure there would be, in general, only one method: Monte Carlo simulation. However,
because of the universality of certain scaling limits \citeaffixed{MR2334189}{for a review see, e.g.,}, a family of distinguished distribution functions enters which
is derived from highly structured matrix models enjoying closed analytic solutions. These functions constitute a new class of
special functions comparable in import to the classic distributions of probability theory. This paper addresses the accurate numerical evaluation\footnote{Limiting
the means, as customary in numerical analysis for reasons of efficiency
and strict adherence to numerical stability, to IEEE double precision hardware arithmetic (about 16 digits precision).} of many of these functions
on the one hand and shows, on the other hand, that such work facilitates numerical explorations that may lead, in the sense of Experimental Mathematics \cite{MR2033012},
to new theoretical discoveries, see the results of Section~\ref{sect:edge}.

\subsection{The Common Point of View}

The closed analytic solutions alluded to above are based
(for deeper reasons or because of contingency)
on two concurrent tools: Fredholm determinants of
integral operators and Painlevé transcendents. Concerning the question which of them is better suited
to be attacked numerically,
there has been a prevailing point of view for the last 15 years or so, explicitly formulated by
\citeasnoun[Footnote~10]{MR1791893}: ``Without the
Painlevé representation, the numerical evaluation of the Fredholm determinants is quite involved.'' To understand the possible genesis of this point of view
let us recall the results
for the two most important scaling limits of the Gaussian Unitary Ensemble~(GUE).

\subsubsection{Level Spacing Function of GUE}\label{sect:gaudin} The large matrix limit of GUE, scaled for level spacing~$1$ in the bulk, yields the function
\begin{equation}
E_2(0;s) = \prob(\text{no levels lie in $(0,s)$}).
\end{equation}
\citeasnoun{Gau61} showed that this function can be represented as a Fredholm determinant, namely,\footnote{We use the
same symbol $K$ to denote both, the integral operator $K\projected{X}$ acting on the Hilbert space $X$ and its
kernel function $K(x,y)$.}
\begin{equation}\label{eq:Ksin}
E_2(0;s) = \det\left(I-K_{\sin}\projected{L^2(0,s)}\right),\qquad K_{\sin}(x,y) = \sinc(\pi(x-y)).
\end{equation}
He proceeded by showing that the eigenfunctions of this selfadjoint integral operator
are the radial prolate spheroidal wave functions with certain parameters.
Using tables \cite{MR0074130} of these special functions
he was finally able to evaluate $E_2(0;s)$
numerically.\footnote{Strictly speaking \citeasnoun{Gau61} was concerned with evaluating the level spacing function $E_1(0;s)$ of GOE
that he represented as the Fredholm determinant of the even sine kernel, see (\ref{eq:E10}) below. However, the extension of Gaudin's method to $E_2(0;s)$
is fairly straightforward, see \citeasnoun{MR0143558} for a determinantal formula that is equivalent to (\ref{eq:Ksin}), namely (\ref{eq:E2sum}) with $k=0$, and
 \citeasnoun{Kahn63} for subsequent numerical work. As was pointed out by \citeasnoun[p.~305]{MR866115},
who himself had calculated $E_2(0;s)$ by Gaudin's method using Van Buren's implementation of the prolate wave functions, Kahn's tables, reproduced in the first 1967 edition of
Mehta's book, are rather inaccurate. In contrast, the tables in \citeasnoun{MR0348823}, reproduced in the second 1991 and third 2002 edition of Mehta's book, are basically accurate
\emph{with the proviso} that the arguments have to be read not as the displayed four digit numbers but rather as $s=2 t/\pi$ with $t=0.0$, $0.1$, $0.2$, $0.3$, etc.
A modern implementation of Gaudin's method for $E_2(0;s)$, using Mathematica's fairly recent ability to evaluate
the prolate wave functions, can be found in \citeasnoun[§7.1]{Bornemann1}.}
 On the other hand, in an admirably intricate analytic {tour de force} \citeasnoun{MR573370} expressed the Fredholm determinant by
\begin{equation}
E_s(0;s) = \exp\left(-\int_0^{\pi s} \frac{\sigma(x)}{x}\,dx\right)
\end{equation}
in terms of the Jimbo--Miwa--Okamoto $\sigma$-form of Painlevé V, namely
\begin{equation}
(x \sigma_{xx})^2 = 4(\sigma - x \sigma_x)(x \sigma_x - \sigma -  \sigma_x^2), \qquad \sigma(x) \simeq \frac{x}{\pi} + \frac{x^2}{\pi^2} \quad(x \to 0).
\end{equation}
\subsubsection{The Tracy--Widom distribution} The large matrix limit of GUE, scaled for the fluctuations at the soft edge (that is, the maximum eigenvalue), yields the function
\begin{equation}
F_2(s) = \prob(\text{no levels lie in $(s,\infty)$}).
\end{equation}
Implicitly known for quite some time \citeaffixed{MR0161680,MR1040213,MR1182169}{see, e.g.,}, the determinantal representation
\begin{equation}\label{eq:KAi}
F_2(s) = \det\left(I-K_{\Ai}\projected{L^2(s,\infty)}\right),\qquad K_{\Ai}(x,y) = \frac{\Ai(x)\Ai'(y)-\Ai'(x)\Ai(y)}{x-y},
\end{equation}
was spelt out by \citeasnoun{MR1236195} and by \citeasnoun{MR1215903}.
The search for an analogue to Gaudin's method remained unsuccessful since there is no solution of the corresponding eigenvalue problem known in terms of classic special functions \cite[p.~453]{MR2129906}.
It was therefore
a major breakthrough when \citename{MR1257246} \citeyear{MR1215903,MR1257246} derived their now famous representation
\begin{equation}
F_2(s) = \exp\left(-\int_s^\infty (x-s) u(x)^2\,dx\right)
\end{equation}
in terms of the Hastings--McLeod \citeyear{MR555581} solution $u(x)$ of Painlevé II, namely
\begin{equation}\label{eq:HM}
u_{xx} =2 u^3 + xu,\qquad u(x) \simeq \Ai(x)\quad(x\to\infty).
\end{equation}
Subsequent numerical evaluations were then, until the recent work of \citeasnoun{Bornemann1}, exclusively based on solving this asymptotic initial value problem.

\subsection{Challenging the Common Point of View}

In this paper we challenge the common point of view that a Painlevé representation would be, at least numerically, preferable to a Fredholm determinant formula. We do so from the following angles:
Simplicity, efficiency, accuracy, and extendibility. Let us briefly indicate the rationale for our point of view.

\begin{enumerate}
\item The numerical evaluation of Painlevé transcendents encountered in RMT is more involved as one would think at first sight. For
reasons of numerical stability one  needs additional, deep analytic knowledge, namely, asymptotic expansions of the corresponding connection formulae
(see Section~\ref{sect:painleve}).\smallskip
\item There is an extremely simple, fast, accurate, and \emph{general} numerical method for evaluating Fredholm determinants (see Section~\ref{sect:fredholm}).
\smallskip
\item Multivariate functions such as joint probability distributions often have a representation by a Fredholm determinant (see Section~\ref{sect:joint}).
On the other hand, if available at all, a representation in terms of a nonlinear \emph{partial} differential equation is of very limited numerical use right now.
\end{enumerate}

\subsection{Outline of the Paper}

In Section~\ref{sect:stage} we collect some fundamental functions of RMT whose numerical solutions will play a role in the sequel.
The intricate issues of a numerical solution of the Painlevé transcendents encountered in
RMT are subject of Section~\ref{sect:painleve}. An exposition of \possessivecite{Bornemann1} method
for the numerical evaluation of Fredholm determinants is given in Section~\ref{sect:fredholm}.
The numerical evaluation of the $k$-level spacings in the bulk of GOE, GUE, and GSE, by using Fredholm determinants, is addressed in Section~\ref{sect:bulk}.
In Section~\ref{sect:edge} we get to new determinantal formulae for the distributions of the $k$-th largest level in the soft edge scaling limit of GOE and GSE. These formulae
rely on a  determinantal identity that we found by extensive numerical experiments before proving it. By a powerful structural analogy, in Section~\ref{sect:hard} these formulae
are easily extended to the $k$-th smallest level in the hard edge scaling limit of LOE and LSE. In Section~\ref{sect:joint}, we discuss some examples of
joint probabilities, like the one for the largest two eigenvalues of GUE at the soft edge or the one of the Airy process for two different times. Finally, in
Section ~\ref{sect:software}, we give a short introduction into using the Matlab toolbox that comes with this paper.

\section{Some Distribution Functions of RMT and Their Representation}\label{sect:stage}

In this section we collect some fundamental functions of RMT whose numerical evaluation will explicitly be discussed in the sequel. We do not strive for
completeness here; but we will give sufficiently many examples to be able to judge of simplicity and generality of the numerical approaches later on.

We confine ourselves
to the Gaussian (Hermite) and Laguerre (Wishart) ensembles. That is, we
consider $n\times n$ random matrix ensembles with real (nonnegative in the case of the Laguerre ensemble) spectrum such that the joint probability distribution of its (unordered) eigenvalues is
given by
\begin{equation}
p(x_1,\ldots,x_n) = c\, \prod_i w(x_i) \cdot \prod_{i<j} |x_i-x_j|^\beta.
\end{equation}
Here,  $\beta$ takes the values $1$, $2$, or $4$ (Dyson's ``three fold way''). The weight function $w(x)$ will be either a Gaussian
or, in the case of the Laguerre ensembles, a function of the type $w_\alpha(x) = x^\alpha e^{-x}$ $(\alpha>-1)$.

\subsection{Gaussian Ensembles} Here, we take the Gaussian weight functions\footnote{We follow  \citeasnoun{MR1842786} in the choice of the variances of the Gaussian weights.
Note that \citeasnoun[Chap.~3]{MR2129906}, such as Tracy and Widom in most of their work,
uses $w(x) = e^{-\beta x^2/2}$. However, one has to be alert: from p.~175 onwards, Mehta uses $w(x) = e^{-x^2}$ for $\beta=4$ in his book, too.}

\begin{itemize}
\item $w(x) = e^{-x^2/2}$ for $\beta=1$, the Gaussian Orthogonal Ensemble (GOE),\smallskip
\item $w(x) = e^{-x^2}$ for $\beta=2$, the Gaussian Unitary Ensemble (GUE),\smallskip
\item $w(x) = e^{-x^2}$ for $\beta=4$, the Gaussian Symplectic Ensemble (GSE).
\end{itemize}

\noindent
We define, for an open interval $J \subset \R$, the basic quantity
\begin{multline}
E_\beta^{(n)}(k;J)\\*[2mm] = \prob\left(\text{exactly $k$ eigenvalues of the $n\times n$ Gaussian $\beta$-ensemble lie in $J$}\right).
\end{multline}
More general quantities will be considered in Section~\ref{sect:general}.

\subsubsection{Scaling limits}
The \emph{bulk scaling} limit is given by \citeaffixed[§§6.3, 7.2, and 11.7]{MR2129906}{see}
\begin{equation}
E_\beta^{(\text{bulk})}(k;J) =
\begin{cases}
\displaystyle\lim_{n\to \infty} E_\beta^{(n)}(k;\pi\, 2^{-1/2} n^{-1/2} J) &\quad\; \text{$\beta=1$ or $\beta=2$},\\*[3mm]
\displaystyle\lim_{n\to \infty} E_\beta^{(n)}(k;\pi\, n^{-1/2} J)  &\quad\; \text{$\beta=4$}.
\end{cases}
\end{equation}
The \emph{soft edge scaling} limit is given by \citeaffixed[p.~194]{MR1842786}{see}
\begin{equation}\label{eq:GEedge}
E_\beta^{(\text{soft})}(k;J) =
\begin{cases}
\displaystyle\lim_{n\to \infty} E_\beta^{(n)}(k;\sqrt{2n} + 2^{-1/2} n^{-1/6} J)   &\quad\;\text{$\beta=1$ or $\beta=2$},\\*[3mm]
\displaystyle\lim_{n\to \infty} E_\beta^{(n/2)}(k;\sqrt{2n} + 2^{-1/2} n^{-1/6} J)&\quad\;\text{$\beta=4$}.
\end{cases}
\end{equation}

\subsubsection{Level spacing}
The $k$-level spacing function $E_\beta(k;s)$ of \citeasnoun[§6.1.2]{MR2129906} is
\begin{equation}\label{eq:Ebeta}
E_\beta(k;s) = E_\beta^{(\text{bulk})}(k;(0,s)).
\end{equation}
The $k$-level spacing density $p_\beta(k;s)$, that is, the probability density of the distance of a level in the bulk to its $(k+1)$-st next neighbor  is \citeaffixed[Eq.~(6.1.18)]{MR2129906}{see}
\begin{equation}\label{eq:pbeta}
p_\beta(k;s) = \frac{d^2}{ds^2} \sum_{j=0}^{k} (k+1-j) E_\beta(j;s)\qquad(k=0,1,2,\ldots).
\end{equation}
Since the bulk scaling limit was made such that the expected distance between neighboring eigenvalues is one,
we have \citeaffixed[Eq.~(6.1.26)]{MR2129906}{see}
\begin{equation}\label{eq:integralconstraint}
\int_0^\infty p_\beta(k;s) = 1,\qquad  \int_0^\infty s p_\beta(k;s) \,ds = k+1.
\end{equation}
Likewise, there holds  for $s \geq 0$
\begin{equation}\label{eq:spacingconstraint}
\sum_{k=0}^\infty E_\beta(k;s) = 1,\qquad \sum_{k=0}^\infty kE_\beta(k;s) = s;
\end{equation}
see, e.g., \citeasnoun[Eqs.~(6.1.19/20)]{MR2129906} or \citeasnoun[p.~119]{MR1677884}.
Such constraints are convenient means to assess the accuracy of numerical methods (see Examples~\ref{sect:constraint}/\ref{sect:constraint2} and, for
a generalization, Example~\ref{sect:generalconstraint}).

\subsubsection{Distribution of the $k$-th largest eigenvalue}

The cumulative distribution function of the $k$-th largest eigenvalue is, in the soft edge scaling limit,
\begin{equation}\label{eq:Fbeta}
F_\beta(k;s) = \sum_{j=0}^{k-1}E_\beta^{(\text{soft})}(j;(s,\infty)).
\end{equation}
The famous Tracy--Widom \citeyear{MR1385083} distributions $F_\beta(s)$ are given by
\begin{equation}\label{eq:TWbeta}
F_\beta(s) = \begin{cases}
F_\beta(1;s) &\qquad \text{$\beta=1$ or $\beta=2$},\\*[2mm]
F_\beta(1;\sqrt{2}\,s) &\qquad \text{$\beta=4$}.
\end{cases}
\end{equation}

\subsection{Determinantal Representations for Gaussian Ensembles}
\subsubsection{GUE}

Here, the basic formula is \citeaffixed[§5.4]{MR1677884}{see}
\begin{subequations}\label{eq:detGUEn}
\begin{align}
E_2^{(n)}(k;J) &= \frac{(-1)^k}{k!}\left.\frac{d^k}{dz^k}\, D_2^{(n)}(z;J)\right|_{z=1},\label{eq:D2nDet}\\*[2mm]
D_2^{(n)}(z;J) &= \det\left(1- z\, K_n \projected{L^2(J)} \right),\label{eq:D2n}
\end{align}
\end{subequations}
with the Hermite kernel (the second form follows from Christoffel--Darboux)
\begin{equation}\label{eq:Kn}
K_n(x,y) = \sum_{k=0}^{n-1} \phi_k(x)\phi_k(y) = \sqrt{\frac n2}\,\frac{\phi_n(x)\phi_{n-1}(y) - \phi_{n-1}(x)\phi_n(y)}{x-y}
\end{equation}
that is built from the $L^2(\R)$-orthonormal system of the Hermite functions
\begin{equation}
\phi_k(x) = \frac{e^{-x^2/2} H_k(x)}{\pi^{1/4}\sqrt{k!} \,2^{k/2}}.
\end{equation}
The bulk scaling limit is given by \citeaffixed[§A.10]{MR2129906}{see}
\begin{subequations}\label{eq:E2bulk}
\begin{align}
E_2^{(\text{bulk})}(k;J) &= \frac{(-1)^k}{k!}\left.\frac{d^k}{dz^k}\, D_2^{(\text{bulk})}(z;J)\right|_{z=1},\\*[2mm]
D_2^{(\text{bulk})}(z;J) &= \det\left(1- z\, K_{\sin} \projected{L^2(J)} \right),\label{eq:D2bulk}
\end{align}
\end{subequations}
with the sine kernel $K_{\sin}$ defined in (\ref{eq:Ksin}).
The soft edge scaling limit is given by \citeaffixed[§3.1]{MR1236195}{see}
\begin{subequations}\label{eq:E2Edge}
\begin{align}
E_2^{(\text{soft})}(k;J) &= \frac{(-1)^k}{k!}\left.\frac{d^k}{dz^k}\, D_2^{(\text{soft})}(z;J)\right|_{z=1},\\*[2mm]
D_2^{(\text{soft})}(z;J) &= \det\left(1- z\, K_{\Ai} \projected{L^2(J)} \right),\label{eq:D2edge}
\end{align}
\end{subequations}
with the Airy kernel $K_{\Ai}$ defined in (\ref{eq:KAi}).

\subsubsection{GSE} Using Dyson's quaternion determinants one can find a
determinantal formula for $E^{(n)}_4(k;J)$ which involves a finite rank matrix kernel \citeaffixed[Chap.~8]{MR2129906}{see}---a formula that is amenable to the
numerical methods of Section~\ref{sect:fredholm}.
We confine ourselves to the soft edge scaling limit of this formula which yields
\cite{MR2187952}
\begin{subequations}\label{eq:GSEMatrixKernelDet}
\begin{equation}
E_4^{(\text{soft})}(k;J) = \frac{(-1)^k}{k!}\left.\frac{d^k}{dz^k}\, \sqrt{D_4(z;J)}\right|_{z=1},
\end{equation}
where the entries of the matrix kernel determinant
\begin{equation}
D_4(z;J) = \det\left(I - \frac{z}{2}
\begin{pmatrix}
S & SD \\*[1mm]
IS & S^*
\end{pmatrix}{\projected{L^2(J)\oplus L^2(J)}}\right)
\end{equation}
\end{subequations}
are given by (with the adjoint kernel $S^*(x,y)=S(y,x)$ obtained from transposition)
\begin{subequations}
\begin{align}
S(x,y) &= K_{\Ai}(x,y) - \frac{1}{2} \Ai(x) \int_y^\infty \Ai(\eta)\,d\eta,\label{eq:S}\\*[1mm]
SD(x,y) &= -\partial_y K_{\Ai}(x,y) - \frac{1}{2} \Ai(x)\Ai(y),\label{eq:SD}\\*[1mm]
IS(x,y) &= -\int_x^\infty K_{\Ai}(\xi,y)\,d\xi + \frac{1}{2} \int_x^\infty\Ai(\xi)\,d\xi\,\int_y^\infty \Ai(\eta)\,d\eta.\label{eq:IS}
\end{align}
\end{subequations}
Albeit this expression is also amenable to the numerical methods of Section~\ref{sect:fredholm}, we will discuss,
for the significant special cases
\begin{equation}
E_4^{(\text{bulk})}(k;(0,s)) \quad\text{and}\quad E_4^{(\text{soft})}(k;(s,\infty)),
\end{equation}
alternative determinantal formulae that are  far more efficient numerically, see Section~\ref{sect:bulk} and Section~\ref{sect:edge}, respectively.

\subsubsection{GOE}
There are also determinantal formulae for $E_1^{(\text{bulk})}(k;J)$ and its various scaling limits,
see \citeasnoun[Chap.~7]{MR2129906} and \citeasnoun{MR2187952}. However, these determinantal formulae are based on matrix kernels that involve a {\em discontinuous} term.
To be specific we recall the result for the soft edge scaling limit:
\begin{subequations}\label{eq:GOEMatrixKernelDet}
\begin{equation}
E_1^{(\text{soft})}(k;J) = \frac{(-1)^k}{k!}\left.\frac{d^k}{dz^k}\, \sqrt{D_1(z;J)}\right|_{z=1},
\end{equation}
where the entries of the matrix kernel determinant
\begin{equation}
D_1(z;J) = \det\left(I - z
\begin{pmatrix}
S & SD \\*[1mm]
IS_\epsilon & S^*
\end{pmatrix}\projected{X_1(J)\oplus X_2(J)}\right)
\end{equation}
\end{subequations}
are given by (with the adjoint kernel $S^*(x,y)=S(y,x)$ obtained from transposition)
\begin{subequations}
\begin{align}
S(x,y) &= K_{\Ai}(x,y) + \frac{1}{2} \Ai(x) \left(1-\int_y^\infty \Ai(\eta)\,d\eta\right),\\*[1mm]
SD(x,y) &= -\partial_y K_{\Ai}(x,y) - \frac{1}{2} \Ai(x)\Ai(y),\\*[1mm]
IS_\epsilon(x,y) &= -\epsilon(x-y) -\int_x^\infty K_{\Ai}(\xi,y)\,d\xi\notag\\*[1mm]
& \qquad + \frac{1}{2} \left( \int_y^x \Ai(\xi)\,d\xi + \int_x^\infty\Ai(\xi)\,d\xi\,\int_y^\infty \Ai(\eta)\,d\eta\right).
\end{align}
with the \emph{discontinuous} function
\begin{equation}
\epsilon(x) = \tfrac{1}{2} \sgn(x).
\end{equation}
\end{subequations}
This discontinuity poses considerable difficulties for the proper theoretical justification of the operator determinant: for appropriately chosen
weighted $L^2$ spaces $X_1(J)$ and $X_2(J)$, the matrix kernel operator is a Hilbert--Schmidt operator with a trace class diagonal and the determinant has to be understood as a Hilbert--Carleman regularized
determinant \citeaffixed[p.~2199]{MR2187952}{see}.
Moreover, it renders the unmodified numerical methods of Section~\ref{sect:fredholm} rather inefficient. Nevertheless, for
the significant special cases
\begin{equation}
E_1^{(\text{bulk})}(k;(0,s)) \quad\text{and}\quad E_1^{(\text{soft})}(k;(s,\infty))
\end{equation}
there are alternative determinantal formulae, which are amenable to an efficient numerical evaluation, see Section~\ref{sect:bulk} and Section~\ref{sect:edge}, respectively.

\subsection{Painlevé Representations for Gaussian Ensembles}

For the important family of \emph{integrable} kernels \citeaffixed{MR1730504}{see}, \citename{MR1277933} \citeyear{MR1253763,MR1215903,MR1257246,MR1277933} found a general method to represent determinants of the
form
\begin{equation}
\det\left(I-z K\projected{L^2(a,b)}\right)
\end{equation}
explicitly by a system of partial differential equations with respect to the independent variables $a$ and $b$. Fixing one of the bounds yields an ordinary differential
equation (that notwithstanding depends on the fixed bound). In RMT, this ordinary differential equation turned out, case by case, to be a Painlevé equation.
The typical choices of intervals with a fixed bound are $J=(0,s)$ or $J=(s,\infty)$, depending on whether one looks at the bulk or the soft edge of the spectrum.

\subsubsection{GUE}

\citeasnoun{MR1277933} calculated, for the determinant (\ref{eq:D2n}) with $J=(s,\infty)$, the representation
\begin{equation}
D_2^{(n)}(z;(s,\infty)) = \exp\left(-\int_s^\infty \sigma(x;z)\,dx\right)
\end{equation}
in terms of the Jimbo--Miwa--Okamoto $\sigma$-form of Painlevé IV, namely
\begin{subequations}\label{eq:PIVGE}
\begin{align}
\sigma_{xx}^2 &= 4(\sigma-x \sigma_x)^2 -4 \sigma_x^2(\sigma_x+2n),\\*[3mm]
 \sigma(x;z)  &\simeq z\, \frac{2^{n-1} x^{2n-2}}{\sqrt\pi\, \Gamma(n)} e^{-x^2} \qquad(x\to\infty).
\end{align}
\end{subequations}
As mentioned in the introduction, for the determinant (\ref{eq:D2bulk}) and $J=(0,s)$, \citeasnoun{MR573370} found  the representation \citeaffixed[Thm.~9]{MR1253763}{see also}
\begin{equation}
D_2^{(\text{bulk})}(z;(0,s)) = \exp\left(-\int_0^{\pi s} \frac{\sigma(x;z)}{x}\,dx\right)
\end{equation}
in terms of the Jimbo--Miwa--Okamoto $\sigma$-form of Painlevé V, namely
\begin{subequations}\label{eq:PVGE}
\begin{align}
(x \sigma_{xx})^2 &= 4(\sigma - x \sigma_x)(x \sigma_x - \sigma -  \sigma_x^2),\\*[1mm]
\sigma(x;z) &\simeq  \frac{z}{\pi}x + \frac{z^2}{\pi^2}x^2 \qquad(x \to 0).
\end{align}
\end{subequations}
Finally, there is \citename{MR1257246}'s \citeyear{MR1215903,MR1257246} famous representation of the determinant (\ref{eq:D2edge}) for $J=(s,\infty)$,
\begin{equation}\label{eq:D2PII}
D_2^{(\text{soft})}(z;(s,\infty)) = \exp\left(-\int_s^\infty (x-s)u(x;z)^2\,dx\right)
\end{equation}
in terms of Painlevé II,\footnote{For a better comparison with the other examples
we recall that \citeasnoun[Eq.~(1.16)]{MR1257246} also gave the respresentation
\begin{equation}\label{eq:F2PII}
D_2^{(\text{soft})}(z;(s,\infty)) = \exp\left(-\int_s^\infty \sigma(x;z)\,dx\right)
\end{equation}
in terms of the Jimbo--Miwa--Okamoto $\sigma$-form of Painlevé II:
\begin{equation}
\sigma_{xx}^2 = -4\sigma_x(\sigma-x\sigma_x) - 4 \sigma_x^3,\qquad \sigma(x;z) \simeq z\, (\Ai'(x)^2 -x \Ai(x)^2) \quad(x\to\infty).
\end{equation}
}
namely
\begin{equation}\label{eq:PIIGE}
u_{xx} = 2u^3 + xu,\qquad
u(x;z)\simeq \sqrt{z}\, \Ai(x)\quad(x\to\infty).
\end{equation}

\subsubsection{GOE and GSE in the bulk} With $\sigma(x)=\sigma(x;1)$ from (\ref{eq:PVGE}) there holds the representation \cite{MR1173848}
\begin{subequations}
\begin{align}
E_1(0;s) &= \exp\left(-\frac12\int_0^{\pi s} \sqrt{\frac{d}{dx}\frac{\sigma(x)}{x}}\,dx\right) E_2(0;s)^{1/2},\\*[2mm]
E_4(0;s/2) &= \cosh\left(\frac12\int_0^{\pi s} \sqrt{\frac{d}{dx}\frac{\sigma(x)}{x}}\,dx\right) E_2(0;s)^{1/2}.
\end{align}
\end{subequations}
Painlevé representations for $E_1(k;s)$ and $E_4(k;s)$ can be found in \citeasnoun{MR1173848}.

\subsubsection{GOE and GSE at the soft edge} \citeasnoun{MR1385083} found, with $u(x)=u(x;1)$ being the Hastings--McLeod solution of
(\ref{eq:PIIGE}), the representation
\begin{subequations}\label{eq:GOEGSEPII}
\begin{align}
F_1(1;s) &= \exp\left(-\frac{1}{2}\int_s^\infty u(x)\,dx \right) F_2(1;s)^{1/2},\\*[2mm]
F_4(1;s) &= \cosh\left(\frac{1}{2}\int_s^\infty u(x)\,dx \right) F_2(1;s)^{1/2}.\label{eq:F4PII}
\end{align}
\end{subequations}
More general, \citeasnoun{Dieng05} found  Painlevé representations of $F_1(k;s)$ and $F_4(k;s)$.

\subsection{Laguerre Ensembles} Here, we take, on $x\in(0,\infty)$ with parameter $\alpha>-1$, the weight functions\footnote{We follow \citeasnoun{MR1842786} in this particular choice of the weights; as for the Gaussian ensembles
 notations differ from reference to reference by various scaling factors.}

\begin{itemize}
\item $w_\alpha(x) = x^\alpha e^{-x/2}$ for $\beta=1$, the Laguerre Orthogonal Ensemble (LOE),\smallskip
\item $w_\alpha(x) = x^\alpha e^{-x}$ for $\beta=2$, the Laguerre Unitary Ensemble (LUE),\smallskip
\item $w_\alpha(x) = x^\alpha e^{-x}$ for $\beta=4$, the Laguerre Symplectic Ensemble (LSE).
\end{itemize}
We define, for an open interval $J \subset(0,\infty)$, the basic quantity
\begin{multline}
E_{\text{$\beta$-LE}}^{(n)}(k;J,\alpha) = \prob(\text{exactly $k$ eigenvalues of the}\\*[1mm]
 \text{$n\times n$ Laguerre $\beta$-ensemble with parameter $\alpha$ lie in $J$}).
\end{multline}

\subsubsection{Scaling limits} The large matrix limit at the {\em hard edge} is \citeaffixed[p.~2993]{MR2275509}{see}
\begin{equation}
E_\beta^{(\text{hard})}(k;J,\alpha) =
\begin{cases}
\displaystyle \lim_{n\to\infty}E_{\text{$\beta$-LE}}^{(n)}(k;4^{-1}n^{-1} J,\alpha) &\quad\; \text{$\beta=1$ or $\beta=2$},\\*[3mm]
\displaystyle \lim_{n\to\infty}E_{\text{$\beta$-LE}}^{(n/2)}(k;4^{-1}n^{-1} J,\alpha) &\quad\; \text{$\beta=4$}.
\end{cases}
\end{equation}
The large matrix limit at the {\em soft edge} gives exactly the same result (\ref{eq:GEedge}) as for the Gaussian ensembles, namely \citeaffixed[p.~2992]{MR2275509}{see}
\begin{equation}
E_\beta^{(\text{soft})}(k;J) =
\begin{cases}
\displaystyle \lim_{n\to\infty}E_{\text{$\beta$-LE}}^{(n)}(k;4n+2(2n)^{1/3} J,\alpha) &\quad\; \text{$\beta=1$ or $\beta=2$},\\*[3mm]
\displaystyle \lim_{n\to\infty}E_{\text{$\beta$-LE}}^{(n/2)}(k;4n+2(2n)^{1/3} J,\alpha) &\quad\; \text{$\beta=4$},
\end{cases}
\end{equation}
independently of $\alpha$.
Likewise, a proper bulk scaling limit yields $E_\beta^{(\text{bulk})}(k;J)$ as for the Gaussian ensembles \citeaffixed[p.~291]{MR1266485}{see}.

In the rest of this section we confine ourselves to the discussion of the LUE; determinantal formulae for LOE and LSE at the hard edge are given in Section~\ref{sect:hard}.

\subsection{Determinantal Representations for the LUE}

Here, the basic formula is \citeaffixed[§19.1]{MR2129906}{see}
\begin{subequations}\label{eq:detLUEn}
\begin{align}
E_{\text{LUE}}^{(n)}(k;J,\alpha) &= \frac{(-1)^k}{k!}\left.\frac{d^k}{dz^k}\, D_{\text{LUE}}^{(n)}(z;J,\alpha)\right|_{z=1},\\*[2mm]
D_{\text{LUE}}^{(n)}(z;J,\alpha) &= \det\left(1- z\, K_{n,\alpha} \projected{L^2(J)} \right),\label{eq:DLUEn}
\end{align}
\end{subequations}
with the Laguerre kernel (the second form follows from Christoffel--Darboux)
\begin{equation}\label{eq:Kna}
K_{n,\alpha}(x,y) = \sum_{k=0}^{n-1} \phi_k^{(\alpha)}(x) \phi_k^{(\alpha)}(y)
= -\sqrt{n(n+\alpha)}\;\frac{\phi_n^{(\alpha)}(x)\phi_{n-1}^{(\alpha)}(y) - \phi_{n-1}^{(\alpha)}(x)\phi_n^{(\alpha)}(y)}{x-y}
\end{equation}
that is built from the Laguerre polynomials $L_k^{(\alpha)}(x)$ by the $L^2(0,\infty)$-orthonormal systems of functions
\begin{equation}
\phi_k^{(\alpha)}(x) = \sqrt{\frac{k!}{\Gamma(k+\alpha+1)}} \,x^{\alpha/2} e^{-x/2} L_k^{(\alpha)}(x).
\end{equation}
The scaling limit at the hard edge is given by \cite{MR1236195}
\begin{subequations}\label{eq:E2hard}
\begin{align}
E_{2}^{(\text{hard})}(k;J,\alpha) &= \frac{(-1)^k}{k!}\left.\frac{d^k}{dz^k}\, D_2^{(\text{hard})}(z;J,\alpha)\right|_{z=1},\\*[2mm]
D_2^{(\text{hard})}(z;J,\alpha) &= \det\left(1- z\, K_{\alpha} \projected{L^2(J)} \right),\label{eq:DLUEhard}
\end{align}
with the Bessel kernel
\begin{equation}\label{eq:Ka}
K_a(x,y) = \frac{J_\alpha(\sqrt{x}) \sqrt{\smash[b]{y}}\, J_\alpha'(\sqrt{\smash[b]{y}}) - \sqrt{x}\, J_\alpha'(\sqrt{x}) J_\alpha(\sqrt{\smash[b]{y}})}{2(x-y)}.
\end{equation}
\end{subequations}

Note that for non-integer parameter $\alpha$ the Laguerre kernel $K_{n,\alpha}(x,y)$ and the Bessel kernel $K_\alpha(x,y)$ exhibit algebraic singularities at $x=0$ or $y=0$; see Section~\ref{sect:singularity}
for the bearing of this fact on the choice of numerical methods.

\subsection{Painlevé Representations for the LUE}

\citeasnoun{MR1277933} calculated, for the determinant (\ref{eq:DLUEn}) with  $J=(0,s)$, the representation
\begin{equation}
D_{\text{LUE}}^{(n)}(z;(0,s),\alpha) = \exp\left(-\int_0^s \frac{\sigma(x;z)}{x}\,dx\right)
\end{equation}
in terms of the Jimbo--Miwa--Okamoto $\sigma$-form of Painlevé V, namely
\begin{subequations}\label{eq:PVLE}
\begin{align}
(x\sigma_{xx})^2 &= (\sigma-x\sigma_x-2\sigma_x^2+(2n+\alpha)\sigma_x)^2 -4\sigma_x^2(\sigma_x-n)(\sigma_x-n-\alpha),\\*[3mm]
 \sigma(x;z)  &\simeq z\, \frac{\Gamma(n+\alpha+1)}{\Gamma(n)\Gamma(\alpha+1)\Gamma(\alpha+2)} x^{\alpha+1} \qquad(x\to 0).
\end{align}
\end{subequations}
Accordingly, \citeasnoun{MR1266485} obtained, for the determinant (\ref{eq:DLUEhard}) of the scaling limit at the hard edge,
the representation
\begin{equation}
D_2^{(\text{hard})}(z;(0,s),\alpha) = \exp\left(-\int_0^s \frac{\sigma(x;z)}{x}\,dx\right)
\end{equation}
in terms of the Jimbo--Miwa--Okamoto $\sigma$-form of Painlevé III, namely
\begin{subequations}\label{eq:PIIILE}
\begin{align}
(x\sigma_{xx})^2 &= \alpha^2 \sigma_x^2-\sigma_x(\sigma-x\sigma_x)(4\sigma_x-1),\\*[3mm]
 \sigma(x;z)  &\simeq \frac{z}{\Gamma(\alpha+1)\Gamma(\alpha+2)} \left(\frac{x}{4}\right)^{\alpha+1} \qquad(x\to 0).
\end{align}
\end{subequations}

\section{Numerics of Painlevé Equations: the Need for Connection Formulae}\label{sect:painleve}

\subsection{The Straightforward Approach: Solving the Initial Value Problem}

All the five examples (\ref{eq:PIVGE}), (\ref{eq:PVGE}), (\ref{eq:PIIGE}), (\ref{eq:PVLE}), and (\ref{eq:PIIILE})
of a Painlevé representation given in Section~\ref{sect:stage} take the form of an \emph{asymptotic} initial
value problem (IVP); that is, one looks, on a given interval $(a,b)$, for the solution $u(x)$ of a second order ordinary differential equation
\begin{equation}
u''(x) = f(x,u(x),u'(x))
\end{equation}
subject to an asymptotic ``initial'' (i.e., one sided) condition, namely \emph{either}
\begin{equation}
u(x) \simeq u_a(x) \qquad(x\to a)
\end{equation}
\emph{or}
\begin{equation}
u(x) \simeq u_b(x) \qquad(x\to b).
\end{equation}
Although we have given only the first terms of an asymptotic expansion, further terms can be obtained by symbolic calculations. Hence,
we can typically {\em choose} the order of approximation of $u_a(x)$ or $u_b(x)$ at the given ``initial'' point. Now, the straightforward approach
for a numerical solution would be to choose $a_+ > a$ or $b_- < b$ sufficiently close and compute a solution $v(x)$ of the initial
value problem
\begin{equation}
 v''(x) = f(x, v(x),v'(x))
\end{equation}
subject to proper initial conditions,  namely \emph{either}
\begin{equation}
v(a_+) = u_a(a_+),\qquad v'(a_+) = u_a'(a_+),
\end{equation}
{\em or}
\begin{equation}
v(b_-) = u_b(b_-),\qquad v'(b_-) = u_b'(b_-).
\end{equation}
However, for principal reasons that we will discuss in this section, the straightforward IVP approach unavoidably runs into instabilities.

\subsubsection{An example: the Tracy--Widom distribution $F_2$} From a numerical point of view, the Painlevé II problem (\ref{eq:PIIGE}) is certainly the
most extensively studied case. We look at the Hastings--McLeod solution $u(x) = u(x;1)$ of Painlevé II and the corresponding Tracy--Widom distribution\footnote{There is no
need for a numerical quadrature here (and in likewise cases): simply add the differential equation $(\log F_2(s))'' = -u(s)^2$
to the system of differential equations to be solved numerically; see \citeasnoun{Edelman05}. The same idea applies to the BVP approach in Section~\ref{sect:TW2}.}
\begin{equation}\label{eq:TW}
F_2(s) = \exp\left(-\int_s^\infty(x-s) u(x)^2\,dx\right).
\end{equation}
The initial value problem to be solved numerically is
\begin{equation}
v''(x) = 2v(x)^3 +x\, v(x),\qquad v(b_-) = \Ai(b_-),\quad v'(b_-) = \Ai'(b_-).
\end{equation}
Any value of $b_-\geq 8$ gives initial values that are good to machine precision (in IEEE double precision, which is about 16 significant decimal places).
We have solved the initial value problem with $b_-=12$, using a Runge--Kutta method with automatic error and step size control as coded in Matlab's {\tt ode45},
which is essentially the code published in \citeasnoun{Edelman05} and \citeasnoun{MR2168344}.
The red lines in Figure~\ref{fig:F2Error} show the absolute error $|v(x)-u(x)|$ and the corresponding error in the calculation of $F_2$. We observe that
the error of $v(x)$ grows exponentially to the left of $b_-$ and the numerical solution ceases to exist (detecting a singularity) at about $x=-5.56626$.
Though the implied values of $F_2$ are not completely inaccurate, there is a loss of more than 10 digits in absolute precision, which renders the straightforward approach  numerically
instable (that is, unreliable in fixed precision hardware arithmetic).

To nevertheless obtain a solution that is accurate to 16 digits \citeasnoun{MR2070096} turned, instead of changing the method, to variable precision software arithmetic (using up to 1500 significant digits in Mathematica) and solved the initial
value problem with $b_-=200$ and appropriately many terms of an asymptotic expansion $u_b(x)$. \citeasnoun{Praehofer03}
put tables of $u(x)$, $F_2(s)$ and related quantities to the web, for arguments from $-40$ to $200$ with a step size of $1/16$. We have used
these data as reference solutions in calculating the errors reported in Figure~\ref{fig:F2Error} and Table~\ref{tab:comparison}.

\begin{figure}[tbp]
\begin{center}
\begin{minipage}{0.49\textwidth}
\begin{center}
\includegraphics[width=\textwidth]{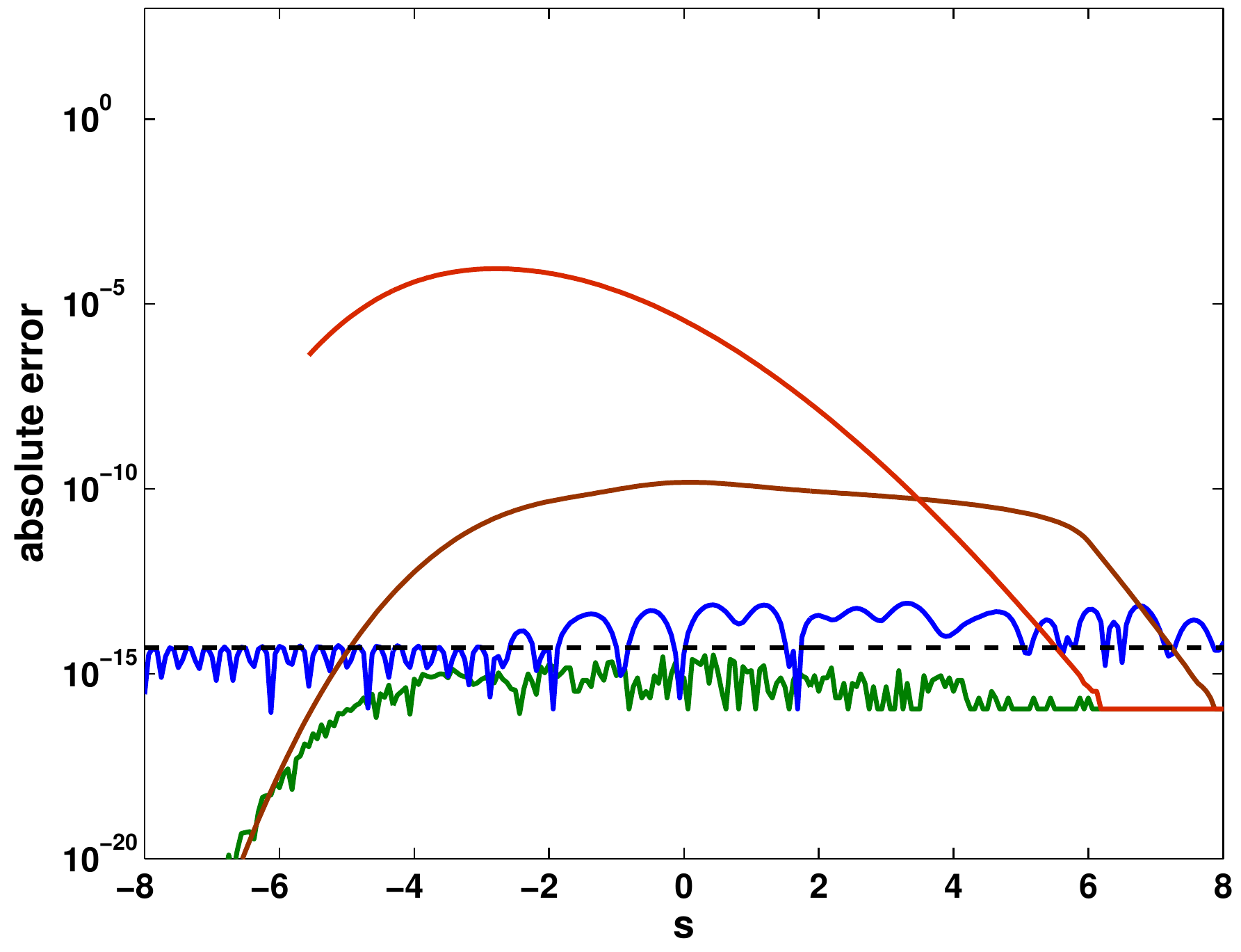}\\*[-1mm]
{\footnotesize a. \, error in evaluating $F_2(s)$}
\end{center}
\end{minipage}
\hfil
\begin{minipage}{0.49\textwidth}
\begin{center}
\includegraphics[width=\textwidth]{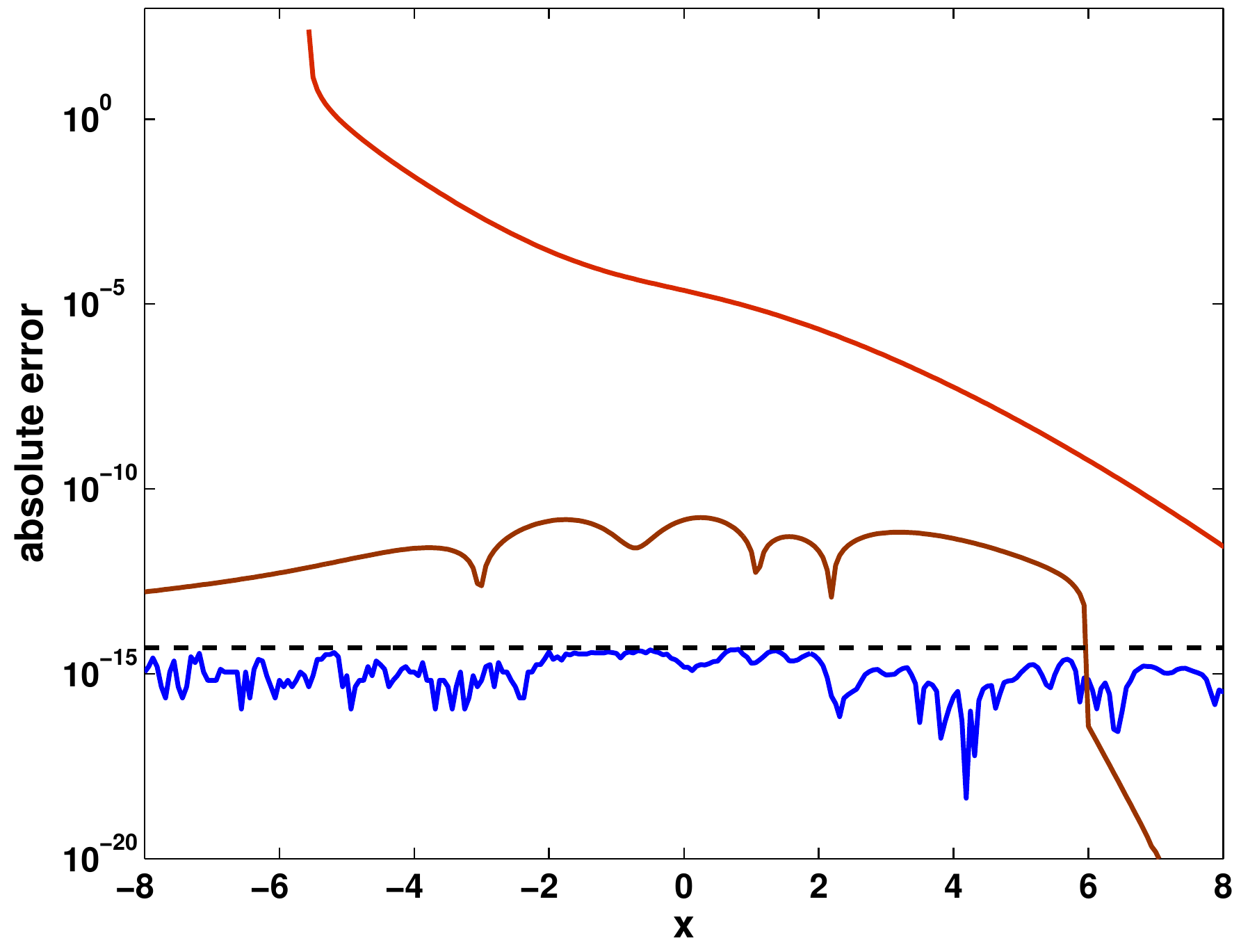}\\*[-1mm]
{\footnotesize b. \, error in evaluating $u(x)$ }
\end{center}
\end{minipage}
\end{center}\vspace*{-0.0625cm}
\caption{Absolute error in evaluating the Tracy--Widom distribution $F_2(s)$ and the Hastings--McLeod solution $u(x)$ of Painlevé~II
using different numerical methods; red: initial value solution \protect\citeaffixed{Edelman05}{Matlab's {\tt ode45} as in}, which breaks down at about $x=-5.56626$; brown:
boundary value solution \protect\citeaffixed{Dieng05}{Matlab's {\tt bvp4c} as in}, blue: boundary value solution by spectral collocation \protect\cite{MR2465699};
green: numerical evaluation of the Airy kernel Fredholm determinant \protect\cite{Bornemann1}, see also Section~\ref{sect:fredholm} (there is no $u(x)$ here). The dashed line shows the tolerance $5\cdot 10^{-15}$ used
in the error control of the last two methods. All calculations were done in IEEE double precision hardware arithmetic.}
\label{fig:F2Error}
\end{figure}

\begin{table}[tbp]
\caption{Maximum absolute error and run time of the methods in Figure~\protect\ref{fig:F2Error}. The calculation was done
for the 401 values of $F_2(s)$ from $s=-13$ to $s=12$ with step size $1/16$. The IVP solution is only available for the 282
values from $s=-5.5625$ to $s=12$. All calculations were done in hardware arithmetic.}
\vspace*{0mm}
\centerline{%
\setlength{\extrarowheight}{3pt}
\begin{tabular}{lllr}\hline
method & reference & max. error & run time\\*[0.5mm]\hline
IVP/Matlab's {\tt ode45} & \citeasnoun{Edelman05} & $9.0\cdot 10^{-5}$ & 11 sec\\
BVP/Matlab's {\tt bvp4c} & \citeasnoun{Dieng05} & $1.5\cdot 10^{-10}$ & 3.7 sec\\
BVP/spectral colloc.  & \citeasnoun{MR2465699} & $8.1\cdot 10^{-14}$ & 1.3 sec\\
Fredholm determinant     & \citeasnoun{Bornemann1} & $2.0\cdot 10^{-15}$ & 0.69 sec\\*[0.5mm]\hline
\end{tabular}}
\label{tab:comparison}
\end{table}

\subsubsection{Explaining the instability of the IVP}\label{sect:instability}

By reversibility of the differential equation, finite precision effects in evaluating the initial values at $x=b_-$ can be pulled back to a perturbation of the asymptotic condition
$u(x) \simeq \Ai(x)$ for $x\to \infty$.
That is, even an \emph{exact} integration of the ordinary differential equation would have to suffer from the result of this perturbation. Let us look at the
specific perturbation
\begin{equation}
u(x;\theta) \simeq \theta \cdot \Ai(x)\quad (x\to \infty)\qquad\text{with}\quad \theta = 1+\epsilon.
\end{equation}
The results are shown, for $\epsilon = \pm 10^{-8}$ and $\epsilon=\pm 10^{-16}$ (which is already below the resolution of hardware arithmetic), in Figure~\ref{fig:theta}
\citeaffixed[Figs.~11/12]{MR2243533}{see also}. Therefore, in hardware arithmetic, an error of order one in computing the Hastings--McLeod solution $u(x)=u(x;1)$ from the IVP is \emph{unavoidable}
already somewhere before $x\approx -12$.

\begin{figure}[tbp]
\begin{center}
\begin{minipage}{0.49\textwidth}
\begin{center}
\includegraphics[width=\textwidth]{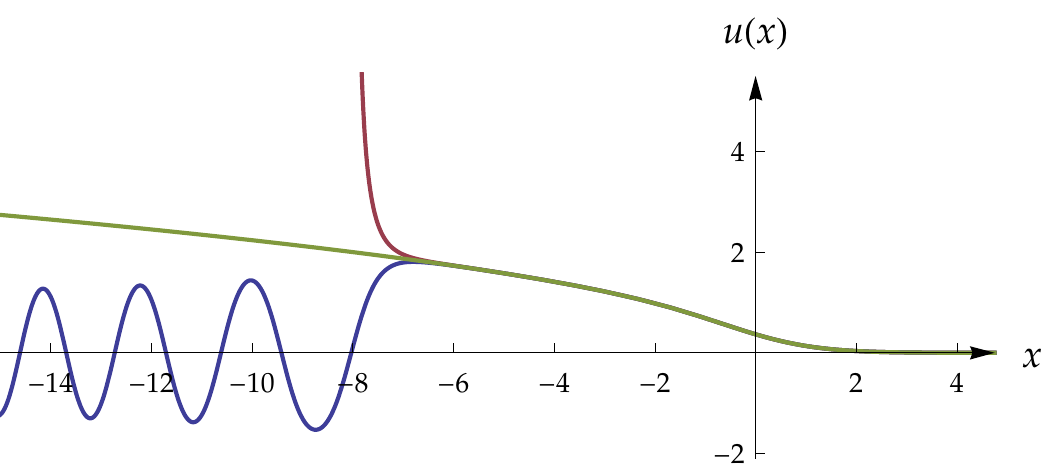}\\*[-1mm]
{\footnotesize a. \, $\epsilon=-10^{-8}$ (blue), $0$ (green), $10^{-8}$ (red)}
\end{center}
\end{minipage}
\hfil
\begin{minipage}{0.49\textwidth}
\begin{center}
\includegraphics[width=\textwidth]{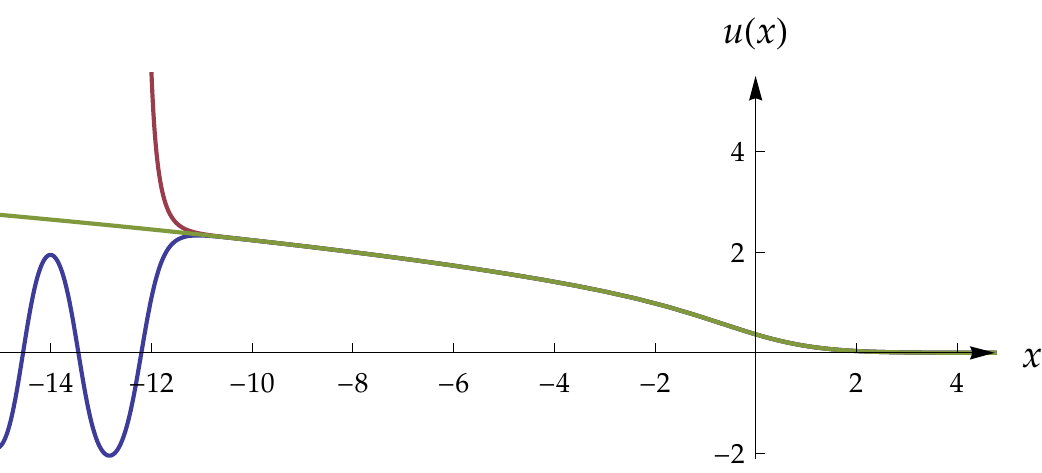}\\*[-1mm]
{\footnotesize b. \, $\epsilon=-10^{-16}$ (blue), $0$ (green), $10^{-16}$ (red)}
\end{center}
\end{minipage}
\end{center}\vspace*{-0.0625cm}
\caption{Sensitivity of Painlevé II with the asymptotic condition $u(x) \simeq (1+\epsilon)\, \Ai(x)$ ($x\to \infty$) for $\epsilon \approx 0$.
The calculation was done with variable precision software arithmetic.
Observe the dependence of the asymptotic behavior, for $x \to -\infty$, on the sign of $\epsilon$.}
\label{fig:theta}
\end{figure}

This sensitive behavior can be fully explained by the \emph{connection formulae} of Painlevé II on the real axis, see \citeasnoun[Thms.~9.1/2]{MR2243533} and \citeasnoun[Thms.~10.1/2]{MR2264522}. Namely, for the Painlevé II
equation (\ref{eq:PIIGE}), the given asymptotic behavior $u(x;\theta) \simeq \theta\cdot \Ai(x)$, $\theta >0$, as $x\to\infty$ implies
explicitly known asymptotic behavior ``in the direction of $x\to-\infty$'':
\begin{enumerate}
\item $0<\theta<1$:
\begin{multline}
\qquad\qquad u(x;\theta)
= d(\theta) |x|^{-1/4} \sin\left(\tfrac23 |x|^{3/2} -\tfrac34 d(\theta)^2 \log|x| - \phi(\theta)\right)\\ + O(|x|^{-7/10}) \qquad (x\to-\infty),
\end{multline}
with
\begin{subequations}
\begin{align}
d(\theta)^2 &= -\pi^{-1} \log(1-\theta^2) \\*[1mm]
\phi(\theta) &= \tfrac{3}{2} d(\theta)^2 \log(2) + \arg\Gamma(1-\tfrac{i}{2}d(\theta)^2)-\tfrac{\pi}{4};
\end{align}
\end{subequations}
\item $\theta = 1$:
\begin{equation}\label{eq:PIIconnect}
u(x;1) = \sqrt{-\frac{x}{2}} + O(x^{-5/2}) \qquad (x\to-\infty);
\end{equation}
\item $\theta > 1$: there is a pole at a finite $x_0(\theta) \in \R$, such that
\begin{equation}
u(x;\theta) \simeq \frac{1}{x-x_0(\theta)} \qquad (x\downarrow x_0(\theta)).
\end{equation}
\end{enumerate}
We observe that, for $x\to -\infty$, the asymptotic behavior of the Hastings--McLeod solution $u(x;1)$ \emph{separates} two completely different
regimes: an oscillatory ($\theta<1$) from a blow-up solution ($\theta>1$). The blow-up points $x_0(\theta)$ are close to the range of values of $x$ which are of interest
in the application to RMT, see Table~\ref{tab:blow-up}.

\subsubsection{Explaining the separation of asymptotic regimes}\label{sect:pole}

The deeper reason for this separation property comes from comparing (\ref{eq:D2PII}) with (\ref{eq:D2edge}), that is, from the equality
\begin{equation}\label{eq:dettheta}
\det(I- \theta^2 K_{\Ai}\projected{L^2(s,\infty)}) = \exp\left(-\int_s^\infty (x-s)u(x;\theta)^2\,dx\right).
\end{equation}
Here, $K_{\Ai}\projected{L^2(s,\infty)}$ is a positive self-adjoint trace class operator with spectral radius \citeaffixed{MR1257246}{see}
\begin{equation}
\rho(s) = \lambda_{\max}\left(K_{\Ai}\projected{L^2(s,\infty)}\right) <1.
\end{equation}
Obviously, there holds $\rho(s)\to 0$ for $s\to\infty$.
On the other hand, since for $\theta=1$ the determinant (\ref{eq:dettheta}) becomes the Tracy--Widom distribution $F_2(s)$ with $F_2(s)\to 0$ as $s\to -\infty$, we conclude that
$\rho(s)\to 1$ as $s\to -\infty$.

Now, we observe that the determinant (\ref{eq:dettheta}) becomes zero if and only if $u(x; \theta)$ blows up at the point $x=s$. By the Painlevé property,
such a singularity must be a pole.
By Lidskii's theorem \cite[Thm.~3.7]{MR2154153}, the determinant (\ref{eq:dettheta}) becomes zero if and only if $\theta^{-2}$ is an eigenvalue of
$K_{\Ai}\projected{L^2(s,\infty)}$. Therefore, a blow-up point of $u(x;\theta)$  at $x=s$ implies necessarily that
\begin{equation}
\theta \geq \rho(s)^{-1/2} > 1.
\end{equation}
On the other hand, if $\theta >1$ there must be, by continuity, a largest point $s=x_0(\theta)$ such that $\theta =  \rho(s)^{-1/2}$, which gives us the position of the pole of the
connection formula. This way, using the methodology of Section~\ref{sect:fredholm}, we have
computed the numbers shown in Table~\ref{tab:blow-up}.

\begin{table}[tbp]
\caption{Blow-up points $x_0(1+\epsilon)$ of $u(x;1+\epsilon)$ with $\epsilon>0$.}
\vspace*{0mm}
\centerline{%
\setlength{\extrarowheight}{3pt}
\begin{tabular}{lr}\hline
\;\;\;$\epsilon$ & $x_0(1+\epsilon)$\qquad\qquad\qquad \\*[0.5mm]\hline
$10^{-4}$  & $ -5.40049\,30292\,23929\cdots$ \\
$10^{-8}$  & $ -8.01133\,67804\,74602\cdots$ \\
$10^{-12}$ & $ -10.2158\,50522\,53541\cdots$\\
$10^{-16}$ & $ -12.1916\,56643\,75788\cdots$\\*[0.5mm]\hline
\end{tabular}}
\label{tab:blow-up}
\end{table}

A similar line of arguments shows that in the other cases (\ref{eq:PIVGE}), (\ref{eq:PVGE}), (\ref{eq:PVLE}), and (\ref{eq:PIIILE})
of a Painlevé representation given in Section~\ref{sect:stage}, the parameter $z=1$ (which is the most significant choice for an application to RMT)
also belongs to a connecting orbit $\sigma(x;1)$ that separates different
asymptotic regimes. In particular, we get poles at finite positions if and only if $z>1$. Hence, the numerical difficulties observed with the initial
value approach have to be expected in general.

\subsection{The Stable Approach: Solving the Boundary Value Problem}

The stable numerical solution of a connecting orbit separating different asymptotic regimes has to be addressed as a two-point \emph{boundary} value problem (BVP), see, e.g., \citeasnoun[Chap.~8]{MR1912409}. That is, we use the information from
a connection formula to infer the asymptotic for $x\to b$ from that of $x\to a$, or vice versa, and to approximate $u(x)$ by solving the BVP
\begin{equation}\label{eq:bvp}
v''(x) = f(x,v(x),v'(x)),\qquad v(a_+)=u_a(a_+),\quad v(b_-)=u_b(b_-).
\end{equation}
Thus, four particular choices have to be made: The values of the finite boundary points $a_+$ and $b_-$, and the truncation indices of the asymptotic expansions at $x\to a$ and $x\to b$ that
give the boundary functions $u_a(x)$ and $u_b(x)$. All this has to be balanced for the accuracy and efficiency of the final method.\footnote{An early variant of this connection-formula based approach can be traced back to the work of
\citeasnoun[App.~A]{McCoy76}: there, a Painlevé III representation of the spin-spin correlation function of the two-dimensional Ising model was evaluated by joining a forward integration of the IVP from $a_+$ to some
interior point $c \in (a_+,b_-)$ with a backward integration of the IVP from $b_-$ to $c$. The difference of the two IVP solutions at $c$ was used as a rough error estimate. Though not quite a BVP solution, it is close in spirit.
Actually, this was the  approach originally used by \citename{MR1257246} \citeyear{MR1215903,MR1257246} in their numerical evaluation of $F_2$ (personal communication by Craig Tracy).}

\subsubsection{An Example: the Tracy--Widom distribution $F_2(x)$}\label{sect:TW2}
Let us look, once more, at the Hastings--McLeod solution $u(x) = u(x;1)$ of (\ref{eq:PIIGE}) and the corresponding Tracy--Widom distribution (\ref{eq:TW}).
By definition, we have
\begin{equation}
u(x) \simeq \Ai(x) \qquad (x\to \infty).
\end{equation}
The asymptotic result for $x\to-\infty$ as given in the connection formula (\ref{eq:PIIconnect}) is not accurate enough to allow a sufficiently large point $a_+$ to be used.
However, using symbolic calculations it is straightforward to obtain, from this seed, the asymptotic expansion \cite[Eq.~(4.1)]{MR1257246}
\begin{multline}\label{eq:HMasympt}
u(x) = \sqrt{-\frac{x}{2}} \left(1+\frac{1}{8}x^{-3} - \frac{73}{128}x^{-6} + \frac{10657}{1024}x^{-9} -\frac{13912277}{32768}x^{-12} + O(x^{-15})\right)\\
 (x\to-\infty).
\end{multline}
\citeasnoun{Dieng05} chooses these terms as $u_a(x)$, as well as $a_+=-10$, $b_-=6$ and $u_b(x)=\Ai(x)$. Using Matlab's fixed-order collocation method
{\tt bvp4c}, he calculates  solutions whose errors are assessed in Figure~\ref{fig:F2Error} and Table~\ref{tab:comparison}. The accuracy is still somewhat limited
and he reports (p.~88) on difficulties in obtaining a starting iterate for the underlying nonlinear solver. A more promising and efficient approach to obtain near
machine precision is the use of spectral collocation methods. Because of analyticity, the convergence will be exponentially fast. This can be most elegantly expressed in
the newly developed {\tt chebop} system of \citeasnoun{MR2465699}, a Matlab extension for the automatic solution of differential
equations by spectral collocation. In fact, the evaluation of the Tracy--Widom distribution is Example~6.2 in that paper. Here,
 the first four terms of (\ref{eq:HMasympt}) are chosen as $u_a(x)$, as well as $a_+=-30$, $b_-=8$ and $u_b(x)=\Ai(x)$. The Newton iteration is started from a simple
affine function satisfying the boundary conditions; see Figure~\ref{fig:F2Error} and Table~\ref{tab:comparison} for a comparison of the accuracy and run time.

\subsection{A List of Connection Formulae}
For the sake of completeness we collect the connection formulae for the other Painlevé representations (\ref{eq:PIVGE}), (\ref{eq:PVGE}), (\ref{eq:PVLE}), and (\ref{eq:PIIILE}).
References are given to the place where we have found each formula; we did not try to locate the historically first source, whatsoever. Note that a rigorous derivation of a
connection formula relies on  deep and involved analytic arguments and calculations; a systematic approach is based on Riemann--Hilbert problems, see \citeasnoun{MR1677884}
and \citeasnoun{MR2264522} for worked out examples.

\begin{itemize}
\item The Painlevé III representation (\ref{eq:PIIILE}), for LUE with parameter $\alpha$ at the hard edge, satisfies \cite[Eq.~(3.1)]{MR1266485}
\begin{equation}
\sigma(x;1) = \frac{x}{4} - \frac{\alpha}{2}\sqrt{x} + O(1)\qquad (x\to\infty).
\end{equation}
\item The Painlevé IV representation (\ref{eq:PIVGE}), for $n$-dimensional GUE, satisfies \cite[Eq.~(5.17)]{MR1277933}
\begin{equation}
\sigma(x;1) = -2n x -n x^{-1} + O(x^{-3}) \qquad (x\to-\infty).
\end{equation}
It is mentioned there that $\sigma(x;z)$ has, for $z>1$, poles at finite positions. This is consistent with the line of arguments that we gave in Section~\ref{sect:pole}.
\item The Painlevé V representation (\ref{eq:PVGE}), for the bulk scaling limit of GUE, satisfies \cite[p.~6]{MR1173848}
\begin{equation}
\sigma(x;z) \simeq
\begin{cases}
\dfrac{x^2}{4} \quad &\quad z=1 \\*[2mm]
-\log(1-z)\dfrac{x}{\pi} \quad &\quad 0<z<1
\end{cases}
\qquad (x\to\infty)
\end{equation}
\item The Painlevé V representation (\ref{eq:PVLE}), for $n$-dimensional LUE with parameter $\alpha$, satisfies \cite[Eq.~(1.42)]{MR1885665}
\begin{equation}
\sigma(x;1) = n x- \alpha n + \alpha n^2 x^{-1} + O(x^{-2})\qquad (x\to\infty).
\end{equation}
\end{itemize}

\subsection{Summary} Let us summarize the  steps that are necessary for the numerical evaluation of a distribution function from RMT given by a Painlevé representation
on the interval $(a,b)$ (that is, the second order differential equation is given together with an asymptotic expansion of its solution at just \emph{one} of the endpoints $a$ or $b$):
\begin{enumerate}
\item Derive (or locate) the corresponding connection formula that gives the asymptotic expansion at the other end point. This requires
considerable analytic skills or, at least, a broad knowledge of the literature.
\item Choose $a_+>a$ and $b_-<b$ together with indices of truncation of the asymptotic expansions such that the expansions themselves are sufficiently accurate in $(a,a_+)$ and $(b_-,b)$ and
the two-point boundary value problem (\ref{eq:bvp}) can be solved efficiently. This balancing of parameters requires a considerable amount of experimentation to be successful.
\item The issues of solving the boundary value problem (\ref{eq:bvp}) have to be addressed: starting values for the Newton iteration, the discretization of the differential equation,
automatic step size control etc. This requires a considerable amount of experience in numerical analysis.
\end{enumerate}
Thus, much work has still to be done to make all this a ``black-box'' approach.

\section{Numerics of Fredholm Determinants and Their Derivatives}\label{sect:fredholm}

\subsection{The Basic Method}\label{sect:basic}

\citeasnoun{Bornemann1} has recently shown that there is an extremely simple, accurate, and general direct numerical method for evaluating Fredholm determinants.
By taking an $m$-point quadrature rule\footnote{We choose Clenshaw--Curtis quadrature, with a
suitable meromorphic transformation for (semi) infinite intervals \cite[Eq.~(7.5)]{Bornemann1}. For the use of Gauss--Jacobi quadrature see Section~\ref{sect:singularity}.}
of order\footnote{A quadrature rule is of order $m$ if it is \emph{exact} for polynomials of degree $m-1$.} $m$ with nodes $x_j\in(a,b)$ and {\em positive} weights $w_j$, written in the form
\begin{equation}\label{eq:quadrature}
\sum_{j=1}^m w_j f(x_j) \approx \int_a^b f(x)\,dx,
\end{equation}
the Fredholm determinant
\begin{equation}
d(z) = \det(I-z K\projected{L^2(a,b)})
\end{equation}
is simply approximated by the corresponding $m$-dimensional determinant
\begin{equation}\label{eq:detfin}
d_m(z) = \det\left(\delta_{ij} - z\, w_i^{1/2} K(x_i,x_j)w_j^{1/2} \right)_{i,j=1}^m.
\end{equation}
This algorithm can straightforwardly be implemented in a few lines. It just needs
to call  the kernel $K(x,y)$ for evaluation and has only one method parameter, the approximation dimension $m$.

If the kernel function $K(x, y)$ is analytic in a complex neighborhood of $(a,b)$, one can prove exponential convergence \cite[Thm.~6.2]{Bornemann1}:
there is a constant $\rho>1$ (depending on the domain of analyticity of $K$) such that
\begin{equation}\label{eq:expconvdet}
d_m(z)-d(z) = O(\rho^{-m})\qquad (m\to\infty),
\end{equation}
locally uniform in $z \in \C$. (Note that $d(z)$ is an {\em entire} function and $d_m(z)$ a polynomial.)
This means, in practice, that doubling $m$ will double the number of correct digits; machine precision of about 16 digits is
then typically obtained for a relatively small dimension $m$ between 10 and 100. This way the evaluation of the
Tracy--Widom distribution $F_2(s)$, at a given argument $s$, takes just a few milliseconds; see Figure~\ref{fig:F2Error} and Table~\ref{tab:comparison}
for a comparison of the accuracy and run time with the evaluation of the Painlevé representation.

\subsection{Numerical Evaluation of Finite-Dimensional Determinants}\label{sect:finitedet}
Let us write
\begin{equation}
d_m(z) = \det(I-z A_m),\qquad A_m \in \R^{m\times m},
\end{equation}
for the finite-dimensional determinant (\ref{eq:detfin}). Depending on whether we need its value for just one $z$ (typically $z=1$ in the context of RMT)
or for several values of $z$ (such as for the calculation of derivatives), we actually proceed as follows:
\begin{enumerate}
\item The \emph{value} $d_m(z)$ at a given point $z\in \C$ is calculated from the $LU$ decomposition of the matrix $I-zA_m$ (with partial pivoting). Modulo the proper sign (obtained from
the pivoting sequence), the value is given by the product of the diagonal entries of $U$ \cite[p.~176]{MR1653546}.
The computational cost is of  order $O(m^3)$, including the cost for obtaining the weights and nodes of
the quadrature method, see \citeasnoun[Footnote~5]{Bornemann1}.
\smallskip
\item The polynomial function $d_m(z)$ itself is represented by
\begin{equation}
d_m(z) = \prod_{j=1}^m (1-z\lambda_j(A_m)).
\end{equation}
Here, we first calculate the eigenvalues $\lambda_j(A_m)$ (which is slightly more expensive than the $LU$ decomposition, although the computational cost of, e.g., the $QR$ algorithm is of order $O(m^3)$, too).
The subsequent evaluation
of $d_m(z)$ costs just $O(m)$ operations for each point~$z$ that we care to address.
\end{enumerate}

\subsection{Numerical Evaluation of Higher Derivatives}\label{sect:highder}

The numerical evaluation of expressions such as (\ref{eq:D2nDet}) requires the computation of derivatives of the determinant $d(z)$ with respect to $z$.
We observe that, by well known results from complex analysis, these derivatives enjoy the same kind of convergence as in (\ref{eq:expconvdet}),
\begin{equation}
d_m^{(k)}(z)-d^{(k)}(z) = O(\rho^{-m})\qquad (m\to\infty),
\end{equation}
locally uniform in $z \in \C$, with an arbitrary but fixed $k=0,1,2,\ldots$

The numerical evaluation of higher derivatives is, in general, a badly conditioned problem.
However, for \emph{entire} functions $f$ such as our determinants we can make use of the Cauchy integrals\footnote{For more general analytic $f$ one would have to bound
the size of the radius $r$ to not leave the domain of analyticity. In particular, when evaluating (\ref{eq:GSEMatrixKernelDet}) we have to take care of the condition $r < \min |(1-\lambda)/\lambda|$, where
$\lambda$ runs through the eigenvalues of the matrix kernel operator.}
\begin{equation}
f^{(k)}(z) = \frac{k!}{2\pi r^k} \int_0^{2\pi} e^{-i k \theta}f(z+r e^{i\theta})\,d\theta\qquad (r>0).
\end{equation}
Since the integrand is analytic and \emph{periodic}, the simple trapezoidal rule is exponentially convergent \cite[§4.6.5]{MR760629}; that is, once again, $p$ quadrature points
give an error  $O(\rho^{-p})$ for some constant $\rho>1$.

Theoretically, all radii $r>0$ are equivalent. Numerically, one has to be very careful in choosing a proper radius $r$ \citeaffixed{D}{for a detailed study see}. The quantity of interest in controlling this choice
is the condition
number of the integral, that is, the ratio
\begin{equation}
\kappa = \left|\int_0^{2\pi} e^{-i k \theta}f(z+r e^{i\theta})\,d\theta\right| / \int_0^{2\pi} | f(z+r e^{i\theta})|\,d\theta \,.
\end{equation}
For reasons of numerical stability, we should choose $r$ such that $\kappa \approx 1$. Some experimentation has led us to the choices $r=1$ for the bulk and $r=0.1$ for the edge
scaling limits \citeaffixed{D}{but see also Example 12.3 in}.

\subsection{Error Control}\label{sect:error}
Exponentially convergent sequences allow us to control the error of approximation in a very simple fashion.
Let us consider a sequence $d_m \to d$ with the convergence estimate
\begin{equation}
d_m - d = O(\rho^{-m})
\end{equation}
for some constant $\rho > 1$. If the estimate is sharp, it implies the quadratic convergence of the contracted sequence $d_{2^q}$, namely
\begin{equation}
|d_{2^{q+1}} - d| \leq c |d_{2^q} - d|^2
\end{equation}
for some $c>0$. A simple application of the triangle inequality gives then
\begin{equation}\label{eq:errorestimate}
|d_{2^q} - d| \leq \frac{|d_{2^q}-d_{2^{q+1}}|}{1-c|d_{2^q}-d|} \simeq |d_{2^q}-d_{2^{q+1}}|\qquad (q\to\infty).
\end{equation}
Thus we take $|d_{2^q}-d_{2^{q+1}}|$ as an excellent error estimate of $|d_{2^q} - d|$ and as a quite ``conservative'' but absolutely
reliable estimate of $|d_{2^{q+1}}-d|$. Table~\ref{tab:errorcontrol} exemplifies this strategy for the calculation of the value $F_2(-2)$.

\begin{table}[tbp]
\caption{Approximation of the Airy kernel determinant $F_2(-2)=\det(I-K_{\Ai}\projected{L^2(-2,\infty)})$ by (\protect\ref{eq:detfin}),
using $m$-point Clenshaw--Curtis quadrature meromorphically transformed to the interval $(-2,\infty)$, see \protect\citeasnoun[Eq.~(7.5)]{Bornemann1}.
Observe the apparent quadratic convergence: the number of correct digits doubles from step to step. Thus, in exact arithmetic,
the value for $m=64$ would be correct to about 20 digits; here, the error saturates at the level of machine precision ($2.22\cdot 10^{-16}$): all 15 digits shown for the $m=64$ approximation
are correct.}
\vspace*{0mm}
\centerline{%
\setlength{\extrarowheight}{3pt}
\begin{tabular}{rcll}\hline
$m$ & $d_m$ &  $|d_m-F_2(-2)|$ & error estimate (\ref{eq:errorestimate})\\*[0.5mm]\hline
$8$ &  $0.38643\,72955\,15158$ & $2.67868\cdot 10^{-2}$ & \quad $2.67817\cdot 10^{-2}$\\
$16$ & $0.41321\,90011\,46910$ & $5.14136\cdot 10^{-6}$ & \quad$5.14138\cdot 10^{-6}$\\
$32$ & $0.41322\,41425\,27728$ & $2.26050\cdot 10^{-11}$ &\quad$2.26046\cdot 10^{-11}$\\
$64$ & $0.41322\,41425\,05123$ & $4.44089\cdot 10^{-16}$ & \quad\quad\quad--- \\*[0.5mm]\hline
\end{tabular}}
\label{tab:errorcontrol}
\end{table}

\subsection{Numerical Evaluation of Densities}\label{sect:dens}

The numerical evaluation of the probability densities belonging to the cumulative distribution functions $F(s)$ given by a determinantal expression
requires a low order differentiation with respect to the {\em real-valued} variable $s$ (which cannot easily be extended {\em numerically}
into the complex domain). Nevertheless these functions are typically real-analytic and therefore amenable to an excellent approximation by interpolation
in Chebyshev points. To be specific, if $F(s)$ is given on the finite interval $[a,b]$ and $s_0,\ldots,s_m$ denote the Chebyshev points of that interval, the polynomial interpolant $p_m(s)$ of degree $m$ is given by \possessivecite{MR0315865} barycentric formula
\begin{equation}\label{eq:bary}
p_m(s) = \frac{{\displaystyle\sum\nolimits}_{k=0}^{\prime\prime m} (-1)^k F(s_k)/(s-s_k)}%
{{\displaystyle\sum\nolimits}_{k=0}^{\prime\prime m}  (-1)^k/(s-s_k)},
\end{equation}
where the double primes denote trapezoidal sums, i.e., the first and last term of the sums get a weight $1/2$. This formula enjoys
perfect numerical stability \cite{MR2094569}. If $F$ is real analytic, we have exponential convergence once more, that is
\begin{equation}
\|F - p_m\|_\infty = O(\rho^{-m})\qquad(m\to \infty)
\end{equation}
for some constant $\rho>1$ \citeaffixed{MR2115059}{see}. Low order derivatives (such as densities) and integrals (such as moments) can easily be calculated from
this interpolant. All that is most conveniently implemented in \possessivecite{MR2087334} {\tt chebfun} package for Matlab \citeaffixed{MR2465699}{see also}.

\subsection{Examples}

We illustrate the method with three examples. More about the software that we have written can be found in Section~\ref{sect:software}.

\subsubsection{Distribution of smallest and largest level in a specific LUE}
We evaluate the cumulative distribution functions (CDF) and probability density functions (PDF) of the smallest and largest eigenvalue of
the $n$-dimensional LUE with parameter $\alpha$ for the specific choices $n=80$ and $\alpha=40$.
(Note that for parameters of this size the numerical evaluation of the Painlevé representation (\ref{eq:PVLE}) becomes extremely challenging.) Specifically, we
evaluate the CDFs (the PDFs are their derivatives)
\begin{equation}\label{eq:G}
\prob(\lambda_{\min} \leq s) = 1 - E^{(n)}_{\text{LUE}}(0;(0,s),\alpha)
\end{equation}
and
\begin{multline}\label{eq:F}
\prob(\lambda_{\max} \leq 4n+2\alpha+2+2(2n)^{1/3}s) \\*[1mm]
= E^{(n)}_{\text{LUE}}(0;(4n+2\alpha+2+2(2n)^{1/3}s,\infty),\alpha).
\end{multline}
Additionally, we calculate the CDFs of the scaling limits; that is,
\begin{equation}\label{eq:tildeG}
1 - E^{(\text{hard})}_2(0;(0,4n s),\alpha)
\end{equation}
at the hard edge and the Tracy--Widom distribution
\begin{equation}\label{eq:F2new}
F_2(s) = E^{(\text{soft})}_2(0;(s,\infty))
\end{equation}
at the soft edge. All that has to be done to apply our method is to simply code the Laguerre, Bessel, and Airy kernels. Figure~\ref{fig:LUE} visualizes the functions
and Table~\ref{tab:LUE} shows their moments to  10 correct decimal places.

\begin{figure}[tbp]
\begin{center}
\begin{minipage}{0.49\textwidth}
\begin{center}
\includegraphics[width=\textwidth]{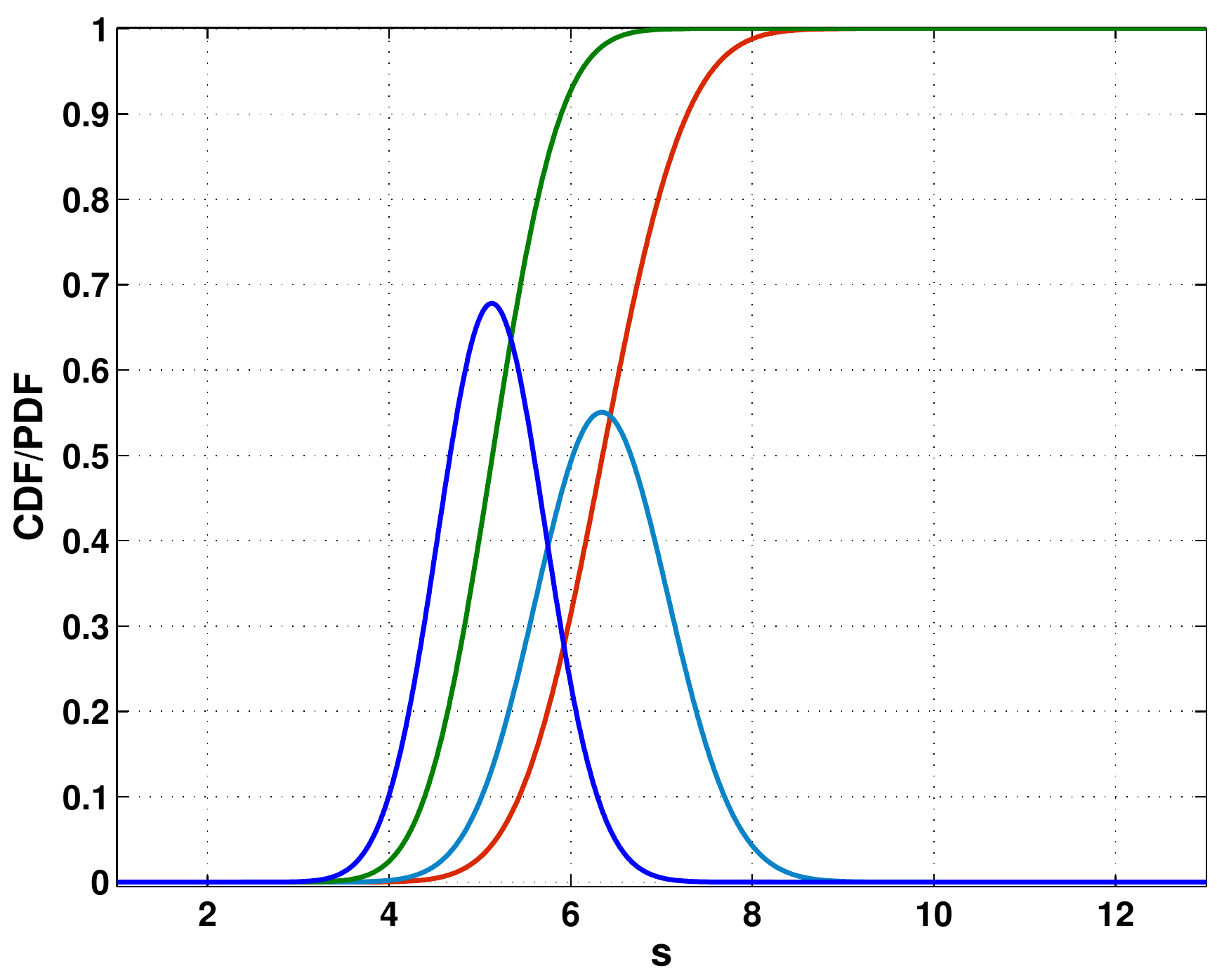}\\*[-1mm]
{\footnotesize a. \, smallest eigenvalue of LUE ($n=80$, $\alpha=40$)}
\end{center}
\end{minipage}
\hfil
\begin{minipage}{0.49\textwidth}
\begin{center}
\includegraphics[width=\textwidth]{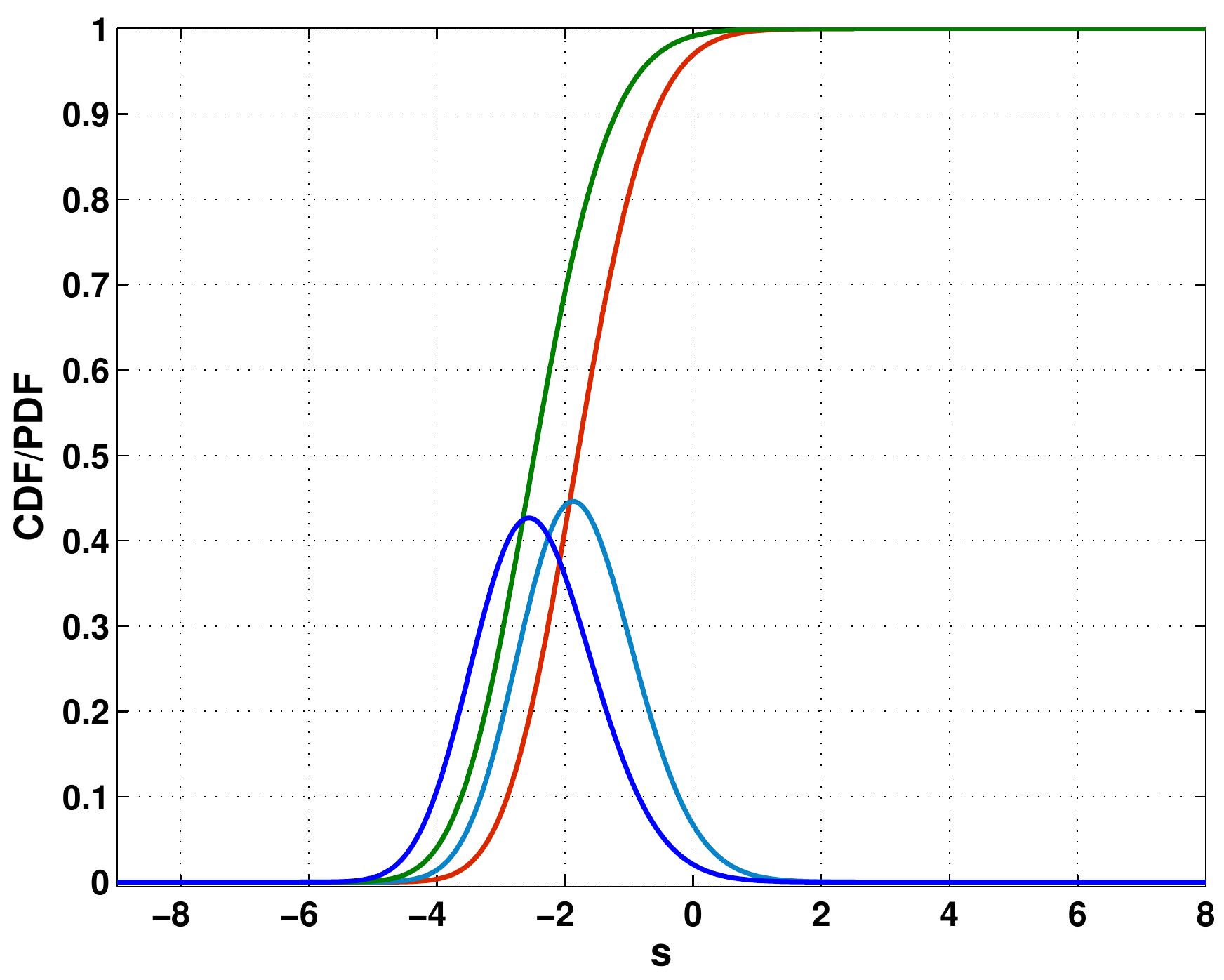}\\*[-1mm]
{\footnotesize b. \, largest eigenvalue of LUE ($n=80$, $\alpha=40$)}
\end{center}
\end{minipage}
\end{center}\vspace*{-0.0625cm}
\caption{CDF (green) and PDF (dark blue) of the smallest and largest eigenvalue for $n$-dimensional LUE with parameter $\alpha$ with $n=80$, $\alpha=40$.
Also shown are the scaling limits at the hard and soft edge (CDF in red, PDF in light blue).}
\label{fig:LUE}
\end{figure}

\begin{table}[tbp]
\caption{Moments of the distributions (\protect\ref{eq:G}), (\protect\ref{eq:F}), (\protect\ref{eq:tildeG}), and (\protect\ref{eq:F2new}) for LUE with $n=80$ and $\alpha=40$.
We show ten correctly \emph{truncated} digits (that passed the error control) and give the computing time. All calculations were done in hardware arithmetic.}
\vspace*{0mm}
\centerline{%
\setlength{\extrarowheight}{3pt}
\begin{tabular}{crrrrr}\hline
CFD & mean\phantom{xxx} &  variance\phantom{xx} & skewness\phantom{xx} & kurtosis\phantom{xx} & time\phantom{x}\\*[0.5mm]\hline
(\protect\ref{eq:G})       & $ 5.14156\,81318$ & $0.34347\,52478$ & $0.04313\,30951$ & $-0.02925\,63564$ & $1.0$ sec\\
(\protect\ref{eq:tildeG})  & $ 6.35586\,98372$ & $0.52106\,15307$ & $0.04102\,67718$ & $-0.02943\,22640$ & $1.2$ sec\\
(\protect\ref{eq:F})       & $-2.43913\,84563$ & $0.89341\,23428$ & $0.26271\,64962$ & $ 0.12783\,51672$ & $4.1$ sec\\
(\protect\ref{eq:F2new})   & $-1.77108\,68074$ & $0.81319\,47928$ & $0.22408\,42036$ & $ 0.09344\,80876$ & $1.0$ sec\\*[0.5mm]\hline
\end{tabular}}
\label{tab:LUE}
\end{table}

\subsubsection{The distribution of k-level spacings in the bulk of GUE}

By (\ref{eq:Ebeta}) and (\ref{eq:E2bulk}), the $k$-level spacing functions in the bulk scaling limit of GUE are given by the determinantal expressions
\begin{equation}\label{eq:E2k}
E_2(k;s) = \left. \frac{(-1)^k}{k!} \frac{d^k}{dz^k} \det\left(I-z\,K_{\sin}\projected{L^2(0,s)}\right) \right|_{z=1}.
\end{equation}
\citeasnoun{MR0348823} evaluated them using Gaudin's method (which be briefly described in Section~\ref{sect:gaudin}); a plot of these functions,
for $k$ from 0 up to 14, can also be found in \citeasnoun[Fig.~6.4]{MR2129906}. Now, the numerical evaluation of the expression (\ref{eq:E2k})
is directly amenable to our approach. The results of our calculations are shown in Figure~\ref{fig:GUEexamples}.a.

\begin{figure}[tbp]
\begin{center}
\begin{minipage}{0.49\textwidth}
\begin{center}
\includegraphics[width=\textwidth]{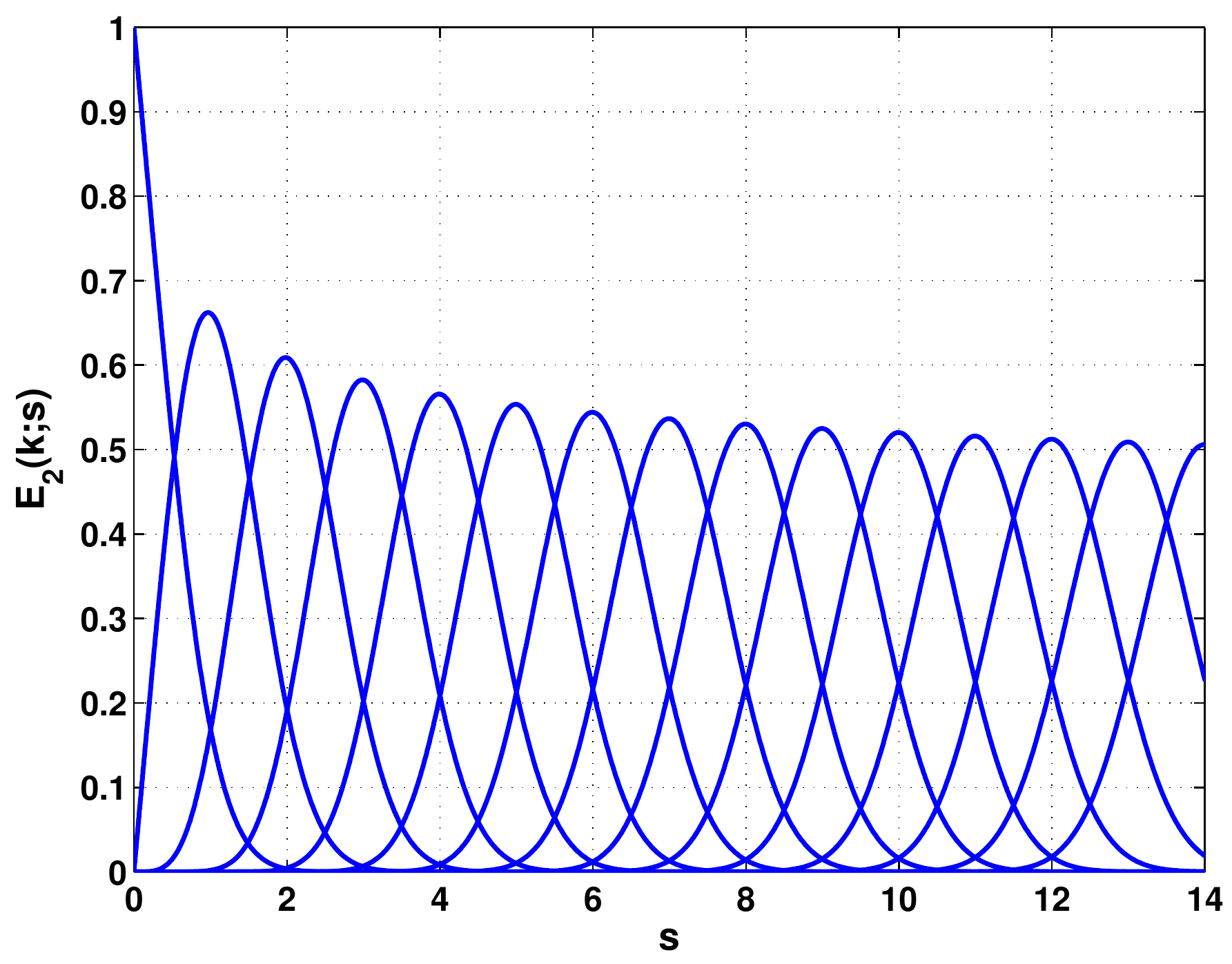}\\*[-1mm]
{\footnotesize a. \, $k$-level spacing functions $E_2(k;s)$ of GUE}
\end{center}
\end{minipage}
\hfil
\begin{minipage}{0.49\textwidth}
\begin{center}
\includegraphics[width=\textwidth]{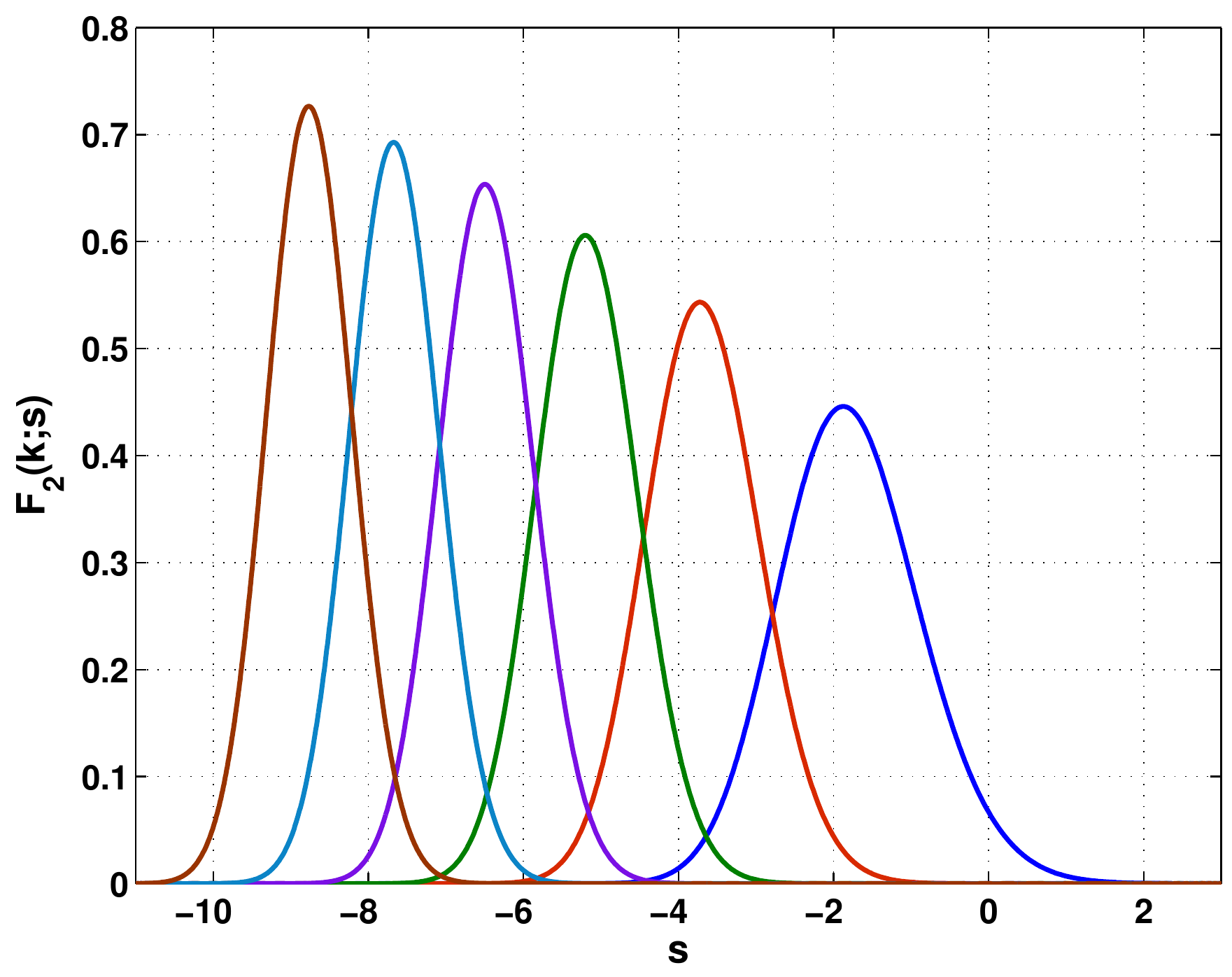}\\*[-1mm]
{\footnotesize b. \, density $\partial_s F_2(k;s)$ of $k$-th largest level in GUE}
\end{center}
\end{minipage}
\end{center}\vspace*{-0.0625cm}
\caption{Plots of the $k$-level spacing functions $E_2(k;s)$ in the bulk scaling limit of GUE ($k=0,\ldots,14$; larger $k$ go to the right),
and of the probability density functions $\partial_s F_2(k;s)$ of the $k$-th largest level in the soft edge scaling limit of GUE ($k=1,\ldots,6$; larger $k$ go to the left). The
underlying calculations were all done in hardware arithmetic and are accurate to an absolute error of about $5\cdot 10^{-15}$. Each of the two plots
took a run time of about 30 seconds. Compare with \protect\citeasnoun[Fig.~6.4]{MR2129906} and \protect\citeasnoun[Fig.~2]{MR1257246}.
}
\label{fig:GUEexamples}
\end{figure}

\subsubsection{The distributions of the k-th largest level at the soft edge of GUE}
By (\ref{eq:Fbeta}) and (\ref{eq:E2Edge}), the cumulative distribution functions of the $k$-th largest level in the soft edge scaling limit of GUE are given by the determinantal expressions
\begin{equation}\label{eq:F2k}
F_2(k;s) = \sum_{j=0}^{k-1} \left. \frac{(-1)^j}{j!} \frac{d^j}{dz^j} \det\left(I-z\,K_{\Ai}\projected{L^2(s,\infty)}\right) \right|_{z=1},
\end{equation}
which is directly amenable to be evaluated by our numerical approach. The results of our calculations are shown in Figure~\ref{fig:GUEexamples}.b.

\subsubsection*{Remark}
To our knowledge, prior to this work, only calculations of the particular cases $k=1$ (the largest level) and $k=2$ (the next-largest level) have been reported in the literature \cite{MR1791893,Dieng05}.
These calculations were based on the representation (\ref{eq:D2PII}) of the determinant in terms of  the Painlevé II equation~(\ref{eq:PIIGE}). The evaluation of
$F_2(s)=F_2(1;s)$ was obtained from the Hastings--McLeod solution $u(x)=u(x;1)$, see Section~\ref{sect:TW2}. On the other hand, the evaluation
of $F_2(2;s)$ required the function
\begin{equation}
w(x) = \partial_z u(x;z)|_{z=1},
\end{equation}
which, by differentiating (\ref{eq:PIIGE}), is easily seen to satisfy the linear ordinary differential equation
\begin{equation}
w''(x) = (6u(x)^2+x)w(x),\qquad w(x)\simeq \tfrac12 \Ai(x)\quad(x\to\infty).
\end{equation}
Obtaining the analogue of the connection formula (\ref{eq:HMasympt}) requires some work (though, since the underlying differential equation is linear, it poses no fundamental difficulty)
and one gets (see \citeasnoun[p.~164]{MR1257246} who also give expansions for larger $k$)
\begin{multline}
w(x) =\frac{e^{2\sqrt2 (-x)^{3/2}/3}}{2^{7/4} \sqrt{\pi}(-x)^{1/4}}\left(1 + \frac{17}{48\sqrt 2}(-x)^{-3/2} -\frac{1513}{9216}x^{-3} + O((-x)^{-9/2})\right)\\*[1mm]
\quad(x\to-\infty).
\end{multline}
Note that the exponential growth points, once more, to the instability we have discussed in Section~\ref{sect:instability}.

\section{The Distribution of k-Level Spacings in the Bulk: GOE and GSE}\label{sect:bulk}
\citeasnoun[Chap.~20]{MR2129906} gives determinantal formulae for the $k$-level spacing functions $E_\beta(k;s)$ in the bulk scaling limit that are
(also in the cases $\beta=1$ and $\beta=4$ of the GOE and GSE, respectively) directly
amenable to the numerical approach of Section~\ref{sect:fredholm}. These formulae are based on a factorization of the sine kernel determinant (\ref{eq:D2bulk}), which we describe first.

Since $K_{\sin}$ is a convolution operator we have the shift invariance
\begin{equation}
\det\left(I-z K_{\sin}\projected{L^2(0,2t)}\right) = \det\left(I-z K_{\sin}\projected{L^2(-t,t)}\right).
\end{equation}
Next, there is the orthogonal decomposition $L^2(-t,t)= X_t^{\text{even}} \oplus X_t^{\text{odd}}$ into the even and odd functions. On the level of operators, this corresponds to the
block diagonalization
\begin{equation}
K_{\sin}\projected{L^2(-t,t)} =
\begin{pmatrix}
K_{\sin}^+ & \\*[2mm]
 & K_{\sin}^-
\end{pmatrix}\projected{X_t^{\text{even}} \oplus X_t^{\text{odd}}}
\end{equation}
with the kernels
\begin{equation}
K^\pm_{\sin} (x,y) = \tfrac12(K_{\sin}(x,y) \pm K_{\sin}(x,-y)).
\end{equation}
Further, there is obviously
\begin{equation}
K_{\sin}^+\projected{L^2(-t,t)} =
\begin{pmatrix}
K_{\sin}^+ & \\*[2mm]
 & 0
\end{pmatrix}\projected{X_t^{\text{even}} \oplus X_t^{\text{odd}}},\qquad
K_{\sin}^-\projected{L^2(-t,t)} =
\begin{pmatrix}
 0& \\*[2mm]
 & K_{\sin}^-
\end{pmatrix}\projected{X_t^{\text{even}} \oplus X_t^{\text{odd}}}.
\end{equation}
Hence, we get the factorization
\begin{equation}\label{eq:bulkfact}
\det\left(I-z K_{\sin}\projected{L^2(-t,t)}\right) = \det\left(I-z K_{\sin}^+\projected{L^2(-t,t)}\right)  \det\left(I-z K_{\sin}^-\projected{L^2(-t,t)}\right).
\end{equation}
Now, upon introducing the functions
\begin{equation}\label{eq:Epm}
E_\pm(k;s) = \left. \frac{(-1)^k}{k!} \frac{d^k}{dz^k} \det\left(I-z\,K_{\sin}^\pm\projected{L^2(-s/2,s/2)}\right) \right|_{z=1}
\end{equation}
we obtain, using the factorization (\ref{eq:bulkfact}) and the Leibniz formula applied to (\ref{eq:E2k}), the representation
\begin{equation}\label{eq:E2sum}
E_2(k;s) = \sum_{j=0}^k E_+(j;s) E_-(k-j;s)
\end{equation}
of the $k$-level spacing functions in the bulk of GUE. The important point here is that \citeasnoun[Eqs.~(20.1.20/21)]{MR2129906} succeeded
in representing  the $k$-level spacing functions of GOE and GSE in terms of the functions $E_\pm$, too:
\begin{subequations}\label{eq:E1}
\begin{align}
E_1(0;s) &= E_+(0;s), \label{eq:E10}\\
E_1(2k-1;s) &= E_-(k-1;s) - E_1(2k-2;s),\\
E_1(2k;s) &= E_+(k;s) - E_1(2k-1;s),
\end{align}
\end{subequations}
($k=1,2,3,\ldots$) for GOE and
\begin{equation}\label{eq:E4}
E_4(k;s) = \tfrac12(E_+(k;2s)+E_-(k;2s))
\end{equation}
($k=0,2,3,\ldots$) for GSE. Based on these formulae, we have used the numerical methods of Section~\ref{sect:fredholm} to reproduce the plots of \citeasnoun[Figs.~7.3/11.1]{MR2129906}.
The results of our calculations are shown in Figure~\ref{fig:GOEGSEbulk}.

\begin{figure}[tbp]
\begin{center}
\begin{minipage}{0.49\textwidth}
\begin{center}
\includegraphics[width=\textwidth]{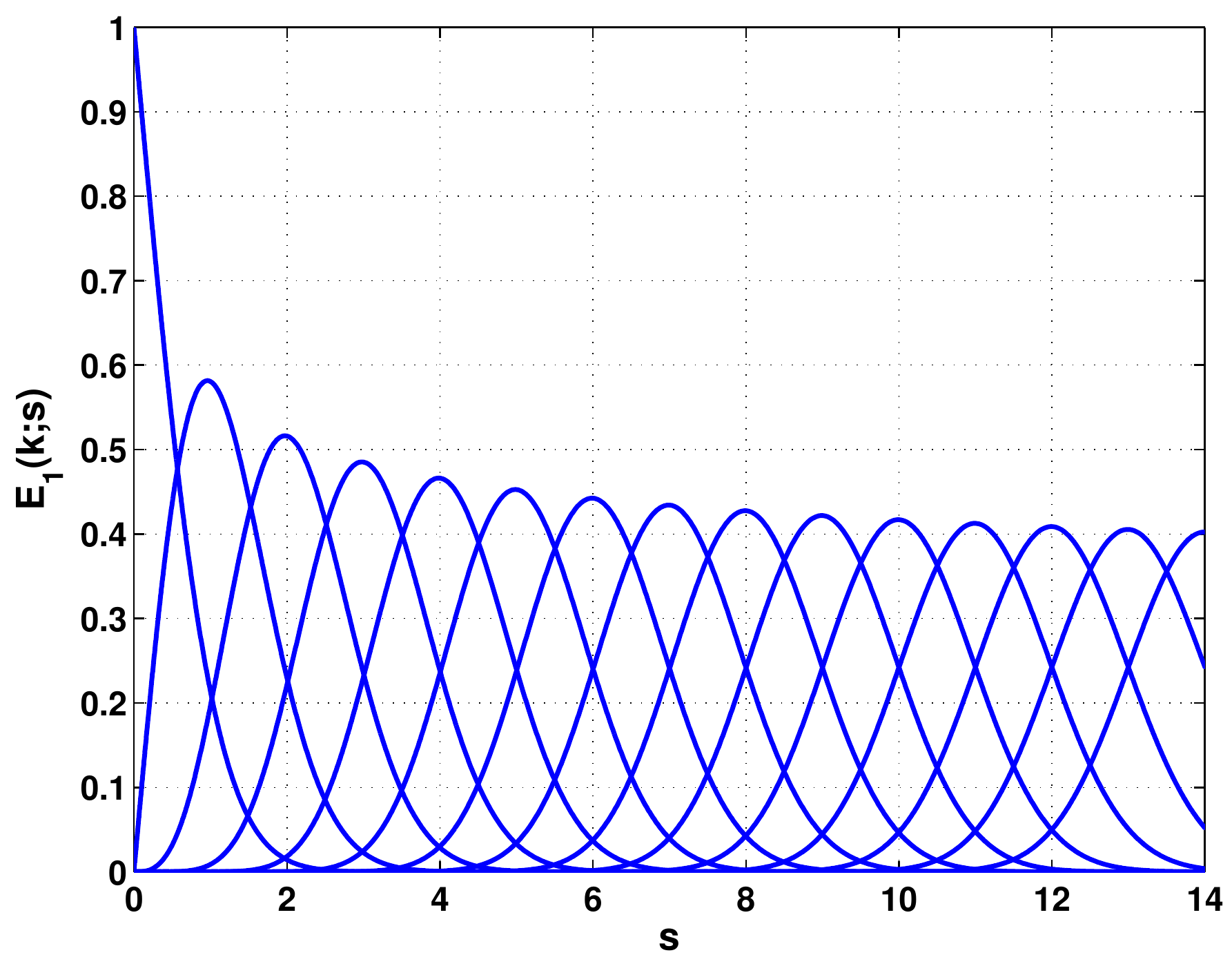}\\*[-1mm]
{\footnotesize a. \, $k$-level spacing functions $E_1(k;s)$ of GOE}
\end{center}
\end{minipage}
\hfil
\begin{minipage}{0.49\textwidth}
\begin{center}
\includegraphics[width=\textwidth]{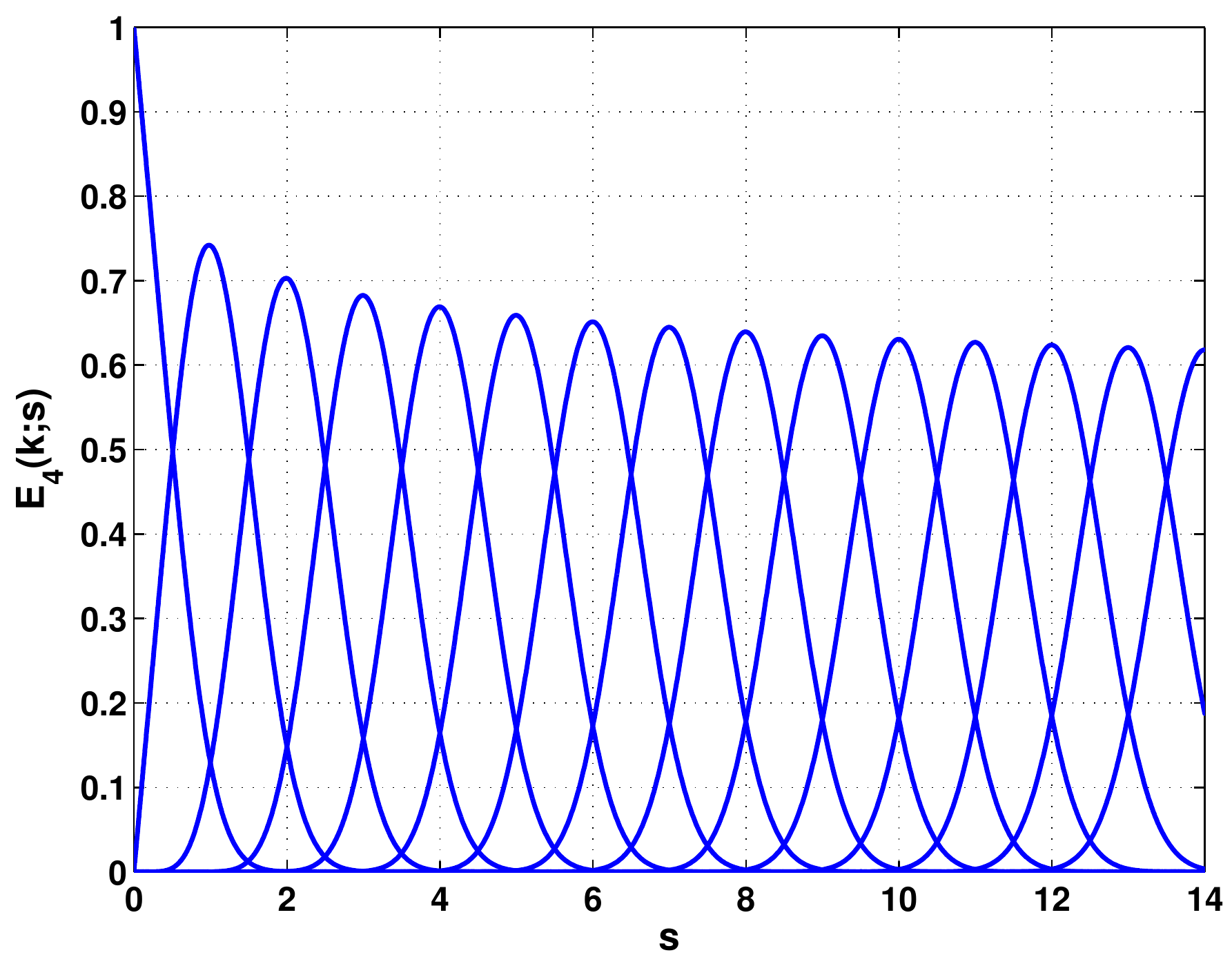}\\*[-1mm]
{\footnotesize b. \, $k$-level spacing functions $E_4(k;s)$ of GSE}
\end{center}
\end{minipage}
\end{center}\vspace*{-0.0625cm}
\caption{Plots of the $k$-level spacing functions $E_1(k;s)$ in the bulk scaling limit of GOE
and $E_4(k;s)$ in the bulk scaling limit of GSE ($k=0,\ldots,14$; larger $k$ go to the right). The
underlying calculations were all done in hardware arithmetic and are accurate to  an absolute error of about $5\cdot 10^{-15}$. Each of the two plots
took a run time of less than one minute. Compare with \protect\citeasnoun[Figs.~7.3/11.1]{MR2129906}.
}
\label{fig:GOEGSEbulk}
\end{figure}

\section{The k-th Largest Level at the Soft Edge: GOE and GSE}\label{sect:edge}

In this section we derive new determinantal formulae for the cumulative distribution functions $F_1(k;s)$ and $F_4(k;s)$ of the $k$-th largest level
in the soft edge scaling limit of GOE and GSE. We recall from (\ref{eq:Fbeta}) that
\begin{equation}\label{eq:Fbeta2}
F_\beta(k;s) = \sum_{j=0}^{k-1} \tilde E_\beta(j;s),
\end{equation}
where we briefly write
\begin{equation}\label{eq:EbetaTilde}
\tilde E_\beta(k;s) = E_\beta^{(\text{soft})}(k;(s,\infty)).
\end{equation}
The new determinantal formulae of this section are amenable to the efficient numerical evaluation by the methods of Section~\ref{sect:fredholm};
but, more important, they are derived from a determinantal equation (\ref{eq:conjecture}) whose truth we established by numerical experiments \emph{before} proving it rigorously.
Therefore, we understand this section as an invitation to the area of Experimental Mathematics \cite{MR2033012}.

In a broad analogy to the previous section we start with a factorization of the Airy kernel determinant (\ref{eq:D2edge}) that we have learnt from \citeasnoun[Eq.~(34)]{MR2165698}.
Namely, the integral representation
\begin{equation}\label{eq:KairyFact}
K_{\Ai}(x,y) = \int_0^\infty \Ai(x+\xi)\Ai(y+\xi)\,d\xi
\end{equation}
of the Airy kernel implies, by introducing the kernels
\begin{equation}\label{eq:K1}
T_s(x,y) = \Ai(x+y+s), \qquad V_{\Ai}(x,y) =\frac{1}{2}\Ai\left(\frac{x+y}{2}\right),
\end{equation}
the factorization
\begin{multline}\label{eq:edgefact}
\det\left(I- z K_{\Ai}\projected{L^2(s,\infty)}\right) = \det\left(I- z\left(T_s\projected{L^2(0,\infty)}\right)^2 \right)\\*[2mm]
=\det\left(I- \sqrt{z}\, T_s\projected{L^2(0,\infty)}\right) \cdot \det\left(I+ \sqrt{z}\, T_s\projected{L^2(0,\infty)}\right)\\*[2mm]
=\det\left(I- \sqrt{z}\, V_{\Ai}\projected{L^2(s,\infty)}\right) \cdot \det\left(I+ \sqrt{z}\, V_{\Ai}\projected{L^2(s,\infty)}\right)
\end{multline}
that is valid for the complex cut plane $z \in \C\setminus(-\infty,0]$.
Now, upon introducing the functions
\begin{equation}\label{eq:EpmTilde}
\tilde E_\pm(k;s) = \left. \frac{(-1)^k}{k!} \frac{d^k}{dz^k} \det\left(I\mp\sqrt{z}\,V_{\Ai}\projected{L^2(s,\infty)}\right) \right|_{z=1}
\end{equation}
we obtain, using the factorization (\ref{eq:edgefact}) and the Leibniz formula applied to (\ref{eq:E2Edge}), the representation
\begin{equation}\label{eq:E2EdgeSum}
\tilde E_2(k;s) = \sum_{j=0}^k \tilde E_+(j;s) \tilde E_-(k-j;s).
\end{equation}
Further, \citeasnoun[Eqs.~(33/35)]{MR2165698} proved that
\begin{subequations}\label{eq:ferrari}
\begin{align}
\tilde E_1(0;s) &= \tilde E_+(0;s) \label{eq:E10Edge}\\*[2mm]
\tilde E_4(0;s) &= \tfrac{1}{2}(\tilde E_+(0;s) + \tilde E_-(0;s)).\label{eq:E40Edge}
\end{align}
\end{subequations}
The similarity of the pairs of formulae (\ref{eq:E2EdgeSum})/(\ref{eq:E2sum}), (\ref{eq:E10Edge})/(\ref{eq:E10}), and (\ref{eq:E40Edge})/(\ref{eq:E4}) (the last with $k=0$) is
absolutely striking. So we asked ourselves whether (\ref{eq:E40Edge}) generalizes to the analogue of (\ref{eq:E4}) for general $k$, that is, whether
\begin{equation}\label{eq:E4Edge}
\tilde E_4(k;s) = \tfrac{1}{2}(\tilde E_+(k;s) + \tilde E_-(k;s))\qquad (k=0,1,2,\ldots)
\end{equation}
is valid in general. In view of (\ref{eq:GSEMatrixKernelDet}) such a result is \emph{equivalent} to the following theorem. We first convinced ourselves of its truth in the
sense of experimental mathematics: by numerically\footnote{The function $D_4(z;(s,\infty))$
was evaluated using the extension of our method to matrix kernel determinants that will be discussed in Section~\ref{sect:matrixkern}; see also Example~\ref{sect:instance}
for a concrete instance.} checking its assertion for 100\,000 randomly chosen arguments. Thus being encouraged, we then worked out the proof given below.\footnote{Later though, we found that
the result has recently been established by \citeasnoun[Eq.~(1.23)]{MR2275509}.}

\begin{theorem}\label{thm:forrester} The determinantal equation
\begin{multline}\label{eq:conjecture}
D_4(z;(s,\infty))^{1/2} \\*[1mm]
= \frac{1}{2}\left(\det\left(I-\sqrt{z}\,V_{\Ai}\projected{L^2(s,\infty)}\right)+\det\left(I+\sqrt{z}\,V_{\Ai}\projected{L^2(s,\infty)}\right)\right)
\end{multline}
holds for all $s \in \R$ and $z$ in the complex domain of analyticity that contains $z=1$.
\end{theorem}

\begin{proof}
Since the operator theoretic arguments of \citeasnoun{MR2165698} cannot directly be extended to yield (\ref{eq:conjecture}) we proceed
by using Painlevé representations. \citeasnoun[Thm.~1.2.1, Eq.~(1.2.2)]{Dieng05} proved that (\ref{eq:F4PII}) generalizes to
\begin{subequations}
\begin{equation}\label{eq:DiengD4}
D_4(\theta^2;(s,\infty))^{1/2} = D_2^{(\text{soft})}(\theta^2;(s,\infty))^{1/2} \cosh\left(\frac{1}{2} \int_s^\infty u(x;\theta)\,dx\right)
\end{equation}
in terms of the Painlevé representation
\begin{equation}
u_{xx} = 2u^3 +x u,\qquad u(x;\theta)\simeq\theta \cdot \Ai(x)\quad(x\to\infty).
\end{equation}
\end{subequations}
Here we put $\sqrt{\smash[b]{z}}=\theta$ and observe that (\ref{eq:DiengD4}) obviously extends, by the symmetry of the Painlevé II equation, from $0<\theta\leq 1$ to the range $-1\leq \theta \leq 1$.
In view of~(\ref{eq:F2PII})  Dieng's representation (\ref{eq:DiengD4}) readily implies, by analytic continuation, the asserted formula (\ref{eq:conjecture}) \emph{if} the representation
\begin{multline}\label{eq:conjecture2}
\det\left(I -\theta\,V_{\Ai}\projected{L^2(s,\infty)}\right) = \exp\left(-\frac{1}{2}\int_s^\infty u(x;\theta)\,dx\right)\det\left(I-\theta^2 \,K_{\Ai}\projected{L^2(s,\infty)}\right)^{1/2}\\*[2mm]
= \exp\left(-\frac{1}{2}\int_s^\infty (u(x;\theta) + (x-s)u(x;\theta)^2)\,dx\right)
\end{multline}
happens to be true for all $-1\leq\theta\leq 1$ (note that we have chosen the signs in accordance with the
the special cases $\theta=\pm 1$ as given by (\ref{eq:ferrari}) and~(\ref{eq:GOEGSEPII})). Now, this particular Painlevé representation can directly be read off from the work of \citeasnoun[Eqs.~(4.8/19)]{MR2229797}, which
completes the proof.
\end{proof}


It remains to establish formulae for the GOE functions $\tilde E_1(k;s)$ that are structurally similar to (\ref{eq:E1}). To this end we use
the interrelationships
between GOE, GUE, and GSE found by \citeasnoun[Thm.~5.2]{MR1842786},\footnote{That is why we have chosen,  in defining the Gaussian ensembles,
the same variances of the Gaussian weights as \citeasnoun{MR1842786}.} which can symbolically be written in the form
\begin{subequations}
\begin{align}
\text{GSE}_{n} &= \text{even}(\text{GOE}_{2n+1}) , \label{eq:GSEfromeven}\\
\text{GUE}_{n} &= \text{even}(\text{GOE}_{n} \cup \text{GOE}_{n+1}).\label{eq:GUEfromeven}
\end{align}
\end{subequations}
The meaning is as follows: First, the statistics of the ordered eigenvalues of the $n$-dimensional GSE is the same as that of the even numbered ordered eigenvalues of the $2n+1$-dimensional GOE.
Second, the statistics of the ordered eigenvalues of the $n$-dimensional GUE is the same as that of the even numbered ordered levels obtained from joining the eigenvalues
of a $n$-dimensional GOE with
the eigenvalues of a statistically independent $n+1$-dimensional GOE.

Now, (\ref{eq:GSEfromeven}) readily implies, in the soft edge scaling limit (\ref{eq:GEedge}), that the cumulative distribution function of the $k$-th largest eigenvalue in GSE agrees with the
cumulative distribution function of the $2k$-th largest of GOE,
\begin{equation}\label{eq:interlacing}
F_4(k;s) = F_1(2k;s)\qquad (k=1,2,3,\ldots),
\end{equation}
the so-called \emph{interlacing property}. Therefore, in view of (\ref{eq:Fbeta2}) and (\ref{eq:E4Edge}) we get
\begin{equation}\label{eq:func3}
\tilde E_1(2k;s) + \tilde E_1(2k+1;s) = \tfrac12(\tilde E_+(k;s) + \tilde E_-(k;s)).
\end{equation}

Further, the combinatorics of (\ref{eq:GUEfromeven}) implies, in the soft edge scaling limit (\ref{eq:GEedge}): exactly $k$ levels of GUE are larger than $s$ if and only if
exactly $2k$ or $2k+1$ levels of the union of GOE with itself
are larger than $s$. Here, $j$ levels are from the first copy of GOE and $2k-j$, or $2k+1-j$, are from the second copy. Since all of these events are mutually exclusive, we get
\begin{equation}\label{eq:func4}
\tilde E_2(k;s) =
 \sum_{j=0}^{2k} \tilde E_1(j;s) \tilde E_1(2k-j;s) + \sum_{j=0}^{2k+1} \tilde E_1(j;s) \tilde E_1(2k+1-j;s).
\end{equation}

Finally, the following theorem gives the desired (recursive) formulae for the functions $\tilde E_1(k;s)$ in terms of $\tilde E_\pm(k;s)$. Note that these recursion formulae, though being quite different from
(\ref{eq:E1}), share the separation into even and odd numbered cases.

\begin{theorem}\label{thm:algebra} The system (\ref{eq:E2EdgeSum}), (\ref{eq:E10Edge}), (\ref{eq:func3}), and (\ref{eq:func4}) of functional equations has the unique, recursively defined solution
\begin{subequations}\label{eq:E1recursion}
\begin{align}
\tilde E_1(2k;s) &= \tilde E_+(k;s) - \sum_{j=0}^{k-1} \frac{\binom{2j}{j}}{2^{2j+1}(j+1)} \tilde E_1(2k-2j-1;s),\\*[2mm]
\tilde E_1(2k+1;s) &= \frac{\tilde E_+(k;s)+\tilde E_-(k;s)}{2} - \tilde E_1(2k;s).
\end{align}
\end{subequations}
\end{theorem}

\begin{proof}
We introduce the generating functions
\begin{subequations}
\begin{align}
f_{\text{even}}(x) &= \sum_{k=0}^\infty \tilde E_1(2k;s) x^{2k}\\*[1mm]
f_{\text{odd}}(x) &= \sum_{k=0}^\infty \tilde E_1(2k+1;s) x^{2k}\\*[1mm]
g_\pm(x) &= \sum_{k=0}^\infty \tilde E_\pm(k;s) x^{2k}.
\end{align}
\end{subequations}
\possessivecite{MR2165698} representation (\ref{eq:E10Edge}) translates into the constant term equality (which breaks the symmetry of the other functional equations)
\begin{equation}\label{eq:f0}
f_{\text{even}}(0)= g_+(0).
\end{equation}
Equating (\ref{eq:E2EdgeSum}) and (\ref{eq:func4})  translates into
\begin{equation}
f_{\text{even}}(x)^2 + 2 f_{\text{even}}(x)f_{\text{odd}}(x) +x^2 f_{\text{odd}}(x)^2 = g_+(x)\cdot g_-(x).
\end{equation}
Finally, (\ref{eq:func3}) translates into
\begin{equation}
f_{\text{even}}(x) + f_{\text{odd}}(x) = \tfrac12 (g_+(x) + g_-(x)).
\end{equation}
Elimination of $g_-$ from the last two equations results in the quadratic equation
\begin{equation}
(f_{\text{even}}(x) - g_+(x))^2 + 2(f_{\text{even}}(x) - g_+(x))f_{\text{odd}}(x) + x^2 f_{\text{odd}}(x)^2 = 0.
\end{equation}
Solving for $f_{\text{even}}(x) - g_+(x)$ gives the two possible solutions
\begin{equation}
f_{\text{even}}(x) - g_+(x) = -\left( 1 \pm \sqrt{1-x^2}\right) f_{\text{odd}}(x).
\end{equation}
Because of $f_{\text{odd}}(0) = \tilde E_1(1;s) >0$ we have to choose the negative sign of the square root to
satisfy (\ref{eq:f0}). To summarize, we have obtained the mutual relations
\begin{subequations}\label{eq:feo}
\begin{align}
f_{\text{even}}(x) &= g_+(x) - \left(1-\sqrt{1-x^2}\right) f_{\text{odd}}(x), \\*[0mm]
 f_{\text{odd}}(x) &= \tfrac12 (g_+(x) + g_-(x)) - f_{\text{even}}(x),
\end{align}
\end{subequations}
which then, by the Maclaurin expansion
\begin{equation}
1-\sqrt{1-x^2} = \sum_{j=0}^\infty \frac{\binom{2j}{j}}{2^{2j+1}(j+1)} x^{2j+2},
\end{equation}
translate back into the asserted recursion formulae of the theorem.
\end{proof}

\begin{remark} The system (\ref{eq:feo}) can readily be solved for $f_{\text{even}}(x)$ and $f_{\text{odd}}(x)$ to yield
\begin{multline}
\sum_{k=0}^\infty \tilde E_1(k;s) x^k = f_{\text{even}}(x) + x f_{\text{odd}}(x) \\*[1mm]
= \frac{1}{2} \left(g_+(x) \left(1+\sqrt\frac{1-x}{1+x}\,\right) + g_-(x)\left(1-\sqrt\frac{1-x}{1+x}\,\right)\right).
\end{multline}
In view of (\ref{eq:GOEMatrixKernelDet}) this implies the determinantal equation \citeaffixed[Eq.~(1.22)]{MR2275509}{cf., after squaring,}
\begin{multline}\label{eq:D1Formel}
D_1(z;(s,\infty))^{1/2} = \frac{1}{2}\left( \det\left(I -\sqrt{z(2-z)} \,V_{\Ai}\projected{L^2(s,\infty)} \right)\left(1+\sqrt\frac{z}{2-z}\,\right)\right.\\*[2mm]
\left. + \det\left(I +\sqrt{z(2-z)} \,V_{\Ai}\projected{L^2(s,\infty)} \right)\left(1-\sqrt\frac{z}{2-z}\,\right)\right)
\end{multline}
for all $s\in\R$ and $z$ in the complex domain of analyticity that contains $z=1$. This formula
thus paves, following the arguments of the proof of Theorem~\ref{thm:forrester}, an elementary road to
the Painlevé representation \cite[Eq.~(1.2.1)]{Dieng05} of $D_1(z;(s,\infty))$ \citeaffixed[Eq.~(1.17)]{MR2275509}{see also}.
\end{remark}

\begin{figure}[tbp]
\begin{center}
\begin{minipage}{0.49\textwidth}
\begin{center}
\includegraphics[width=\textwidth]{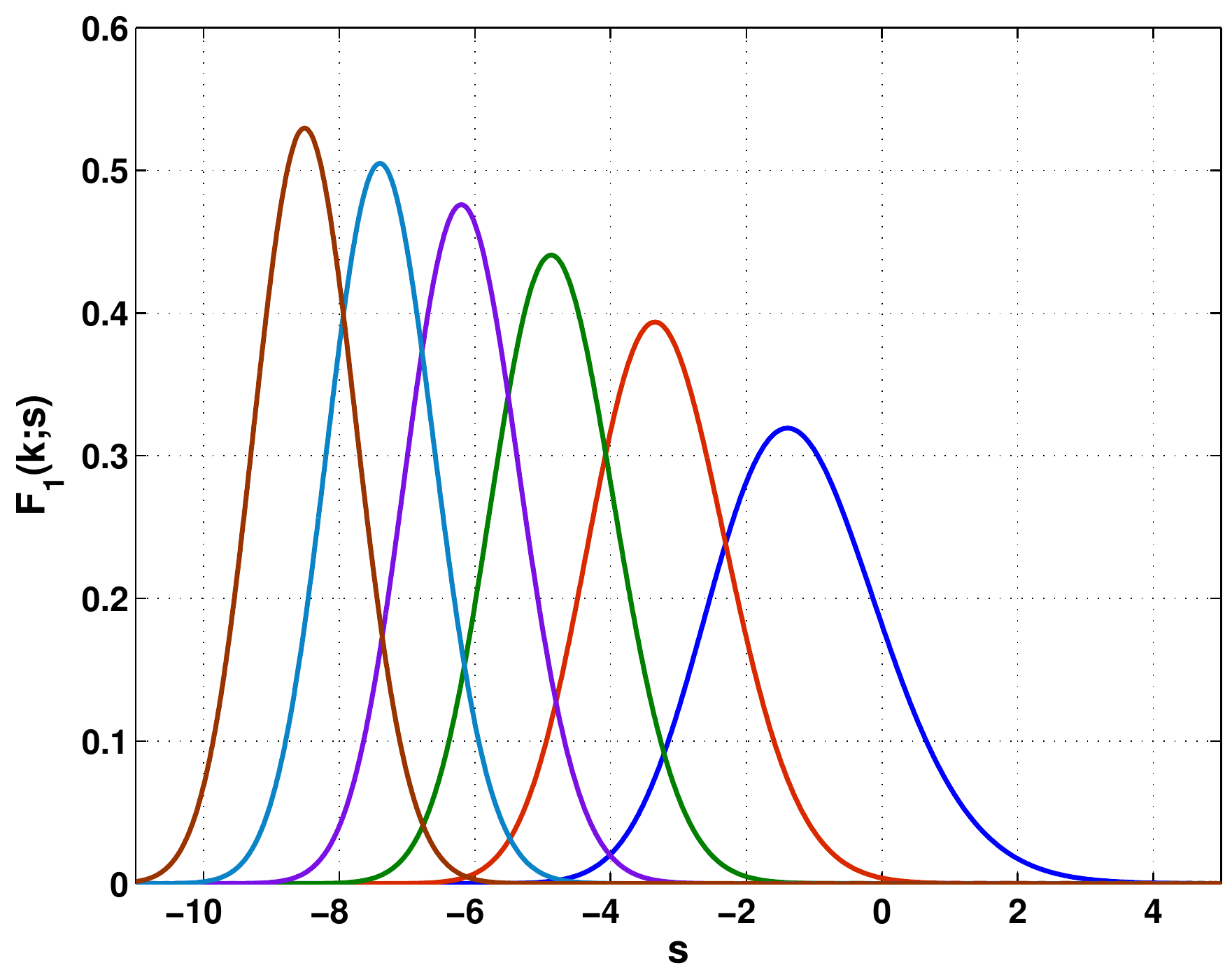}\\*[-1mm]
{\footnotesize a. \, density $\partial_s F_1(k;s)$ of $k$-th largest level in GOE}
\end{center}
\end{minipage}
\hfil
\begin{minipage}{0.49\textwidth}
\begin{center}
\includegraphics[width=\textwidth]{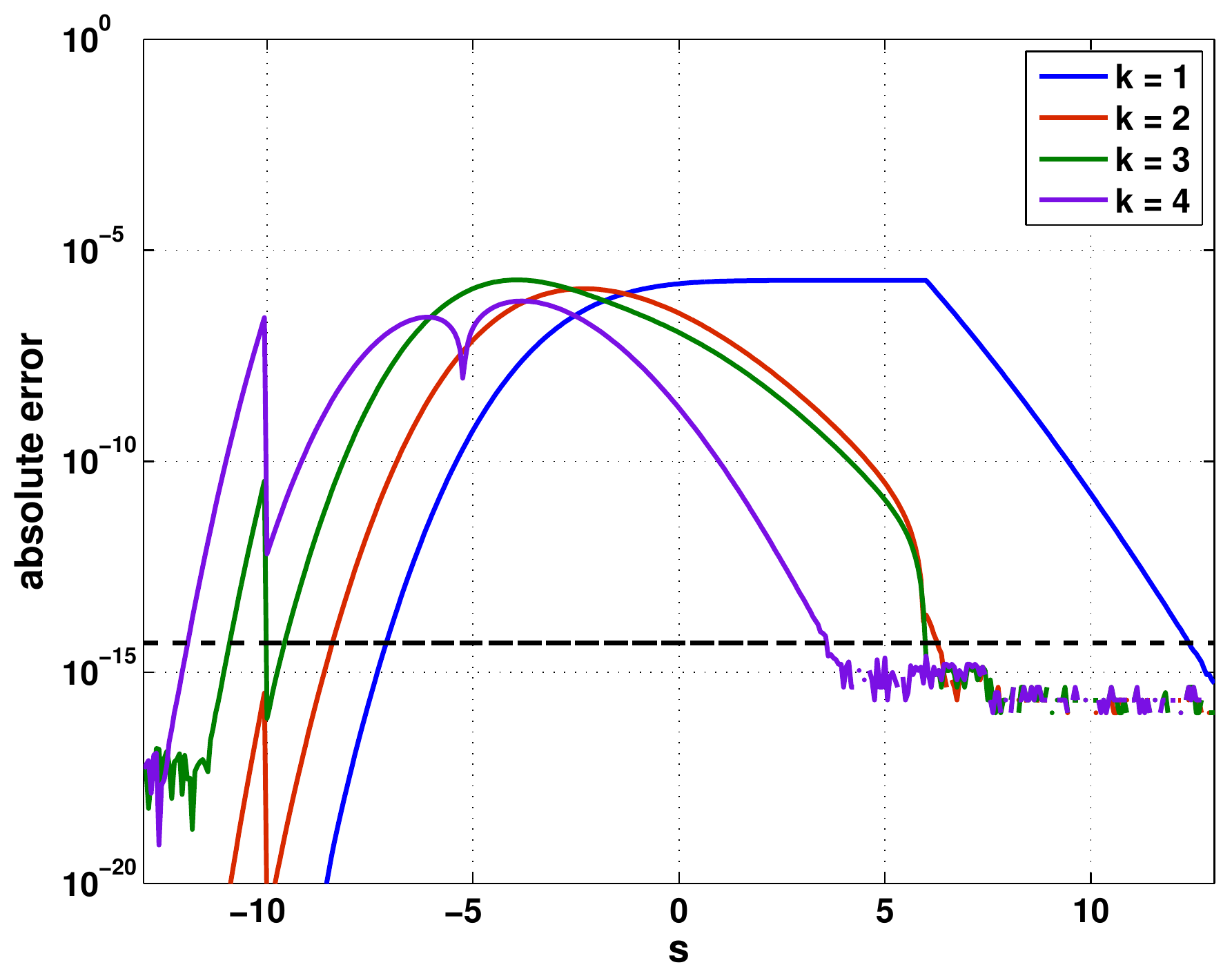}\\*[-1mm]
{\footnotesize b. \, absolute error of \protect\possessivecite{Dieng05} evaluation}
\end{center}
\end{minipage}
\end{center}\vspace*{-0.0625cm}
\caption{Left: plots of the probability densities of the $k$-th largest level in the soft edge scaling limit of GOE ($k=1,\ldots,6$; larger $k$ go to the left)
based on the recursion formulae (\ref{eq:E1recursion}). The underlying calculations were all done in hardware arithmetic and are accurate to  an
absolute error of about $5\cdot 10^{-15}$ (dashed line). Right: the absolute error (taking the values of the calculations on the left as reference) of the Painlevé II based calculations by \protect\citeasnoun{Dieng05}.
}
\label{fig:GOEedge}
\end{figure}

Based on the recursion formulae (\ref{eq:E1recursion}) and the numerical methods of Section~\ref{sect:fredholm} we calculated the distribution
functions $F_1(k;s)$ of the $k$-th largest level in the soft edge scaling limit of GOE---note that because of (\ref{eq:interlacing}) there is
no need for a separate calculation of the corresponding distributions $F_4(k;s)$ for GSE. The corresponding densities are shown, for $k=1,\ldots,6$, in Figure~\ref{fig:GOEedge}.a.
These calculations
are accurate to the imposed absolute tolerance of $5\cdot 10^{-15}$. Taking our numerical solutions as reference, Figure~\ref{fig:GOEedge}.b shows the absolute error
of the Painlevé II based numerical evaluations by \citeasnoun{Dieng05}, which are available for $k=1,\ldots,4$. Perfectly visible are the points $a_+=-10$ and $b_-=6$ where \citename{Dieng05}
chose to switch from an asymptotic formula to the BVP solution of Painlevé II, see Section~\ref{sect:TW2}.

\section{The k-th Smallest Level at the Hard Edge: LOE and LSE}\label{sect:hard}

In this section we derive determinantal formulae for the cumulative distribution functions $F_{\beta,\alpha}(k;s)$ of the $k$-th smallest
level in the hard edge scaling limit of LOE, LUE, and LSE with parameter $\alpha$. It turns out that these formulae have the same \emph{algebraic} structure as those
developed for the $k$-th largest level in the soft edge scaling limit. Therefore, though more concise, we proceed step-by-step in parallel to the arguments of the preceding section.

By the underlying combinatorial structure we have
\begin{equation}\label{eq:FLE}
F_{\beta,\alpha}(k;s) = 1-\sum_{j=0}^{k-1}  E_{\beta,\alpha}(j;s),
\end{equation}
where we briefly write
\begin{equation}\label{eq:ELE}
E_{\beta,\alpha}(k;s) = E_\beta^{(\text{hard})}(k;(0,s),\alpha).
\end{equation}
The integral representation
\begin{equation}\label{eq:KBesselFact}
K_{\alpha}(x,y) = \frac14 \int_0^1 J_\alpha\left(\sqrt{\xi x}\right)J_\alpha\left(\sqrt{\xi \smash[b]{y}}\right)\,d\xi
\end{equation}
of the Bessel kernel implies, by introducing the kernels
\begin{equation}\label{eq:VJ}
T_{s,\alpha}(x,y) = \frac{\sqrt{s}}{2}J_\alpha\left(\sqrt{s x \smash[b]{y}}\right), \qquad V_\alpha(x,y) = \frac12 J_\alpha\left(\sqrt{x\smash[b]{y}}\right),
\end{equation}
the factorization
\begin{multline}\label{eq:hardfact}
\det\left(I- z K_{\alpha}\projected{L^2(0,s)}\right) = \det\left(I- z\left(T_{s,\alpha}\projected{L^2(0,1)}\right)^2 \right)\\*[2mm]
=\det\left(I- \sqrt{z}\, T_{s,\alpha}\projected{L^2(0,1)}\right) \cdot \det\left(I+ \sqrt{z}\, T_{s,\alpha}\projected{L^2(0,1)}\right)\\*[2mm]
=\det\left(I- \sqrt{z}\, V_\alpha\projected{L^2(0,\sqrt{s\,})}\right) \cdot \det\left(I+ \sqrt{z}\, V_\alpha\projected{L^2(0,\sqrt{s\,})}\right)
\end{multline}
that is valid for the complex cut plane $z \in \C\setminus(-\infty,0]$. Now, upon introducing the functions
\begin{equation}\label{eq:Epmalpha}
E_{\pm,\alpha}(k;s) = \left. \frac{(-1)^k}{k!} \frac{d^k}{dz^k} \det\left(I\mp\sqrt{z}\,V_\alpha\projected{L^2(0,\sqrt{s\,})}\right) \right|_{z=1}
\end{equation}
we obtain from (\ref{eq:E2hard})
\begin{equation}\label{eq:E2HardSum}
E_{2,\alpha}(k;s) = \sum_{j=0}^k E_{+,\alpha}(j;s) E_{-,\alpha}(k-j;s).
\end{equation}
\citeasnoun[Prop.~1]{MR2229797} proved that
\begin{subequations}\label{eq:desfor}
\begin{align}
E_{1,\frac{\alpha-1}{2}}(0;s) &= E_{+,\alpha}(0;s) \label{eq:E10Hard}\\*[2mm]
E_{4,\alpha+1}(0;s) &= \tfrac{1}{2}(E_{+,\alpha}(0;s) + E_{-,\alpha}(0;s)).\label{eq:E40Hard}
\end{align}
\end{subequations}
The last generalizes to
\begin{equation}\label{eq:E4Hard}
E_{4,\alpha+1}(k;s) = \tfrac{1}{2}(E_{+,\alpha}(k;s) + E_{-,\alpha}(k;s))\qquad (k=0,1,2,\ldots),
\end{equation}
or, equivalently, to the generating function
\begin{multline}\label{eq:Gen4Hard}
\sum_{k=0}^\infty E_{4,\alpha+1}(k;s) (1-z)^k \\
= \frac{1}{2}\left(\det\left(I-\sqrt{z}\,V_\alpha\projected{L^2(0,\sqrt{s\,})}\right)+\det\left(I+\sqrt{z}\,V_\alpha\projected{L^2(0,\sqrt{s\,})}\right)\right)
\end{multline}
that holds for all $s \in (0,\infty)$ and $z$ in the complex domain of analyticity that contains $z=1$. A proof of (\ref{eq:Gen4Hard}), using Painlevé representations,
was given by \citeasnoun[Eq.~(1.38)]{MR2275509}.

It remains to discuss the LOE. To this end we use
the interrelationships
between LOE, LUE, and LSE found by \citeasnoun[Thms.~4.3/5.1]{MR1842786},\footnote{Again, that is why we have chosen the weight functions scaled as in \citeasnoun{MR1842786}.}
which can symbolically be written in the form
\begin{subequations}
\begin{align}
\text{LSE}_{n,\alpha+1} &= \text{even}\left(\text{LOE}_{2n+1,\frac{\alpha-1}{2}}\right), \label{eq:LSEfromeven}\\*[1mm]
\text{LUE}_{n,\alpha} &= \text{even}\left(\text{LOE}_{n,\frac{\alpha-1}{2}} \,\cup\, \text{LOE}_{n+1,\frac{\alpha-1}{2}}\right).\label{eq:LUEfromeven}
\end{align}
\end{subequations}
Here, we write $\text{LOE}_{n,\alpha}$ for the $n$-dimensional LOE with parameter $\alpha$, etc. Otherwise, these symbolic equations have the same meaning
as (\ref{eq:GSEfromeven}) and (\ref{eq:GUEfromeven}). For the same purely combinatorial reasons as in the preceding section we hence get 
\begin{equation}\label{eq:interlacing2}
F_{4,\alpha+1}(k;s) = F_{1,\frac{\alpha-1}{2}}(2k;s)\qquad (k=1,2,3,\ldots),
\end{equation}
or, equivalently by (\ref{eq:FLE}) and (\ref{eq:E4Hard}),
\begin{equation}\label{eq:func3hard}
 E_{1,\frac{\alpha-1}{2}}(2k;s) +  E_{1,\frac{\alpha-1}{2}}(2k+1;s) = \tfrac12( E_{+,\alpha}(k;s) +  E_{-,\alpha}(k;s)),
\end{equation}
as well as \citeaffixed[Cor.~4]{MR2275509}{see also}
\begin{multline}\label{eq:func4hard}
E_{2,\alpha}(k;s) \\
 = \sum_{j=0}^{2k} E_{1,\frac{\alpha-1}{2}}(j;s) E_{1,\frac{\alpha-1}{2}}(2k-j;s) + \sum_{j=0}^{2k+1}  E_{1,\frac{\alpha-1}{2}}(j;s)  E_{1,\frac{\alpha-1}{2}}(2k+1-j;s).
\end{multline}
Since Theorem~\ref{thm:algebra} was in fact just addressing the solution of a specific system of functional
equations, we obtain the same result here because the system (\ref{eq:E2HardSum}), (\ref{eq:E10Hard}), (\ref{eq:func3hard}), and (\ref{eq:func4hard}) possesses
exactly the algebraic structure considered there. That is, we get the solution
\begin{subequations}\label{eq:E1Hardrecursion}
\begin{align}
E_{1,\frac{\alpha-1}{2}}(2k;s) &= E_{+,\alpha}(k;s) - \sum_{j=0}^{k-1} \frac{\binom{2j}{j}}{2^{2j+1}(j+1)}  E_{1,\frac{\alpha-1}{2}}(2k-2j-1;s),\\*[2mm]
E_{1,\frac{\alpha-1}{2}}(2k+1;s) &= \frac{ E_{+,\alpha}(k;s)+ E_{-,\alpha}(k;s)}{2} -  E_{1,\frac{\alpha-1}{2}}(2k;s).
\end{align}
\end{subequations}
and, completely parallel to (\ref{eq:D1Formel}), the corresponding generating function
\begin{multline}
\sum_{k=0}^\infty E_{1,\frac{\alpha-1}{2}}(k;s)(1-z)^k = \frac{1}{2}\left( \det\left(I -\sqrt{z(2-z)} \,V_\alpha\projected{L^2(0,\sqrt{s}\,)} \right)\left(1+\sqrt\frac{z}{2-z}\,\right)\right.\\*[2mm]
\left. + \det\left(I +\sqrt{z(2-z)} \,V_\alpha\projected{L^2(0,\sqrt{s}\,)} \right)\left(1-\sqrt\frac{z}{2-z}\,\right)\right)
\end{multline}
for all $s\in (0,\infty)$ and $z$ in the complex domain of analyticity that contains $z=1$.

\begin{remark}
Given the large order asymptotics \cite[Eq.~(9.5.01)]{MR0435697}
\begin{equation}
J_\alpha(\alpha + \zeta \alpha^{1/3}) = 2^{1/3} \alpha^{-1/3} \Ai(-2^{1/3} \zeta) + O(\alpha^{-1}) \qquad(\alpha\to\infty),
\end{equation}
which holds uniformly for bounded $\zeta$, we get, by changing variables subject to
\begin{equation}
x = \sqrt{\alpha^2 - 2\alpha (\alpha/2)^{1/3} \xi},\qquad  y = \sqrt{\alpha^2 - 2\alpha (\alpha/2)^{1/3} \eta},
\end{equation}
the kernel approximation
\begin{equation}
V_\alpha(x,y) \frac{dy}{d\eta} = - V_{\Ai}(\xi,\eta) + O(\alpha^{-2/3})\qquad(\alpha\to\infty),
\end{equation}
uniformly for bounded $\xi,\eta$. Since $y=0$ is mapped to $\eta = (\alpha/2)^{2/3} \to \infty$ ($\alpha\to \infty$) we therefore obtain, observing the fast decay of the Bessel and Airy functions,
the operator approximation (in trace class norm)
\begin{equation}
V_\alpha \projected{L^2\left(0,\sqrt{\alpha^2 - 2\alpha (\alpha/2)^{1/3} s}\,\right)} =  V_{\Ai}\projected{L^2(s,\infty)} + O(\alpha^{-2/3}) \qquad(\alpha\to\infty).
\end{equation}
This approximation implies the limit of the associated Fredholm determinants,
\begin{equation}
\lim_{\alpha\to \infty} \det\left(I-zV_\alpha \projected{L^2\left(0,\sqrt{\alpha^2 - 2\alpha (\alpha/2)^{1/3} s}\,\right)} \right) = \det\left(I-zV_{\Ai}\projected{L^2(s,\infty)}\right),
\end{equation}
or, equivalently,
\begin{equation}
\lim_{\alpha\to \infty} E_{\pm,\alpha}(k;\alpha^2 - 2\alpha (\alpha/2)^{1/3} s) = \tilde E_\pm(k;s) \qquad  (k=0,1,2,\ldots).
\end{equation}
Plugging this limit into (\ref{eq:E2HardSum}), (\ref{eq:E4Hard}), and (\ref{eq:E1Hardrecursion}) yields, by (\ref{eq:E2EdgeSum}), (\ref{eq:E4Edge}), and (\ref{eq:E1recursion}), the
hard-to-soft transition \citeaffixed[§4]{MR1986402}{see also}
\begin{equation}
\lim_{\alpha\to\infty} E_{\beta,\alpha(\beta)}(k;\alpha^2 - 2\alpha (\alpha/2)^{1/3} s) = \tilde E_\beta(k;s)
\end{equation}
with
\begin{equation}
\alpha(\beta) = \begin{cases}
(\alpha-1)/2, &\; \beta = 1,\\*[1mm]
\alpha, &\; \beta=2,\\*[1mm]
\alpha+1, &\; \beta = 4.
\end{cases}
\end{equation}
This transition is a further convenient mean to validate our numerical methods.
\end{remark}

\section{Matrix Kernels and Examples of Joint Probability Distributions}\label{sect:joint}

\subsection{Matrix Kernels}\label{sect:matrixkern}

 \citeasnoun[§8.1]{Bornemann1} showed that the quadrature based approach to the numerical approximation of Fredholm determinants that we described in Section~\ref{sect:basic}, can fairly easily be extended to
matrix kernel determinants of the form
\begin{equation}
d(z) = \det\left(I - z
\begin{pmatrix}
K_{11} & \cdots & K_{1N} \\
\vdots & & \vdots \\
K_{N1} & \cdots & K_{NN}
\end{pmatrix}
\projected{L^2(J_1)\oplus\cdots\oplus L^2(J_N)}\right),
\end{equation}
where $J_1,\ldots,J_N$ are open intervals and the smooth matrix kernel generates a trace class operator on $L^2(J_1)\oplus\cdots\oplus L^2(J_N)$. By taking an
$m$-point quadrature rule  of order $m$ with nodes $x_{ij}\in J_i$ and \emph{positive} weights $w_{ij}$, written in the form
\begin{equation}
\sum_{j=1}^m w_{ij} f(x_{ij}) \approx \int_{J_i} f(x)\,dx \qquad(i=1,\ldots,N),
\end{equation}
we approximate $d(z)$ by the $N\cdot m$-dimensional determinant
\begin{subequations}\label{eq:detfinN}
\begin{equation}
d_m(z) = \det\left(I - z
\begin{pmatrix}
A_{11} & \cdots & A_{1N} \\
\vdots & & \vdots \\
A_{N1} & \cdots & A_{NN}
\end{pmatrix}\right)
\end{equation}
with block entries $A_{ij}$ that are the $m\times m$ matrices given by
\begin{equation}
(A_{ij})_{p,q} = w_{ip}^{1/2} K_{ij}(x_{ip},x_{jq}) w_{jq}^{1/2}\qquad (p=1,\ldots,m;\,q=1,\ldots, m).
\end{equation}
\end{subequations}
If the kernel functions $K_{ij}(x,y)$ are analytic in a complex neighborhood of $J_i \times J_j$, one can prove exponential convergence \cite[Thm.~8.1]{Bornemann1}: there
is a constant $\rho>1$ such that
\begin{equation}
d_m(z) - d(z) = O(\rho^{-m})\qquad (m\to \infty),
\end{equation}
locally uniform in $z \in \C$. The results of Section~\ref{sect:finitedet}--\ref{sect:dens} apply then verbatim. This approach was used to evaluate
the matrix kernel determinant (\ref{eq:GSEMatrixKernelDet}) for the numerical checks of the fundamental equation (\ref{eq:conjecture}) in Section~\ref{sect:edge}.

\subsubsection{An example: the joint probability distribution of GUE matrix diffusion}

\citeasnoun{MR1933446} proved that the joint probability of the maximum eigenvalue of GUE matrix diffusion at two different times is given, in the soft edge scaling limit, by the
operator determinant
\begin{subequations}\label{eq:airyprocess}
\begin{equation}
\prob(\mathcal{A}_2(t)\leq s_1, \mathcal{A}_2(0) \leq s_2) = \det\left(I -
\begin{pmatrix}
K_0 & K_t \\*[1mm]
K_{-t} & K_0
\end{pmatrix}{\projected{L^2(s_1,\infty)\oplus L^2(s_2,\infty)}}\right)
\end{equation}
with kernel\vspace*{-2mm}
\begin{equation}\label{eq:Kt}
K_t(x,y) = \begin{cases}
\phantom{-}\displaystyle\int_0^\infty e^{-\xi t} \Ai(x+\xi)\Ai(y+\xi)\,d\xi &\qquad (t\geq 0), \\*[4mm]
-\displaystyle\int_{-\infty}^0 e^{-\xi t} \Ai(x+\xi)\Ai(y+\xi)\,d\xi &\qquad (t< 0).
\end{cases}
\end{equation}
\end{subequations}
(Note that $K_0=K_{\Ai}$, see (\ref{eq:KairyFact}).) This expression is directly amenable to our numerical methods; Figure~\ref{fig:joint}.a shows the covariance function that was calculated this way.
\begin{figure}[tbp]
\begin{center}
\begin{minipage}{0.478\textwidth}
\begin{center}
\includegraphics[width=\textwidth]{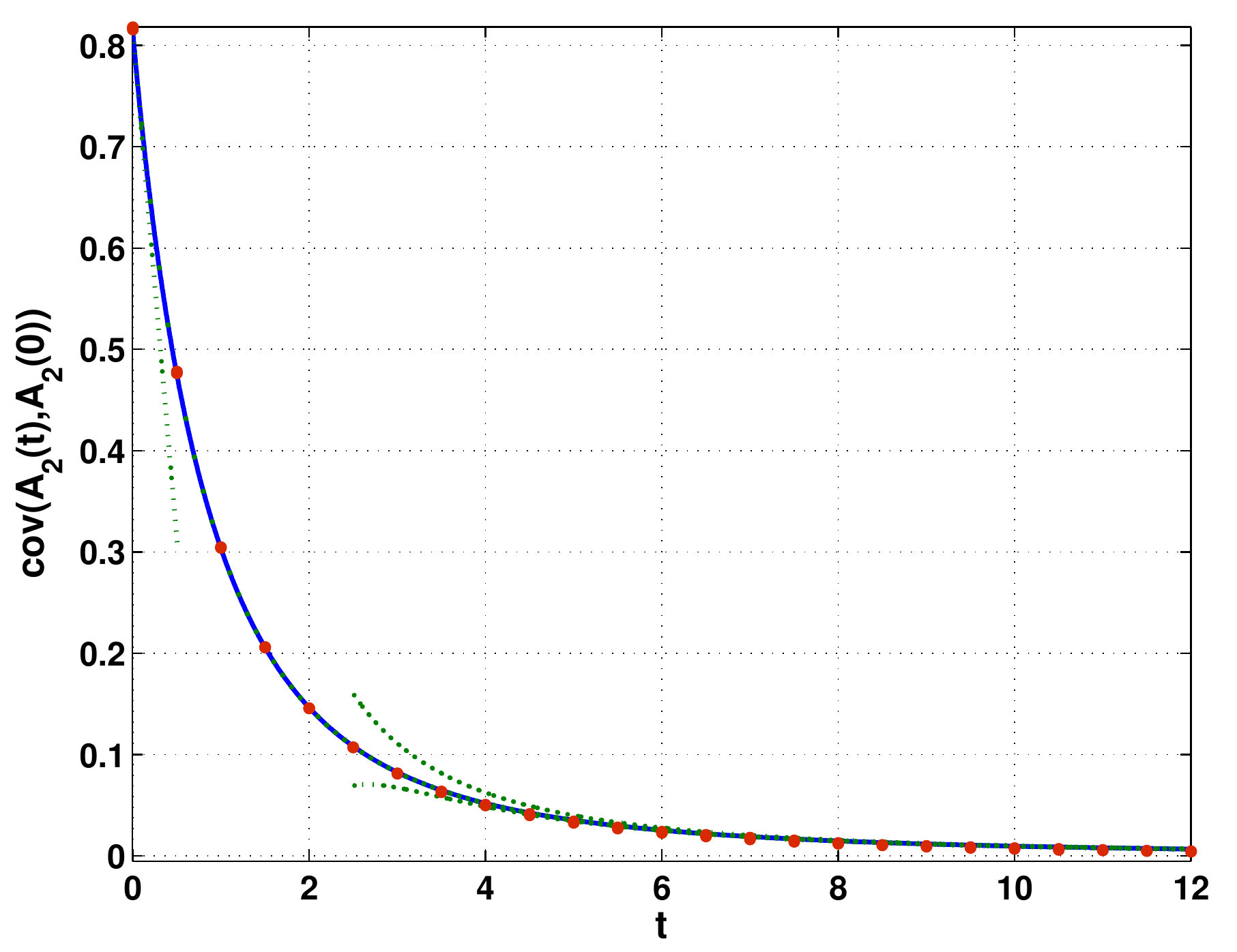}\\*[-1mm]
{\footnotesize a. \, $\cov(\mathcal{A}_2(t),\mathcal{A}_2(0))$ for GUE matrix diffusion}
\end{center}
\end{minipage}
\hfil
\begin{minipage}{0.502\textwidth}
\vspace*{0.8mm}
\begin{center}
\includegraphics[width=\textwidth]{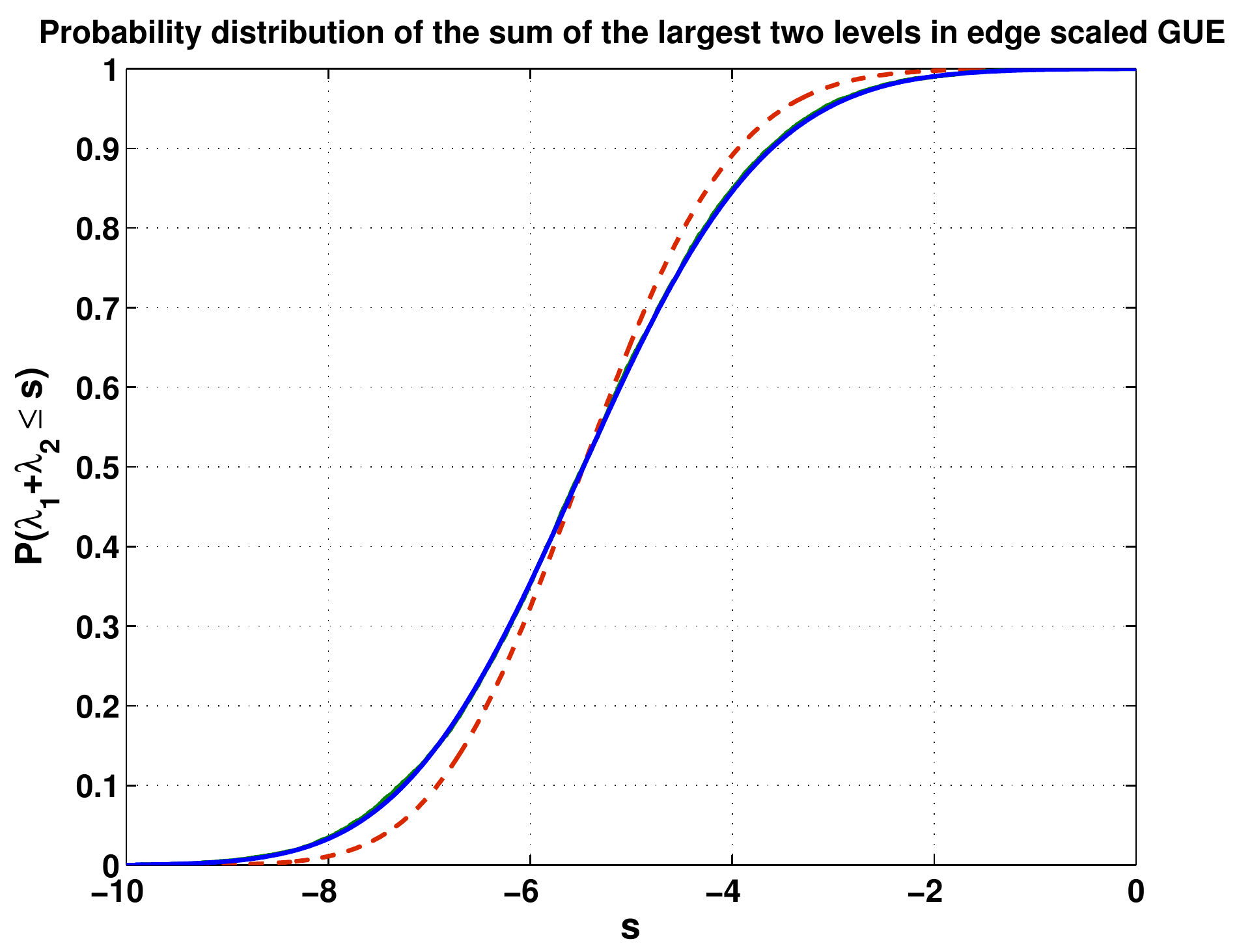}\\*[-1mm]
{\footnotesize b. \, CDF of the sum of largest two levels in GUE}
\end{center}
\end{minipage}
\end{center}\vspace*{-0.0625cm}
\caption{Left: plot (blue line) of the covariance $\cov(\mathcal{A}_2(t),\mathcal{A}_2(0))$ of the maximum eigenvalue of GUE matrix diffusion
at two different times (soft edge scaling limit), calculated from the joint probability distribution (\ref{eq:airyprocess}) to 10 digits accuracy (absolute tolerance $5\cdot 10^{-11}$).
The red dots show the data obtained from a Monte Carlo simulation with matrix dimensions $m=128$ and $m=256$ (there is no difference visible between the
two dimensions on the level of plotting accuracy). The dashed, green lines show the asymptotic expansions (\ref{eq:airy2asympt1}) and (\ref{eq:airy2asympt2}).
Right: plot (blue line) of the CDF (\ref{eq:CFDsum}) of the sum of the largest two levels of GUE (soft edge scaling limit), calculated
from the joint probability distribution (\ref{eq:twojoint}) to 10 digits accuracy (absolute tolerance $5\cdot 10^{-11}$). To compare with, we also show (red, dashed line)
the convolution (\ref{eq:CFDconv}) of the corresponding individual CDFs from Figure~\ref{fig:GUEexamples}.b; because of statistical dependence, there is a clearly visible difference.
}
\label{fig:joint}
\end{figure}
The results were cross-checked with a Monte Carlo simulation (red dots), and with the help of the following asymptotic expansions (dashed lines): for small $t$
with the expansion \cite{MR1933446,Hagg07}
\begin{equation}\label{eq:airy2asympt1}
\cov(\mathcal{A}_2(t),\mathcal{A}_2(0)) = \var(F_2) -t + O(t^2)\qquad (t\to 0),
\end{equation}
where $F_2$ denotes the Tracy--Widom distribution for GUE (the numerical value of $\var(F_2)$ can be found in Table~\ref{tab:LUE}); for large $t$
with the expansion
 \cite{MR2054175,MR2150191}
\begin{equation}\label{eq:airy2asympt2}
\cov(\mathcal{A}_2(t),\mathcal{A}_2(0)) = t^{-2} +  c t^{-4} + O(t^{-6})\qquad (t\to\infty),
\end{equation}
where the constant $c=-3.542\cdots$ can explicitly be expressed in terms of the Hastings--McLeod solution (\ref{eq:HM}) of Painlevé II.

Similar calculations related to GOE matrix diffusion, and their impact on ``experimentally disproving'' a conjectured determinantal formula, are discussed in a
recent paper by \citeasnoun{MR2448629}.

\subsubsection*{Remark}
In a masterful analytic study, \citeasnoun{MR2150191} proved that the function
$G(t,x,y) = \log\mathbb{P}(\mathcal{A}_2(t)\leq x, \mathcal{A}_2(0) \leq y)$ satisfies the following nonlinear 3rd order partial differential
equation (together with certain asymptotic boundary conditions):
\begin{multline}
t \frac{\partial}{\partial t}\left(\frac{\partial^2}{\partial x^2}- \frac{\partial^2}{\partial y^2}\right) G =
\frac{\partial^3 G}{\partial x^2\partial y}\left(2\frac{\partial^2 G}{\partial y^2}+ \frac{\partial^2 G}{\partial x\partial y}-\frac{\partial^2 G}{\partial x^2}+x-y-t^2\right)\\*[2mm]
-\frac{\partial^3 G}{\partial y^2\partial x}\left(2\frac{\partial^2 G}{\partial x^2}+ \frac{\partial^2 G}{\partial x\partial y}-\frac{\partial^2 G}{\partial y^2}-x+y-t^2\right)\\*[2mm]
+\left(\frac{\partial^3 G}{\partial x^3}\frac{\partial}{\partial y}-\frac{\partial^3 G}{\partial y^3}\frac{\partial}{\partial x}\right)
\left(\frac{\partial}{\partial x}+\frac{\partial}{\partial y}\right) G.
\end{multline}
The reader should contemplate a numerical calculation of the covariance function based on this PDE, rather than directly treating the Fredholm determinant
(\ref{eq:airyprocess}) as suggested in this paper.

\subsection{Operators Acting on Unions of Intervals}\label{sect:union}

Determinants of integral operators $K$ acting on the space $L^2(J_1 \cup \cdots \cup J_N)$ of functions defined on a union of mutually {\em disjoint} open intervals can be dealt with
by transforming them to a matrix kernel determinant \cite[Thm.~VI.6.1]{MR1744872}:
\begin{multline}\label{eq:union1}
\det\left(I- z K\projected{L^2(J_1\cup \cdots \cup J_N)}\right) \\*[2mm]
=
\det\left(I - z
\begin{pmatrix}
K_{11} & \cdots & K_{1N} \\
\vdots & & \vdots \\
K_{N1} & \cdots & K_{NN}
\end{pmatrix}
\projected{L^2(J_1)\oplus\cdots\oplus L^2(J_N)}\right),
\end{multline}
where $K_{ij} : L^2(J_j) \to L^2(J_i)$ denotes the integral operator induced by the given kernel function $K(x,y)$, that is,
\begin{equation}\label{eq:Kij}
K_{ij} u(x) = \int_{J_j} K(x,y) u(y)\,dy\qquad (x\in J_i).
\end{equation}
More general, along the same lines, with $\chi_J$ denoting the characteristic function of an interval $J$ and $z_j \in \C$, we have
\begin{multline}\label{eq:union2}
\det\left(I- \left(K \sum_{j=1}^N z_j \chi_{J_j} \right)\projected{L^2(\R)}\right) \\*[2mm]
 =
\det\left(I -
\begin{pmatrix}
z_1 K_{11} & \cdots & z_N K_{1N} \\
\vdots & & \vdots \\
z_1 K_{N1} & \cdots & z_N K_{NN}
\end{pmatrix}\projected{L^2(J_1)\oplus\cdots\oplus L^2(J_N)}\right),
\end{multline}
where the $K_{ij} : L^2(J_j) \to L^2(J_i)$ denote, once more, the integral operators defined in~(\ref{eq:Kij}).
To indicate the fact
that the operators $K_{ij}$ share one and the same kernel function $K(x,y)$ we also write, by ``abus d'langage'',
\begin{equation}
\det\left(I - z
\begin{pmatrix}
K & \cdots & K \\
\vdots & & \vdots \\
K & \cdots & K
\end{pmatrix}\projected{L^2(J_1)\oplus\cdots\oplus L^2(J_N)}\right)
\end{equation}
for (\ref{eq:union1}) and
\begin{equation}
\det\left(I -
\begin{pmatrix}
z_1 K & \cdots & z_N K \\
\vdots & & \vdots \\
z_1 K & \cdots & z_N K
\end{pmatrix}\projected{L^2(J_1)\oplus\cdots\oplus L^2(J_N)}\right)
\end{equation}
for (\ref{eq:union2}).
Clearly, in both of the cases (\ref{eq:union1}) and (\ref{eq:union2}), we then apply the method of Section~\ref{sect:matrixkern} for the
numerical evaluation of the equivalent matrix kernel determinant.

\subsection{Generalized Spacing Functions}\label{sect:general}

The determinantal formulae (\ref{eq:detGUEn}), (\ref{eq:E2bulk}), (\ref{eq:E2Edge}), (\ref{eq:GSEMatrixKernelDet}), (\ref{eq:detLUEn}), and (\ref{eq:E2hard})
have all the common form
\begin{equation}\label{eq:detEgen}
E(k;J) = \prob(\text{exactly $k$ levels lie in $J$}) = \frac{(-1)^k}{k!} \left. \frac{d^k}{dz^k} \det\left(I-z\,K\projected{L^2(J)}\right)\right|_{z=1},
\end{equation}
which is based on an underlying combinatorial structure that can be extended to describe, for a multi-index $\alpha \in \N_0^N$ and mutually \emph{disjoint} open intervals $J_j$ ($j=1,\ldots,N$),
the generalized spacing function
\begin{subequations}\label{eq:spacinggeneral}
\begin{equation}
E(\alpha;J_1,\ldots,J_N) = \prob(\text{exactly $\alpha_j$ levels lie in $J_j$, $j=1,\ldots,N$})
\end{equation}
by the determinantal formula (see  \citename{MR1253763} \citeyear{MR1253763}, Thm.~6, or \citeyear{MR1657844}, Eq.~(4.1))
\begin{multline}
E(\alpha;J_1,\ldots,J_N) \\*[1mm]
= \frac{(-1)^{|\alpha|}}{\alpha!} \left. \frac{\partial^\alpha}{\partial z^\alpha} \det\left(I - \left(K \sum_{j=1}^N z_j \chi_{J_j}\right)\projected{L^2(\R)}\right)\right|_{z_1=\cdots=z_N=1}.
\end{multline}
\end{subequations}
By the results of Section~\ref{sect:union} and \ref{sect:highder}, such an expression is, in principle at least, amenable to the numerical methods of this paper.
However, differentiation with respect to $l$ different variables $z_j$ means to evaluate an $l$-dimensional Cauchy integral of a function that is calculated from
approximating a Fredholm determinant. This can quickly become very expensive, indeed.

\subsubsection{Efficient numerical evaluation of the finite-dimensional determinants}
Let us briefly discuss the way one would adapt the methods of Section~\ref{sect:finitedet} to the current situation.
We will confine ourselves to the important case of a multi-index $\alpha$ which has the form $\alpha = (k,0,\ldots,0)$. If we are to compute the derivatives
by a Cauchy integral, we have to evaluate
a determinant of the form\footnote{Which is understood to be an approximation of the determinant in (\ref{eq:spacinggeneral}), expressed as the matrix kernel determinant (\ref{eq:union2}).}
\begin{equation}
d(z) = \det(I - (zA \,|\, B)) \qquad A \in \R^{m,p}, B \in \R^{m,m-p}
\end{equation}
for several different arguments $z \in \C$. By putting
\begin{equation}
T = I - (0 \,|\, B),\qquad \tilde A = T^{-1} (A \,|\, 0),
\end{equation}
we have $I - (zA \,|\, B) = T (I - z \tilde A)$
and hence the factorization
\begin{equation}
d(z) = \det(T) \cdot \det(I-z \tilde A)
\end{equation}
to which the results of Section~\ref{sect:finitedet} apply verbatim.

\subsubsection{An example: the joint distribution of the largest two eigenvalues in GUE}

Let us denote by $\lambda_1$ the largest and by $\lambda_2$ the second largest level of GUE in the soft edge scaling limit (\ref{eq:GEedge}). We want to evaluate
their joint probability distribution function
\begin{equation}
F(x,y) = \prob(\lambda_1 \leq x, \lambda_2 \leq y).
\end{equation}
Since $\lambda_1 \leq x \leq y$ certainly implies $\lambda_2 \leq y$, we have
\begin{equation}
F(x,y) = \prob(\lambda_1 \leq x) = F_2(x) \qquad (x \leq y).
\end{equation}
On the other hand, if $x > y$, the open intervals $(y,x)$ and $(x,\infty)$
are disjoint and we obtain for simple combinatorial reasons
\begin{multline}
F(x,y) = E_2^{(\text{soft})}(0,0;(y,x),(x,\infty)) + E_2^{(\text{soft})}(1,0;(y,x),(x,\infty))\\*[1mm]
= E_2^{(\text{soft})}(0;(y,\infty)) + E_2^{(\text{soft})}(1,0;(y,x),(x,\infty))\\*[1mm]
= F_2(y) + E_2^{(\text{soft})}(1,0;(y,x),(x,\infty)) \qquad (x > y).
\end{multline}
By (\ref{eq:spacinggeneral}) and (\ref{eq:union2}) we finally get
\begin{multline}\label{eq:twojoint}
F(x,y)  \\*[1mm]
= \begin{cases}
F_2(x) & \;\, (x \leq y),\\*[2mm]
F_2(y) - \left.\dfrac{\partial}{\partial z} \det\left(I-
\begin{pmatrix}
z \,K_{\Ai} & K_{\Ai} \\
z \,K_{\Ai} & K_{\Ai}
\end{pmatrix}\projected{L^2(y,x)\oplus L^2(x,\infty)}\right)\right|_{z=1} & \;\, (x > y).
\end{cases}
\end{multline}
We have used this formula to calculate the correlation of the largest two levels as
\begin{equation}
\rho(\lambda_1,\lambda_2) = 0.50564\,72315\,9\ldots;
\end{equation}
where all the 11 digits shown are estimated to be correct and the run time was about 16 hours using hardware arithmetic.\footnote{So, this example
stretches our numerical methods pretty much to the edge of what is possible. Note, however, that these numerical results are completely out of the reach of a representation
by partial differential equations.}  So, as certainly was to be expected, the largest two levels
are statistically very much \emph{dependent}. Figure~\ref{fig:joint}.b plots the CDF of the sum of the largest two levels,
\begin{equation}\label{eq:CFDsum}
\prob(\lambda_1+\lambda_2 \leq s) = \int_{-\infty}^\infty \partial_1 F(x,s-x)\,dx,
\end{equation}
and compares the result with the convolution of the individual CDFs of these levels, that is with the $s$-dependent function
\begin{equation}\label{eq:CFDconv}
\int_{-\infty}^\infty F_2'(1;x) F_2(2;s-x)\,dx.
\end{equation}
The clearly visible difference between those two functions is a further corollary of the statistical dependence of the largest two level.

Yet another calculation of joint probabilities in the spectrum of GUE (namely, related to the statistical {\em independence} of the extreme eigenvalues) can be found in \citeasnoun{Bornemann09}.

\section{Software}\label{sect:software}

We have coded the numerical methods of this paper in a Matlab toolbox. (For the time being, it can be obtained from the author upon request by e-mail. At a later stage
it will be made freely available at the web.) In this section we explain the design and use of the toolbox.

\subsection{Low Level Commands}

\subsubsection{Quadrature rule}

The command

\smallskip
{\small\begin{verbatim}
>> [w,x] = ClenshawCurtis(a,b,m)
\end{verbatim}}
\smallskip

\noindent
calculates the $m$-point Clenshaw--Curtis quadrature rules (suitably transformed if $a=-\infty$ or $b=\infty$).
The result is a row vector {\tt w} of weights and a column vector {\tt x} of nodes. This way, the application (\ref{eq:quadrature}) of
the quadrature rule to a (vectorizing) function ${\tt f}$ goes simply by the following command:

\smallskip
{\small\begin{verbatim}
>> w*f(x)
\end{verbatim}}
\smallskip

\noindent
Once called for a specific number $m$, the (untransformed) weights and nodes are cached for later use. As an alternative the toolbox also offers Gauss--Jacobi quadrature,
see Section~\ref{sect:singularity} for its use in the context of algebraic kernel singularities.

\begin{table}[tbp]
\caption{Toolbox commands for kernel functions $K(x,y)$. If the value of $K(x,y)$ is defined as an
integral, there is an additional argument {\tt m} to the command that assigns the number of quadrature points to be used.}
\vspace*{0mm}
\centerline{%
\setlength{\extrarowheight}{5pt}
{\small\begin{tabular}{lllc}\hline
kernel & formula & command & vectorization mode\\*[0.5mm]\hline
 $K_{\sin}(x,y)$     & (\ref{eq:Ksin})       & {\tt sinc(pi*(x-y))}              & {\tt 'grid'} \\
 $V_{\Ai}(x,y)$      & (\ref{eq:K1})         & {\tt airy((x+y)/2)/2}             & {\tt 'grid'} \\
 $V_\alpha(x,y)$     & (\ref{eq:VJ})         & {\tt besselj(alpha,sqrt(x.*y))/2} & {\tt 'grid}  \\*[0.5mm]\hline
 $K_n(x,y)$          & (\ref{eq:Kn})         & {\tt HermiteKernel(n,x,y)}        & {\tt 'outer'}\\
 $K_{\Ai}(x,y)$      & (\ref{eq:KAi})        & {\tt AiryKernel(x,y)}             & {\tt 'outer'}\\*[0.5mm]\hline
 $K_{n,\alpha}(x,y)$ & (\ref{eq:Kna})        & {\tt LaguerreKernel(n,alpha,x,y)} & {\tt 'outer'}\\
 $K_\alpha(x,y)$     & (\ref{eq:Ka})         & {\tt BesselKernel(alpha,x,y)}     & {\tt 'outer'}\\*[0.5mm]\hline
 $S(x,y)$            & (\ref{eq:S})          & {\tt F4MatrixKernel(x,y,m,'SN')}  & {\tt 'outer'}\\
 $S^*(x,y)$          & (\ref{eq:S})          & {\tt F4MatrixKernel(x,y,m,'ST')}  & {\tt 'outer'}\\
 $SD(x,y)$           & (\ref{eq:SD})         & {\tt F4MatrixKernel(x,y,m,'SD')}  & {\tt 'outer'}\\
 $IS(x,y)$           & (\ref{eq:IS})         & {\tt F4MatrixKernel(x,y,m,'IS')}  & {\tt 'outer'}\\*[0.5mm]\hline
 $K_t(x,y)$          & (\ref{eq:Kt})         & {\tt Airy2ProcessKernel(t,x,y,m)} & {\tt 'outer'}\\*[0.5mm]\hline
\end{tabular}}}
\label{tab:kernels}
\end{table}

\subsubsection{Kernels and vectorization modes} The numerical approximation of Fredholm determinants, by (\ref{eq:detfin}) or (\ref{eq:detfinN}),
requires the ability to build, for given $m$-dimensional column vectors $x$ and $y$, the $m\times m$ matrix
\[
A = (K(x_i,y_j))_{i,j=1}^m.
\]
For reasons of efficiency we make a meticulous use of Matlab's vectorization capabilities. Depending on the specific structure of the coding of the kernel
function, we distinguish between two vectorization modes:

\begin{enumerate}
\item {\tt 'grid'}: The matrix $A$ is built from the vectors {\tt x}, {\tt y} and the kernel function {\tt K} by the commands

\smallskip
{\small\begin{verbatim}
>> [X,Y] = ndgrid(x,y);
>> A = K(X,Y);
\end{verbatim}}
\smallskip

\item {\tt 'outer'}: The matrix $A$ is built from the vectors {\tt x}, {\tt y} and the kernel function {\tt K} by the commands

\smallskip
{\small\begin{verbatim}
>> A = K(x,y);
\end{verbatim}}
\smallskip

\end{enumerate}

Table~\ref{tab:kernels} gives the commands for all the kernels used in this paper together with their vectorization modes.

\subsubsection{Approximation of Fredholm determinants}
Having built the matrix $A$ the approximation (\ref{eq:detfin}) is finally evaluated by the following commands:

\smallskip
{\small\begin{verbatim}
>> w2 = sqrt(w);
>> det(eye(size(A))-z*(w2'*w2).*A)
\end{verbatim}}

\subsubsection{Example}\label{sect:softex1}

Let us evaluate the values $F_1(0)$ and $F_2(0)$ of the Tracy--Widom distributions for GOE and GUE by these low level commands
using $m=64$ quadrature points. The reader should observe the different vectorization modes:

\smallskip
{\small\begin{verbatim}
>> m = 64; [w,x] = ClenshawCurtis(0,inf,m); w2 = sqrt(w);
>> [xi,xj] = ndgrid(x,x);
>> K1 = @(x,y) airy((x+y)/2)/2;
>> F10 = det(eye(m)-(w2'*w2).*K1(xi,xj))

   F10 = 0.831908066202953

>> KAi = @AiryKernel;
>> F20 = det(eye(m)-(w2'*w2).*KAi(x,x))

   F20 = 0.969372828355262
\end{verbatim}}
\smallskip

\noindent
A look into \possessivecite{Praehofer03} tables teaches that both results are correct to one unit of the last decimal place.

\subsection{Medium Level Commands}

The number of quadrature points can be hidden from the user by means of the automatic error control of Section~\ref{sect:error}.
This means, we start thinking in terms of the limit of the approximation sequence, that is, in terms of the corresponding integral
operators. This way, we evaluate the operator determinant
\[
\det(I - z K\projected{L^2(J)})
\]
for a given kernel function $K(x,y)$ by the following commands:

\smallskip
{\small\begin{verbatim}
>> K.ker = @(x,y) K(x,y); k.mode = vectorizationmode;
>> Kop = op(K,J);
>> [val,err] = det1m(Kop,z);
\end{verbatim}}
\smallskip

\noindent
(The argument {\tt z} may be omitted in {\tt det1m} if $z=1$.) Here, {\tt val} gives the value of the operator determinant and {\tt err}
is a conservative error estimate. The code tries to observe a given absolute tolerance {\tt tol} that can be set by

\smallskip
{\small\begin{verbatim}
>> pref('tol',tol);
\end{verbatim}}
\smallskip

\noindent
The default is $5\cdot 10^{-15}$, that is, {\tt tol = 5e-15}. The results can be nicely printed in a way such that, within the given error, either
just the correctly \emph{rounded} decimal places are displayed ({\tt printmode = 'round'}), or the correctly \emph{truncated} places ({\tt printmode = 'trunc'}):

\smallskip
{\small\begin{verbatim}
>> pref('printmode',printmode);
>> PrintCorrectDigits(val,err);
\end{verbatim}}
\smallskip

\noindent
For an integral operator $K$, the expressions
\begin{multline*}
\left.\frac{(-1)^k}{k!} \frac{d^k}{dz^k} \det(I-z K)\right|_{z=1},\quad\left.\frac{(-1)^k}{k!} \frac{d^k}{dz^k} \det(I-\sqrt{z}K)\right|_{z=1},\\*[1mm]
\left.\frac{(-1)^k}{k!} \frac{d^k}{dz^k}  \sqrt{\det(I-z K)}\right|_{z=1},
\end{multline*}
are then evaluated, with error estimate, by the following commands:

\smallskip
{\small\begin{verbatim}
>> [val,err] = dzdet(K,k);
>> [val,err] = dzdet(K,k,@sqrt);
>> [val,err] = dzsqrtdet(K,k);
\end{verbatim}}

\subsubsection{Example~\protect\ref{sect:softex1} revisited}\label{sect:softex2}
Let us now evaluate the values $F_1(0)$ and $F_2(0)$ of the Tracy--Widom distributions for GOE and GUE by these medium level commands.

\smallskip
{\small\begin{verbatim}
>> pref('printmode','trunc');
>> K1.ker = @(x,y) airy((x+y)/2)/2; K1.mode = 'grid';
>> [val,err] = det1m(op(K1,[0,inf]));
>> PrintCorrectDigits(val,err);

   0.83190806620295_

>> KAi.ker = @AiryKernel; KAi.mode = 'outer';
>> [val,err] = det1m(op(KAi,[0,inf]));
>> PrintCorrectDigits(val,err);

   0.96937282835526_
\end{verbatim}}
\smallskip

\noindent
So, the automatic error control supposes 14 digits to be correct in both cases;
a look into  \possessivecite{Praehofer03} tables teaches us that this is true, indeed.

\subsubsection{Example: an instance of equation (\ref{eq:conjecture})}\label{sect:instance}

We now give the line of commands that we used to experimentally check the truth of the determinantal equation (\ref{eq:conjecture}) before we worked out the proof. For a single instance
of a real value of $s$ and a complex value of $z$ we obtain:

\smallskip
{\small
\begin{verbatim}
>> s = -1.23456789; z = -3.1415926535 + 2.7182818284i;
>>
>> K11.ker = @(x,y,m) F4MatrixKernel(x,y,m,'SN'); K11.mode = 'outer';
>> K12.ker = @(x,y,m) F4MatrixKernel(x,y,m,'SD'); K12.mode = 'outer';
>> K21.ker = @(x,y,m) F4MatrixKernel(x,y,m,'IS'); K21.mode = 'outer';
>> K22.ker = @(x,y,m) F4MatrixKernel(x,y,m,'ST'); K22.mode = 'outer';
>> K = op({K11 K12; K21 K22},{[s,inf],[s,inf]});
>> val1 = sqrt(det1m(K,z))

   val1 = 1.08629916321436 -    0.0746712169305511i

>> K1.ker = @(x,y) airy((x+y)/2)/2; K1.mode = 'grid';
>> K = op(K1,[s,inf]);
>> val2 = (det1m(K,sqrt(z)) + det1m(K,-sqrt(z)))/2

   val2 = 1.08629916321436 -    0.0746712169305508i

>> dev = abs(val1-val2)

   dev = 5.23691153334427e-016
\end{verbatim}}
\smallskip

\noindent
This deviation is below the default tolerance $5\cdot 10^{-15}$ which was used for the calculation.

\subsection{High Level Commands}
Using the low and medium level commands we straightforwardly coded all the functions that we have discussed in this paper so far. Table~\ref{tab:functions} lists
the corresponding commands. The reader is encouraged to look into the actual code of these commands to see how closely we followed the determinantal formulae of this paper.

\begin{table}[tbp]
\caption{Toolbox commands for all the probability distributions of this paper. A call by {\tt val = E(...)} etc. gives the value; with {\tt [val,err] = E(...)}
etc. we get the value and a conservative error estimate.}
\vspace*{0mm}
\centerline{%
\setlength{\extrarowheight}{5pt}
{\small\begin{tabular}{lll}\hline
function & defining formulae & command\\*[0.5mm]\hline
interval $J=(s_1,s_2)$               &                                                                                             & {\tt J = [s1,s2]}                    \\
interval $J=(s,\infty)$              &                                                                                             & {\tt J = [s,inf]}                    \\
interval $J=(-\infty,s)$             &                                                                                             & {\tt J = [-inf,s]}                   \\*[0.5mm]\hline
$E_2^{(n)}(k;J)$                     &  (\ref{eq:detGUEn})                                                                         & {\tt E(2,k,J,n)}                     \\
$E_2^{(n)}((k,0);J_1,J_2)$           &  (\ref{eq:spacinggeneral})                                                                  & {\tt E(2,[k,0],\{J1,J2\},n)}         \\
$E_2^{(\text{bulk})}(k;J)$           &  (\ref{eq:E2bulk})                                                                          & {\tt E(2,k,J,'bulk')}                \\
$E_2^{(\text{bulk})}((k,0);J_1,J_2)$ &  (\ref{eq:spacinggeneral})                                                                  & {\tt E(2,[k,0],\{J1,J2\},'bulk')}    \\
$E_2^{(\text{soft})}(k;J)$           &  (\ref{eq:E2Edge})                                                                          & {\tt E(2,k,J,'soft')}                \\
$E_2^{(\text{soft})}((k,0);J_1,J_2)$ &  (\ref{eq:spacinggeneral})                                                                  & {\tt E(2,[k,0],\{J1,J2\},'soft')}    \\
$F(x,y)$                             &  (\ref{eq:twojoint})                                                                        & {\tt F2Joint(x,y)}                   \\*[0.5mm]\hline
$E_4^{(\text{soft})}(k;J)$           &  (\ref{eq:GSEMatrixKernelDet})                                                              & {\tt E(4,k,J,'soft','MatrixKernel')} \\*[0.5mm]\hline
$E_{\text{LUE}}^{(n)}(k;J,\alpha)$   &  (\ref{eq:detLUEn})                                                                         & {\tt E('LUE',k,J,n,alpha)}           \\
$E_{\text{LUE}}^{(n)}((k,0);J_1,J_2,\alpha)$           &  (\ref{eq:spacinggeneral})                                                & {\tt E('LUE',[k,0],\{J1,J2\},n,alpha)}\\
$E_2^{(\text{hard})}(k;J,\alpha)$    &  (\ref{eq:E2hard})                                                                          & {\tt E(2,k,J,'hard',alpha)}          \\
$E_2^{(\text{hard})}((k,0);J_1,J_2,\alpha)$    &  (\ref{eq:spacinggeneral})                                                        & {\tt E(2,[k,0],\{J1,J2\},'hard',alpha)}\\*[0.5mm]\hline
$E_+(k;s)$                           &  (\ref{eq:Epm})                                                                             & {\tt E('+',k,s)}                     \\
$E_-(k;s)$                           &  (\ref{eq:Epm})                                                                             & {\tt E('-',k,s)}                     \\
$E_\beta(k;s)$                       &  (\ref{eq:Ebeta}), (\ref{eq:E2sum}), (\ref{eq:E1}),  (\ref{eq:E4})                          & {\tt E(beta,k,s)}                    \\*[0.5mm]\hline
$\tilde E_+(k;s)$                    &  (\ref{eq:EpmTilde})                                                                        & {\tt E('+',k,s,'soft')}              \\
$\tilde E_-(k;s)$                    &  (\ref{eq:EpmTilde})                                                                        & {\tt E('-',k,s,'soft')}              \\
$\tilde E_\beta(k;s)$                &  (\ref{eq:EbetaTilde}), (\ref{eq:E2EdgeSum}), (\ref{eq:E4Edge}), (\ref{eq:E1recursion})     & {\tt E(beta,k,s,'soft')}             \\
$F_\beta(k;s)$                       &  (\ref{eq:Fbeta}), (\ref{eq:Fbeta2})                                                        & {\tt F(beta,k,s)}                    \\
$F_\beta(s)$                         &  (\ref{eq:TWbeta})                                                                          & {\tt F(beta,s)}                      \\*[0.5mm]\hline
$E_{+,\alpha}(k;s)$                  &  (\ref{eq:Epmalpha})                                                                        & {\tt E('+',k,s,'hard',alpha)}        \\
$E_{-,\alpha}(k;s)$                  &  (\ref{eq:Epmalpha})                                                                        & {\tt E('-',k,s,'hard',alpha)}        \\
$E_{\beta,\alpha}(k;s)$              &  (\ref{eq:ELE}), (\ref{eq:E2HardSum}), (\ref{eq:E4Hard}), (\ref{eq:E1Hardrecursion})        & {\tt E(beta,k,s,'hard',alpha)}       \\
$F_{\beta,\alpha}(k;s)$              &  (\ref{eq:FLE})                                                                             & {\tt F(beta,k,s,alpha)}              \\*[0.5mm]\hline
\end{tabular}}}
\label{tab:functions}
\end{table}

\subsubsection{Example~\protect\ref{sect:softex2} revisited}
Let us evaluate, for the last time in this paper, the  values~$F_1(0)$ and~$F_2(0)$ of the Tracy--Widom distributions for GOE and GUE, now using those high level commands.

\smallskip
{\small
\begin{verbatim}
>> pref('printmode','trunc');
>> [val,err] = F(1,0);
>> PrintCorrectDigits(val,err);

   0.83190806620295_

>> [val,err] = F(2,0);
>> PrintCorrectDigits(val,err);

   0.96937282835526_
\end{verbatim}}

\subsubsection{Example: checking the k-level spacing functions against a constraint}\label{sect:constraint}

An appropriate way of checking the quality of the automatic error control goes by evaluating certain constraints such as the
mass and mean given in (\ref{eq:spacingconstraint}):

\smallskip
{\small
\begin{verbatim}
>> pref('printmode','round');
>> s = 2.13; beta = 1;
>> mass = 0; errmass = 0; mean = 0; errmean = 0;
>> M = 10; for k=0:M
>>    [val,err] = E(beta,k,s);
>>    mass = mass+val; errmass = errmass+err;
>>    mean = mean+k*val; errmean = errmean+k*err;
>> end
>> PrintCorrectDigits(mass,errmass);

   1.0000000000000_

>> PrintCorrectDigits(mean,errmean);

   2.130000000000__
\end{verbatim}}
\smallskip

\noindent
The results of (\ref{eq:spacingconstraint}) are perfectly matched. The reader is invited to repeat this experiment
with a larger truncation index for the series.

\subsubsection{Example: more general constraints}\label{sect:generalconstraint}

The preceding example can be extended to more general probabilities $E(k;J)$ that are given by a determinantal expression of the form (\ref{eq:detEgen}), that is,
\begin{equation}
E(k;J) = \left.\frac{(-1)^k}{k!} \frac{d^k}{dz^k} \det\left(I - z K\projected{L^2(J)}\right)\right|_{z=1},
\end{equation} for some trace class operator $K\projected{L^2(J)}$. Expanding the entire function
\begin{equation}
d(z) = \det\left(I - z K\projected{L^2(J)}\right)
\end{equation}
into a power series at $z=1$ yields
\begin{subequations}\label{eq:gencon}
\begin{align}
\sum_{k=0}^\infty E(k;J) &= \sum_{k=0}^\infty \frac{(-1)^k}{k!} d^{(k)}(1) = d(0) = 1,\\*[2mm]
\sum_{k=0}^\infty k\, E(k;J) &= -\sum_{k=0}^\infty \frac{(-1)^k}{k!} d^{(k+1)}(1) = -d'(0) = \tr \left(K\projected{L^2(J)}\right);
\end{align}
\end{subequations}
both of which have a probabilistic interpretation \citeaffixed[p.~119]{MR1677884}{see}. Now, for the Airy kernel we get
\begin{multline}
\tr \left(K_{\Ai}\projected{L^2(s,\infty)}\right) = \int_s^\infty K_{\Ai}(x,x)\, dx \\*[1mm]
= \tfrac13 (2s^2 \Ai(s)^2-2s\Ai'(s)^2-\Ai(s)\Ai'(s)),\qquad
\end{multline}
with the specific value (for $s=0$)
\begin{equation}
\tr \left(K_{\Ai}\projected{L^2(0,\infty)}\right) = \frac{1}{9 \Gamma(\tfrac13)\Gamma(\tfrac23)}.
\end{equation}
Now, let us check the quality of the numerical evaluation of (\ref{eq:gencon}) (and the automatic error control) using this value.

\smallskip
{\small
\begin{verbatim}
>> pref('printmode','round');
>> s = 0; beta = 2;
>> mass = 0; errmass = 0; mean = 0; errmean = 0;
>> M = 3; for k=0:M
>>     [val,err] = E(beta,k,s,'soft');
>>     mass = mass+val; errmass = errmass+err;
>>     mean = mean+k*val; errmean = errmean+k*err;
>> end
>> PrintCorrectDigits(mass,errmass);

   1.000000000000__

>> PrintCorrectDigits(mean,errmean);

   0.030629383079___

>> 1/9/gamma(1/3)/gamma(2/3)

   0.0306293830789884

\end{verbatim}}

\noindent
The results are in perfect match with (\ref{eq:gencon}): they are correctly \emph{rounded}, indeed. The reader is invited to play with
the truncation index of the series.

\subsubsection{Example: calculating quantiles}

Quantiles are easy to compute; here come the $5\%$ and $95\%$ quantiles of the Tracy--Widom distribution for GOE:

\smallskip
{\small
\begin{verbatim}
>> F1inv = vec(@(p) (fzero(@(s) F(1,s)-p,0)));
>> F1inv([0.05 0.95])

   -3.18037997693773         0.979316053469556
\end{verbatim}}

\subsection{Densities and Moments}

Probability densities and moments are computed by barycentric interpolation in $m$ Chebyshev points as described in Section~\ref{sect:dens}.
This is most conveniently done by installing the functionality of the {\tt chebfun} package \cite{chebfun}. Our basic command  is then,
for a given cumulative distribution function $F(s)$:

\smallskip
{\small
\begin{verbatim}
>> [val,err,supp,PDF,CDF] = moments(@(s) F(s),m);
\end{verbatim}}
\smallskip

\noindent
If one skips the argument $m$, the number of points will be chosen automatically. The results are: {\tt val} gives the first four moments, that is, mean, variance, skewness,
and kurtosis of the distribution; {\tt err} gives the absolute errors of each of those moments; {\tt supp} gives the numerical support of the density; {\tt PDF} gives
the interpolant (\ref{eq:bary}) of the probability density function $F'(s)$ in form of a {\tt chebfun} object; {\tt CDF} gives the same for the function $F(s)$ itself.
If one sets

\smallskip
{\small
\begin{verbatim}
>> pref('plot',true);
\end{verbatim}}
\smallskip

\noindent
a call of the command {\tt moments} will plot the PDF and the CDF in passing.

\subsubsection{Example: the first four moments of the Tracy--Widom distributions for GOE, GUE, and GSE}\label{sect:exTW}

The following fills in the missing high-precision digits for the GSE in \citeasnoun[Table~1]{Tracy08}; Figure~\ref{fig:TW} is automatically generated in passing:

\smallskip
{\small
\begin{verbatim}
>> pref('printmode','trunc'); pref('plot',true);
>> for beta = [1 2 4]
>>    [val,err] = moments(@(s) F(beta,s), 128);
>>    PrintCorrectDigits(val,err);
>> end

 -1.2065335745820_  1.607781034581__   0.29346452408____  0.1652429384_____
 -1.771086807411__  0.8131947928329__  0.224084203610___  0.0934480876_____
 -2.306884893241__  0.5177237207726__  0.16550949435____  0.0491951565_____
\end{verbatim}}
\smallskip

\begin{figure}[tbp]
\begin{center}
\begin{minipage}{0.49\textwidth}
\vspace*{0.45cm}
\begin{center}
\includegraphics[width=\textwidth]{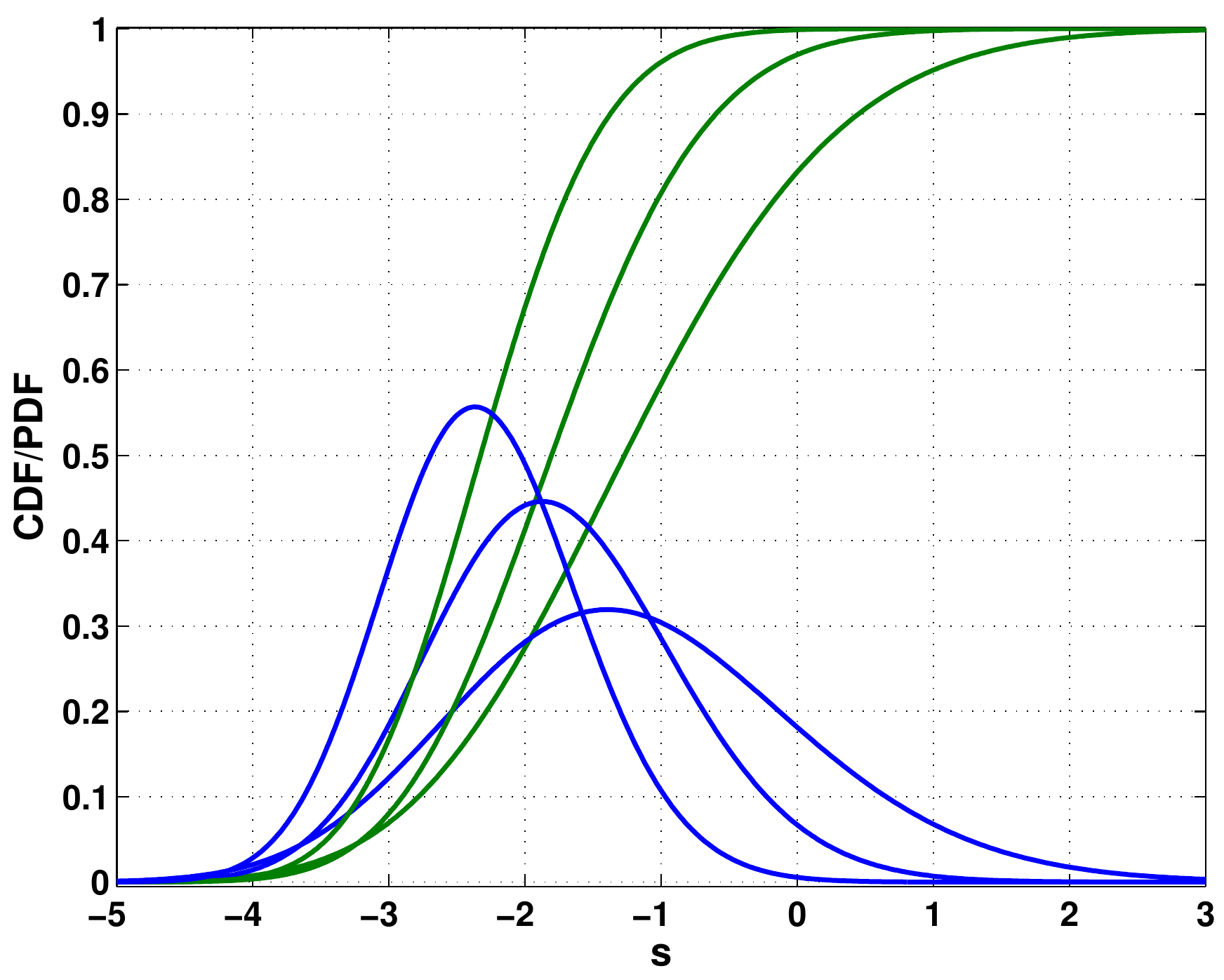}
\end{center}
\end{minipage}
\end{center}\vspace*{-0.0625cm}
\caption{Probability density functions $F'_\beta(s)$ (PDF, blue) and cumulative distribution functions $F_\beta(s)$ (CDF, green) of the Tracy--Widom distributions (\ref{eq:TWbeta})
for GOE ($\beta=1$), GUE ($\beta=2$), and GSE ($\beta=4$); larger $\beta$ go to the left. This plot was automatically generated by the commands in Example~\protect\ref{sect:exTW}.
Compare with \protect\citeasnoun[Fig.~1]{MR1385083} and \protect\citeasnoun[Fig.~1.1]{Dieng05}.}
\label{fig:TW}
\end{figure}

\subsubsection{Example: checking the k-level spacing functions against a constraint}\label{sect:constraint2}

In the final example of this paper we check our automatic error control in dealing with densities.
We take the level spacing densities defined in (\ref{eq:pbeta}) and evaluate the integral
constraints (\ref{eq:integralconstraint}):

\smallskip
{\small
\begin{verbatim}
>> pref('printmode','round'); beta = 1; M = 10; [dom,s] = domain(0,M);
>> E0 = chebfun(vec(@(s) E(beta,0,s)),dom);
>> E1 = chebfun(vec(@(s) E(beta,1,s)),dom);
>> E2 = chebfun(vec(@(s) E(beta,2,s)),dom);
>> p2 = diff(3*E0+2*E1+E2,2);
>> mass = sum(p2); errmass = cheberr(p2);
>> mean = sum(s.*p2); errmean = errmass*sum(s);
>> PrintCorrectDigits(mass,errmass);

   1.0000000000____

>> PrintCorrectDigits(mean,errmean);

   3.000000000_____
\end{verbatim}}
\smallskip

\noindent
Once more, the results of (\ref{eq:integralconstraint}) are perfectly matched. The reader is invited to repeat this experiment
with a larger truncation point for the integrals.

\renewcommand{\thesection}{A}
\section{Appendices}

\subsection{Algebraic Kernel Singularities and Gauss--Jacobi Quadrature}\label{sect:singularity}

If the kernel is not analytic in a complex neighborhood of the interval defining the integral operator, the straightforward
use of the method of Section~\ref{sect:basic}, based on Clenshaw--Curtis quadrature, does not generally yield exponential convergence.
Since the software limits the number of quadrature points (default is a maximum of 256) this will result in rather inaccurate
results (inaccuracies which are detected, however, by the automatic error control). This remark applies in particular to the Laguerre
kernel $K_{n,\alpha}(x,y)$ and the Bessel kernel $K_\alpha(x,y)$ which, for non-integer parameter~$\alpha$, exhibit algebraic singularities at $x=0$ or $y=0$, namely
\begin{subequations}\label{eq:besselsing}
\begin{align}
K_{n,\alpha}(x,y) \propto x^{\alpha/2} \quad(x\to 0),& \qquad K_\alpha(x,y) \propto x^{\alpha/2} \quad(x\to 0),\\*[1mm]
K_{n,\alpha}(x,y) \propto y^{\alpha/2} \quad(y\to 0),& \qquad K_\alpha(x,y) \propto y^{\alpha/2} \quad(y\to 0).
\end{align}
\end{subequations}

\subsubsection{Example: level spacing at the hard edge for non-integer parameter}\label{exam:bessel1}

Hence, for example, if we want to evaluate the values
\[
E_2^{(\text{hard})}(1;(0,6),\pm\tfrac{1}{2}) = - \left.\frac{d}{dz} \det\left(I-z K_{\pm \frac12}\projected{L^2(0,6)}\right)\right|_{z=1},
\]
the method of Section~\ref{sect:basic} gives, using Clenshaw--Curtis quadrature, just 1 digit for the stronger singularity $\alpha=-\tfrac12$ and
6 digits for the weaker singularity $\alpha=\tfrac12$:

\smallskip
{\small
\begin{verbatim}
>> pref('quadrature',@ClenshawCurtis); pref('printmode','trunc');
>> s = 6;
>> alpha = -0.5; K.ker = @(x,y) BesselKernel(alpha,x,y); K.mode = 'outer';
>> [val1,err1] = dzdet(op(K,[0,s]),1);
>> alpha =  0.5; K.ker = @(x,y) BesselKernel(alpha,x,y); K.mode = 'outer';
>> [val2,err2] = dzdet(op(K,[0,s]),1);
>> PrintCorrectDigits([val1 val2],[err1 err2]);

    0.8______________    0.524976_________
\end{verbatim}}
\smallskip

\subsubsection{Gauss--Jacobi quadrature}

This loss of accuracy can be circumvented if we switch from Clenshaw--Curtis quadrature to Gauss--Jacobi quadrature \cite[§15.3]{MR0372517}, which
exactly addresses this type of algebraic singularities at the boundary points of the interval. Specifically, this quadrature rule addresses
functions $f$ on the interval $(a,b)$
such that, for $\alpha,\beta>-1$, the function $\tilde f$ that is obtained from removing the singularities
\begin{equation}
f(x) \propto (x-a)^\alpha\quad (x\to a),\qquad f(x) \propto (b-x)^\beta \quad (x\to b),
\end{equation}
by
\begin{equation}
\tilde f (x) = \frac{f(x)}{(x-a)^\alpha(b-x)^\beta}
\end{equation}
is analytic in a complex neighborhood of $(a,b)$. Then, the $m$-point Gauss--Jacobi quadrature
provides nodes $x_j \in (a,b)$ and
\emph{positive} weights $w_j$ such that the approximation
\begin{equation}
\sum_{j=1}^m w_j f(x_j) \approx \int_a^b f(x)\,dx
\end{equation}
is exponentially convergent: there is a $\rho>1$ such that the error behaves like $O(\rho^{-m})$ for $m\to \infty$. In our toolbox the Gauss--Jacobi quadrature is called
by:

\smallskip
{\small\begin{verbatim}
>> [w,x] = GaussJacobi(a,b,alpha,beta,m)
\end{verbatim}}
\smallskip

\subsubsection{Example~\protect\ref{exam:bessel1} revisited}

The application of the Gauss--Jacobi quadrature to the
Bessel kernel determinant (\ref{eq:E2hard}) requires the determination of the singular exponents that are relevant for the Fredholm determinant. Now, from
(\ref{eq:besselsing}) we infer that the eigenfunctions
\begin{equation}
\int_0^s K_\alpha(x,y) u_\lambda(y)\,dy = \lambda u_\lambda(x)
\end{equation}
of the Bessel kernel exhibit the same type of algebraic singularity at $x\to 0$:
\begin{equation}
u_\lambda(x) \propto x^{\alpha/2}\quad (x\to 0).
\end{equation}
It turns out that the singularity of $K_\alpha(x,y) u_\lambda(y)$ at $y\to 0$ governs the behavior of the determinant approximation, that is, we should take the singular exponent of
\begin{equation}
K_\alpha(x,y) u_\lambda(y) \propto y^\alpha \quad (y\to 0).
\end{equation}
Indeed, with this choice of the singular exponent we get 14 digits for both cases:

\smallskip
{\small
\begin{verbatim}
>> alpha = -0.5; pref('quadrature',@(a,b,m) GaussJacobi(a,b,alpha,0,m));
>> K.ker = @(x,y) BesselKernel(alpha,x,y); K.mode = 'outer';
>> [val1,err1] = dzdet(op(K,[0,s]),1);
>> alpha =  0.5; pref('quadrature',@(a,b,m) GaussJacobi(a,b,alpha,0,m));
>> K.ker = @(x,y) BesselKernel(alpha,x,y); K.mode = 'outer';
>> [val2,err2] = dzdet(op(K,[0,s]),1);
>> PrintCorrectDigits([val1 val2],[err1 err2]);

    0.86114217058328_    0.52497677921859_
\end{verbatim}}
\smallskip

All this is most conveniently hidden from the user who can simply call the high level command that takes care of choosing the appropriate
quadrature rule:

\smallskip
{\small
\begin{verbatim}
>> alpha = -0.5; [val1,err1] = E(2,1,[0,s],'hard',alpha);
>> alpha =  0.5; [val2,err2] = E(2,1,[0,s],'hard',alpha);
>> PrintCorrectDigits([val1 val2],[err1 err2]);

    0.86114217058328_    0.52497677921859_
\end{verbatim}}

\subsubsection*{Remark}
We have studied this particular example because it provides a cross check by the relation
\begin{equation}
E_2^{(\text{hard})}(k;(0,s),\pm\tfrac{1}{2}) = E_\mp(k;2\sqrt{s}/\pi).
\end{equation}
This relation can be obtained from a simple integral transformation that allows us to rewrite the determinant of the Bessel kernel for $\alpha=\pm\tfrac12$ as that
of the odd or even sine kernel. This way (switching back to Clenshaw--Curtis quadrature which is appropriate here) we get values
that are in perfect agreement with the ones obtained for the Bessel kernel with Gauss--Jacobi quadrature:
\smallskip
{\small
\begin{verbatim}
>> [val1,err1] = E('+',1,2*sqrt(s)/pi);
>> [val2,err2] = E('-',1,2*sqrt(s)/pi);
>> PrintCorrectDigits([val1 val2],[err1 err2])

    0.861142170583288    0.524976779218593
\end{verbatim}}

\subsection{Tables of Some Statistical Properties}\label{sect:tables}

We provide some tables of statistical properties of the $k$-level spacing densities $p_\beta(k;s)$, defined in (\ref{eq:pbeta}), and
the distributions $F_\beta(k;s)$ of the $k$-th largest level at the soft edge, defined in (\ref{eq:Fbeta}). The tables display correctly \emph{truncated} digits that have passed the automatic error control
of the software in Section~\ref{sect:software}. Computing times are in seconds.

Note that because of the interrelations between GOE and GSE there is no need to separately tabulate the values for $\beta=4$: first, the interlacing property (\ref{eq:interlacing}) gives $F_4(k;s)=F_1(2k;s)$.
Second, we infer from (\ref{eq:E1}) and (\ref{eq:E4}) that
\begin{equation}
E_4(k;s) = \tfrac{1}{2} E_1(2k-1;2s) + E_1(2k;2s)+\tfrac{1}{2} E_1(2k+1;2s)
\end{equation}
which implies
\begin{equation}
p_4(k;s) = 2p_1(2k+1;2s).
\end{equation}

\begin{table}[tbp]
\caption{Statistical properties of $p_1(k;s)$ for various $k$. Note that because of $p_4(k;s) = 2 p_1(2k+1,2s)$ one can read off
the values for $p_4(k;s)$ from those for $p_1(2k+1;s)$: Just divide the mean by two and the variance by four, skewness and kurtosis remain unchanged.}
\vspace*{-3mm}
\centerline{%
\setlength{\extrarowheight}{3pt}
{\begin{tabular}{crrrrr}\hline
PDF & mean&  variance\phantom{xx} & skewness\phantom{xx} & kurtosis\phantom{xx} & time\\*[0.5mm]\hline
$p_1(0;s)$ & $ 1\;\;\;$ & $0.28553\,06557$ & $0.68718\,99889$ & $  0.37123\,80638$ & $   0.67$\\
$p_1(1;s)$ & $ 2\;\;\;$ & $0.41639\,36889$ & $0.34939\,68438$ & $  0.02858\,27332$ & $   1.22$\\
$p_1(2;s)$ & $ 3\;\;\;$ & $0.49745\,52604$ & $0.22741\,44134$ & $ -0.01329\,56588$ & $   2.59$\\
$p_1(3;s)$ & $ 4\;\;\;$ & $0.55564\,24180$ & $0.16645\,68639$ & $ -0.01994\,68028$ & $   4.24$\\
$p_1(4;s)$ & $ 5\;\;\;$ & $0.60091\,83521$ & $0.13042\,07251$ & $ -0.02007\,29233$ & $   5.81$\\
$p_1(5;s)$ & $ 6\;\;\;$ & $0.63794\,46245$ & $0.10679\,47124$ & $ -0.01884\,07449$ & $   7.43$\\
$p_1(6;s)$ & $ 7\;\;\;$ & $0.66925\,53948$ & $0.09018\,32871$ & $ -0.01743\,19487$ & $   9.46$\\
$p_1(7;s)$ & $ 8\;\;\;$ & $0.69637\,60657$ & $0.07790\,15490$ & $ -0.01613\,54800$ & $  11.04$\\
$p_1(8;s)$ & $ 9\;\;\;$ & $0.72029\,45046$ & $0.06847\,07897$ & $ -0.01500\,75200$ & $  12.80$\\
$p_1(9;s)$ & $10\;\;\;$ & $0.74168\,65573$ & $0.06101\,25387$ & $ -0.01404\,07984$ & $  15.26$\\*[0.5mm]\hline
\end{tabular}}}
\label{tab:3}
\end{table}

\begin{table}[tbp]
\caption{Statistical properties of $p_2$ for various $k$.}
\vspace*{-3mm}
\centerline{%
\setlength{\extrarowheight}{3pt}
{\begin{tabular}{crrrrr}\hline
PDF & mean&  variance\phantom{xx} & skewness\phantom{xx} & kurtosis\phantom{xx} & time\\*[0.5mm]\hline
$p_2(0;s)$ & $ 1\;\;\;$ & $0.17999\,38776$ & $ 0.49706\,36204$ & $   0.12669\,98480$ & $  0.63$\\
$p_2(1;s)$ & $ 2\;\;\;$ & $0.24897\,77536$ & $ 0.24167\,43158$ & $  -0.01494\,23984$ & $  1.36$\\
$p_2(2;s)$ & $ 3\;\;\;$ & $0.29016\,98290$ & $ 0.15542\,00591$ & $  -0.02317\,40428$ & $  1.44$\\
$p_2(3;s)$ & $ 4\;\;\;$ & $0.31944\,35563$ & $ 0.11334\,61773$ & $  -0.02150\,23114$ & $  1.67$\\
$p_2(4;s)$ & $ 5\;\;\;$ & $0.34214\,08054$ & $ 0.08871\,43069$ & $  -0.01914\,18388$ & $  2.09$\\
$p_2(5;s)$ & $ 6\;\;\;$ & $0.36067\,45961$ & $ 0.07263\,43907$ & $  -0.01714\,28515$ & $  2.06$\\
$p_2(6;s)$ & $ 7\;\;\;$ & $0.37633\,63928$ & $ 0.06135\,08835$ & $  -0.01555\,25979$ & $  2.19$\\
$p_2(7;s)$ & $ 8\;\;\;$ & $0.38989\,74631$ & $ 0.05301\,56552$ & $  -0.01428\,79010$ & $  2.41$\\
$p_2(8;s)$ & $ 9\;\;\;$ & $0.40185\,51105$ & $ 0.04661\,73337$ & $  -0.01326\,81121$ & $  2.56$\\
$p_2(9;s)$ & $10\;\;\;$ & $0.41254\,86854$ & $ 0.04155\,73856$ & $  -0.01243\,20513$ & $  2.74$\\*[0.5mm]\hline
\end{tabular}}}
\label{tab:4}
\end{table}

\begin{table}[tbp]
\caption{Statistical properties of $F_1(k;s)$ for various $k$. Note that because of $F_4(k;s) = F_1(2k,s)$ one can directly read off
the values for $F_4(k;s)$ from those for $F_1(2k;s)$.}
\vspace*{-3mm}
\centerline{%
\setlength{\extrarowheight}{3pt}
{\begin{tabular}{crrrrr}\hline
CDF & mean&  variance\phantom{xx} & skewness\phantom{xx} & kurtosis\phantom{xx} & time\\*[0.5mm]\hline
$F_1(1;s)$ & $ -1.20653\,35745$ & $ 1.60778\,10345$ & $ 0.29346\,45240$ & $  0.16524\,29384$ & $   4.59$\\
$F_1(2;s)$ & $ -3.26242\,79028$ & $ 1.03544\,74415$ & $ 0.16550\,94943$ & $  0.04919\,51565$ & $  12.45$\\
$F_1(3;s)$ & $ -4.82163\,02757$ & $ 0.82239\,01151$ & $ 0.11762\,14761$ & $  0.01977\,46604$ & $  30.04$\\
$F_1(4;s)$ & $ -6.16203\,99636$ & $ 0.70315\,81054$ & $ 0.09232\,83954$ & $  0.00816\,06305$ & $  51.24$\\
$F_1(5;s)$ & $ -7.37011\,47042$ & $ 0.62425\,23679$ & $ 0.07653\,98210$ & $  0.00245\,40580$ & $  77.49$\\
$F_1(6;s)$ & $ -8.48621\,83723$ & $ 0.56700\,71487$ & $ 0.06567\,07705$ & $ -0.00073\,42515$ & $ 112.00$\\*[0.5mm]\hline
\end{tabular}}}
\label{tab:5}
\end{table}

\begin{table}[tbp]
\caption{Statistical properties of $F_2(k;s)$ for various $k$.}
\vspace*{-3mm}
\centerline{%
\setlength{\extrarowheight}{3pt}
{\begin{tabular}{crrrrr}\hline
CDF & mean&  variance\phantom{xx} & skewness\phantom{xx} & kurtosis\phantom{xx} & time\\*[0.5mm]\hline
$F_2(1;s)$ & $ -1.77108\,68074$ & $  0.81319\,47928$ & $  0.22408\,42036$ & $   0.09344\,80876$  & $        1.84$\\
$F_2(2;s)$ & $ -3.67543\,72971$ & $  0.54054\,50473$ & $  0.12502\,70941$ & $   0.02173\,96385$  & $        5.44$\\
$F_2(3;s)$ & $ -5.17132\,31745$ & $  0.43348\,13326$ & $  0.08880\,80227$ & $   0.00509\,66000$  & $       10.41$\\
$F_2(4;s)$ & $ -6.47453\,77733$ & $  0.37213\,08147$ & $  0.06970\,92726$ & $  -0.00114\,15160$  & $       17.89$\\
$F_2(5;s)$ & $ -7.65724\,22912$ & $  0.33101\,06544$ & $  0.05777\,55438$ & $  -0.00405\,83706$  & $       25.56$\\
$F_2(6;s)$ & $ -8.75452\,24419$ & $  0.30094\,94654$ & $  0.04955\,14791$ & $  -0.00559\,98554$  & $       34.72$\\*[0.5mm]\hline
\end{tabular}}}
\label{tab:6}
\end{table}

\section*{Acknowledgements}

The author expresses his gratitude to Craig Tracy and an anonymous referee for pointing out some early references;
and to Peter Forrester for various remarks on a preliminary version of this paper, for comments relating to the work of \citeasnoun{MR2229797}
and its proof of (\ref{eq:conjecture2}), and for his suggestion to include the tables of Section~\ref{sect:tables}.

\bibliographystyle{kluwer}
\bibliography{article}

\end{document}